\documentclass{amsart}%
\usepackage{amssymb}
\usepackage{amsmath}
\usepackage{amsfonts}
\usepackage[all,cmtip]{xy}
\usepackage[margin=1.26in]{geometry}
\usepackage{hyperref}%
\usepackage{graphicx}
\newtheorem{theorem}{Theorem}
\newtheorem*{theoremannounce}{Theorem}
\numberwithin{theorem}{section}
\theoremstyle{plain}
\newtheorem*{acknowledgement}{Acknowledgement}

\newtheorem{conjecture}{Conjecture}
\newtheorem{corollary}[theorem]{Corollary}

\newtheorem{definition}[theorem]{Definition}

\newtheorem{lemma}[theorem]{Lemma}

\newtheorem{question}{Question}
\newtheorem{proposition}[theorem]{Proposition}
\theoremstyle{remark}
\newtheorem{example}[theorem]{Example}
\newtheorem{elaboration}{Elaboration}

\newtheorem{remark}[theorem]{Remark}

\numberwithin{equation}{section}

\begin{document}
\title[The relative $K$-group in the ETNC]{On the relative $K$-group in the ETNC}
\author{Oliver Braunling}
\address{Freiburg Institute for Advanced Studies (FRIAS), University of Freiburg,
D-79104 Freiburg im\ Breisgau, Germany}
\thanks{The author was supported by DFG GK1821 \textquotedblleft Cohomological Methods
in Geometry\textquotedblright\ and a Junior Fellowship at the Freiburg
Institute for Advanced Studies (FRIAS)}
\date{%
{\today}%
}

\begin{abstract}
We consider the Burns--Flach formulation of the equivariant Tamagawa number
conjecture (ETNC). In their setup, a Tamagawa number is an element of a
relative $K$-group. We show that this relative group agrees with an ordinary
$K$-group, namely of the category of locally compact topological modules over
the order. Its virtual objects are an equivariant Haar measure in a precise
sense. We expect that all relative $K$-groups in the ETNC will have analoguous
interpretations. At present, we need to restrict to regular orders, e.g. hereditary.

\end{abstract}
\maketitle

\section{Introduction}

\subsection{What is a Tamagawa number?\label{subsect_WhatIsTamagawaNum}}

We assume basic familiarity with the Burns--Flach formulation of the ETNC
\cite{MR1884523}, at least with the introduction. We follow the same notation.

A Tamagawa number $-$ most classically $-$ is a ratio of volumes and as such a
positive real number. This concept has undergone wide generalizations, and in
the context of the Burns--Flach formulation, it is a value in the relative
$K$-group $K_{0}(\mathfrak{A},\mathbb{R})$. In the classical case,
$\mathfrak{A}=\mathbb{Z}$, one has a canonical isomorphism%
\[
K_{0}(\mathbb{Z},\mathbb{R})\cong\mathbb{R}_{>0}^{\times}\text{,}%
\]
and by this identification the interpretation as a volume is still visible. A
relative $K$-group, like the group $K_{0}(\mathfrak{A},\mathbb{R})$ here, is
defined as the relevant term to complete a long exact sequence; in the case at
hand this sequence is%
\begin{equation}
\cdots\longrightarrow K_{1}(\mathfrak{A})\longrightarrow K_{1}(A_{\mathbb{R}%
})\longrightarrow K_{0}(\mathfrak{A},\mathbb{R})\longrightarrow K_{0}%
(\mathfrak{A})\longrightarrow\cdots\label{l_Intro1}%
\end{equation}
Besides fitting in the right spot, no further interpretation comes with
relative $K$-groups. In this paper, we propose a different viewpoint, where
the volume interpretation of $K_{0}(\mathfrak{A},\mathbb{R})$ is strengthened.
We introduce the category of locally compact topologized right $\mathfrak{A}%
$-modules, call it $\mathsf{LCA}_{\mathfrak{A}}$. It turns out to be an exact
category, so it comes with its own $K$-theory groups, its own determinant
functors, etc. We prove:

\begin{theoremannounce}
Suppose $\mathfrak{A}$ is a regular order in a finite-dimensional semisimple
$\mathbb{Q}$-algebra $A$. Then there is a canonical long exact sequence%
\[
\cdots\longrightarrow K_{n}(\mathfrak{A})\longrightarrow K_{n}(A_{\mathbb{R}%
})\longrightarrow K_{n}(\mathsf{LCA}_{\mathfrak{A}})\longrightarrow
K_{n-1}(\mathfrak{A})\longrightarrow\cdots\text{,}%
\]
and for all $n$, there are canonical isomorphisms%
\[
K_{n}(\mathsf{LCA}_{\mathfrak{A}})\cong K_{n-1}(\mathfrak{A},\mathbb{R}%
)\text{.}%
\]

\end{theoremannounce}

This will be Theorem \ref{thm_main_asannounced}. In particular, for $n=1$ we
obtain the new interpretation $K_{1}(\mathsf{LCA}_{\mathfrak{A}})\cong
K_{0}(\mathfrak{A},\mathbb{R})$. So, when rephrasing the ETNC using this
sequence instead, one can get rid of relative $K$-groups and every term is
interpreted as an `absolute' $K$-group of a suitable category.

The idea for the above theorem can be explained as follows: By Deligne's work,
every exact category has a universal determinant functor to its virtual
objects. We shall deduce from computations of Clausen \cite{clausen} that the
universal determinant functor of the category of locally compact abelian (LCA)
groups is precisely the Haar measure, see Theorem \ref{thm_UnivOfHaarTorsor}.
Since classical Tamagawa numbers stem from the Haar measure, one may now
speculate that the universal determinant functor of the \textit{equivariant}
counterpart, i.e. LCA groups with an action of $\mathfrak{A}$ (which is
exactly our category $\mathsf{LCA}_{\mathfrak{A}}$!), should be the right
receptacle for equivariant Tamagawa measures. However, since so far the
literature has used the relative $K$-group $K_{0}(\mathfrak{A},\mathbb{R})$
instead, there is no reasonable way to escape the hope that the groups must
actually agree. Our main theorem above shows that this speculation is spot on
for regular orders.

We should clarify that our methods are by no means a formal construction
transforming any sort of relative $K$-group into an absolute one. Indeed, our
methods work very specifically for the relative $K$-group \textit{with respect
to the reals}, by exploiting that $\mathbb{R}$ is a locally compact
topological field. Although we do not touch on this topic in this paper,
variations for $\mathbb{Q}_{p}$ or more generally other locally compact
topological fields should exist.

Our main theorem also yields an alternative description of $K_{0}%
(\mathfrak{A},\mathbb{R})$ by using Nenashev's generator and relator
presentation for $K_{1}$-groups.

\begin{theoremannounce}
Let $\mathfrak{A}$ be a regular order in a finite-dimensional semisimple
$\mathbb{Q}$-algebra. Then $K_{0}(\mathfrak{A},\mathbb{R})$ admits a
presentation where generators are double short exact sequences%
\[%
\xymatrix{
A \ar@<1ex>@{^{(}->}[r]^{p} \ar@<-1ex>@{^{(}->}[r]_{q} & B \ar@<1ex>@{->>}%
[r]^{r} \ar@<-1ex>@{->>}[r]_{s} & C,
}%
\]
for $A,B,C$ locally compact right $\mathfrak{A}$-modules, modulo suitable
relations coming from $(3\times3)$-diagrams (We explain this in detail in
\S \ref{subsect_Nenashev}).
\end{theoremannounce}

See Theorem \ref{thm_NenashevStylePresentation}. Unfortunately, the
identification of the image of $K_{0}(\mathfrak{A},\mathbb{Q})$ inside
$K_{1}(\mathsf{LCA}_{\mathfrak{A}})$ remains open at the moment. This would be
important for the formulation of Rationality in the ETNC. Moreover, the
restriction to regular orders is a nuisance and we hope to remove this
assumption in the future. Maximal orders, and more generally hereditary orders
are regular. Group rings $\mathbb{Z}[G]$ are only Gorenstein, we cannot handle
them at the moment. It is clear however what the issue is, since we can obtain
the analogue of our main theorem for $G$-theory for all orders, with no extra conditions.

\begin{theoremannounce}
Suppose $\mathfrak{A}$ is an arbitrary order in a finite-dimensional
semisimple $\mathbb{Q}$-algebra $A$. Then there is a long exact sequence%
\[
\cdots\longrightarrow G_{n}(\mathfrak{A})\longrightarrow K_{n}(A_{\mathbb{R}%
})\longrightarrow K_{n}(\mathsf{LCA}_{\mathfrak{A}})\longrightarrow
G_{n-1}(\mathfrak{A})\longrightarrow\cdots\text{,}%
\]
where $G_{n}(\mathfrak{A}):=K_{n}(\mathsf{Mod}_{\mathfrak{A},fg})$ denotes the
$K$-theory of the category of finitely generated right $\mathfrak{A}$-modules
(this is often called \textquotedblleft$G$-theory\textquotedblright).
\end{theoremannounce}

See Theorem \ref{thm_Gthy_Version}. There is always a morphism $K_{n}%
(\mathfrak{A})\rightarrow G_{n}(\mathfrak{A})$, but it need not be an
isomorphism in general. This result gives a hint how one may remove the
regularity assumption. To complete the overview, let us state the following
result, Theorem \ref{thm_UnivOfHaarTorsor}, which is essentially due to
Clausen, even though he has not spelled it out in this fashion in
\cite{clausen}.

\begin{theoremannounce}
The Haar functor $Ha:\mathsf{LCA}_{\mathbb{Z}}^{\times}\rightarrow
\mathsf{Tors}(\mathbb{R}_{>0}^{\times})$ is the universal determinant functor
of the category $\mathsf{LCA}_{\mathbb{Z}}$. Here

\begin{enumerate}
\item for any LCA\ group $G$, $Ha(G)$ denotes the $\mathbb{R}_{>0}^{\times}%
$-torsor of all Haar measures on $G$, and

\item Deligne's Picard groupoid of virtual objects for $\mathsf{LCA}%
_{\mathbb{Z}}$ turns out to be isomorphic to the Picard groupoid of
$\mathbb{R}_{>0}^{\times}$-torsors.
\end{enumerate}
\end{theoremannounce}

Our central concept, the \textit{equivariant Haar measure}, will generalize
this theorem from $\mathsf{LCA}_{\mathbb{Z}}$ to $\mathsf{LCA}_{\mathfrak{A}}
$. However, this is in a way tautological since, guided by the above theorem,
we simply define it as the universal determinant functor of the category
$\mathsf{LCA}_{\mathfrak{A}}$. Thus, studying the categorical properties of
$\mathsf{LCA}_{\mathfrak{A}}$ will be the key tool in the paper.

Recall the notation $A_{p}:=A\otimes_{\mathbb{Q}}\mathbb{Q}_{p}$ and
$\widehat{A}:=A\otimes_{\mathbb{Z}}\widehat{\mathbb{Z}}$. We also prove the
following reciprocity law.

\begin{theoremannounce}
[Reciprocity Law]Let $\mathfrak{A}$ be an arbitrary order in a
finite-dimensional semisimple $\mathbb{Q}$-algebra $A$. Then the composition
of natural maps%
\[
K(A)\longrightarrow K(\widehat{A})\oplus K(A_{\mathbb{R}})\longrightarrow
K(\mathsf{LCA}_{\mathfrak{A}})
\]
is zero.
\end{theoremannounce}

See Theorem \ref{thm_Recip}. This can be used to understand the functors%
\[
K(A_{p})\longrightarrow K(\mathsf{LCA}_{\mathfrak{A}})\text{,}%
\]
originating from the $p$-adic completion of $A$, sending a projective $A_{p}%
$-module to itself, equipped with the $p$-adic topology. We demonstrate how to
use this for $p$-adic computations in Proposition \ref{prop_CC1}.

\begin{acknowledgement}
I heartily thank P. Arndt, B. Chow, D. Clausen, A. Huber, B. Morin for
discussions and input. We heartily thank the Freiburg Institute for Advanced
Studies (FRIAS) for providing truly wonderful working conditions and a very
inspiring environment.
\end{acknowledgement}

\section{\label{subsect_HaarIsDetFunctor}Where does this picture come from?}

We will now give a detailed exposition how the category $\mathsf{LCA}%
_{\mathfrak{A}}$ shows up. Those who are in a hurry may consider jumping to
\S \ref{sect_CatFramework} right away. On the other hand, this section may
serve as a survey.

Before Bloch and Kato proposed to interpret Tamagawa numbers in terms of
realizations of motives, Tamagawa numbers would just stem from the evaluation
of integrals on adelic points of algebraic groups; see the famous map of the
`land of Tamagawa numbers' \cite[bottom of p 336]{MR1086888}, and generally
the introduction of their article.

Measures or volumes have since often been interpreted in terms of determinant
functors in the sense of Deligne \cite{MR902592}, \cite{MR1914072}. We refer
the reader to \S \ref{section_HaarMeasureDet} for a quick summary. The
determinant line is a functor defined for all \textit{commutative} rings. If,
say, $R$ happens to be an integral domain (or more generally if
$\operatorname*{Spec}R$ is connected), then it is%
\[
\operatorname*{PMod}(R)^{\times}\longrightarrow\mathsf{Pic}_{R}^{\mathbb{Z}%
}\text{,}\qquad\qquad M\longmapsto\left(  \bigwedge
\nolimits^{\operatorname*{top}}M,\operatorname*{rk}M\right)  \text{,}%
\]
where $\operatorname*{PMod}(R)$ is the exact category of finitely generated
projective $R$-modules. $\operatorname*{PMod}(R)^{\times}$ denotes the same
category but with all non-isomorphisms deleted from the collection of
morphisms, and $\mathsf{Pic}_{R}^{\mathbb{Z}}$ is the Picard groupoid of
$\mathbb{Z}$-graded lines: Its objects are pairs $(L,n)$, where $L$ is a rank
one free $R$-module and $n\in\mathbb{Z}$. For a full definition and
description of the Picard groupoid structure we refer to \cite[\S 2.5]%
{MR1884523}, \cite[Chapter 1]{MR0437541}. For $R=\mathbb{R}$ one might hope
that this would give the Haar measure, in the following sense:

The top exterior power $\bigwedge\nolimits^{\operatorname*{top}}M$ corresponds
to a top form against which we can integrate. Nonetheless, this picture is
slightly flawed. The issue is that in the category $\operatorname*{PMod}%
(\mathbb{R})$ we have, vaguely speaking, orientation-reversing automorphisms
like $\mathbb{R}^{n}\rightarrow\mathbb{R}^{n}$, $x\mapsto-x$, changing the
sign of the top form. This differs from the Haar measure, which always returns
non-negative volumes. So there is a difference:%
\[%
\begin{tabular}
[c]{ccc}%
$\underset{\mathbb{R}^{n}}{%
{\displaystyle\int}
}f\,\mathrm{d}\omega$ & $\qquad\qquad$ & $\underset{\mathbb{R}^{n}}{%
{\displaystyle\int}
}f\,\left\vert \mathrm{d}\omega\right\vert $.\\
top exterior power &  & Haar measure\\
(signed) & $\qquad\qquad$ & (unsigned)
\end{tabular}
\]
This might just seem like an innocent subtlety, but the homological
implications of this difference run deeper. For example, the Picard groupoid
$\mathsf{Pic}_{R}^{\mathbb{Z}}$ has the non-trivial symmetry constraint%
\[%
\xymatrix@C=0.8in{
(L,n)\otimes(L^{\prime},n^{\prime})\cong(L\otimes_{R}L^{\prime},n+n^{\prime
}) \ar[r]^{{(-1)^{nn^{\prime}}\cdot s}} & (L^{\prime}\otimes_{R}L,n+n^{\prime
})\cong(L^{\prime},n^{\prime})\otimes(L,n)\text{,}
}%
\]
where $s:L\otimes_{R}L^{\prime}\cong L^{\prime}\otimes_{R}L$ denotes the usual
plain symmetry constraint of the symmetric monoidal structure of
$\operatorname*{PMod}(R)$. This structure only stems from the signed nature
and in the Haar setting there is no interplay between a sign and the ranks
$n,n^{\prime}$.

We will now be heading towards the viewpoint that once the Haar measure has
been set up carefully as a determinant functor in the sense of Deligne, the
corresponding $\mathfrak{A}$-equivariant analogue naturally takes values in
$K_{0}(\mathfrak{A},\mathbb{R})$ instead of $\mathbb{R}_{>0}^{\times}$ (and
this aspect cannot be seen when using the top exterior power determinant line
instead, because already in the non-equivariant setting there is a difference
between the two when it comes to handling the signs).

This viewpoint supplies \textquotedblleft\textit{In our approach Tamagawa
numbers are then elements of a relative algebraic }$K$\textit{-group}
$K_{0}(\mathfrak{A},\mathbb{R})$\textquotedblright\ \cite[p. 502]{MR1884523}
with a very concrete philosophical interpretation.\medskip

If $G$ is a locally compact abelian (LCA) group, it carries a Haar measure
$\mu_{G}$, well-defined up to rescaling by a positive real constant, see
\cite[Ch. II, \S 5]{MR1622489} for a review. If%
\[
G^{\prime}\hookrightarrow G\twoheadrightarrow G^{\prime\prime}%
\]
denotes an exact sequence of LCA groups, and given choices of Haar measures
$\mu_{G^{\prime}}$, $\mu_{G^{\prime\prime}}$ on $G^{\prime}$ and
$G^{\prime\prime}$ respectively, then there is a unique choice of Haar measure
$\mu_{G}$ on $G$ such that%
\begin{equation}
\int_{G}f(g)\mathrm{d}\mu_{G}=\int_{G^{\prime}}\int_{G^{\prime\prime}%
}f(g^{\prime}+g^{\prime\prime})\mathrm{d}\mu_{G^{\prime\prime}}\mathrm{d}%
\mu_{G^{\prime}}\label{lccac1}%
\end{equation}
holds for all $f\in C_{c}(G)$ (compactly supported continuous functions). In
particular, $\mu_{G}$ induces $\mu_{G^{\prime}}$ on $G^{\prime}$, and
$\mu_{G^{\prime\prime}}$ is the corresponding quotient measure. This is the
fundamental compatibility construction for Haar measures in exact sequences.

(a) These constructions give rise to the \textit{modulus}: Suppose
$f:G\overset{\sim}{\rightarrow}G$ is an automorphism of the LCA group. Then
for any choice of the Haar measure $\mu_{G}$, the pullback $f^{\ast}\mu_{G}$
is again a Haar measure. Being unique up to a positive scalar, this gives some
unique $c\in\mathbb{R}_{>0}^{\times}$ such that $\mu_{G}=c\cdot f^{\ast}%
\mu_{G}$. This construction can be promoted to a group homomorphism%
\[
\left\vert -\right\vert :\operatorname*{Aut}(G)\longrightarrow\mathbb{R}%
_{>0}^{\times}\text{.}%
\]
This homomorphism is well-defined and not longer dependent on the initial
choice of $\mu_{G}$. See \cite[Ch. II, \S 5.6]{MR1622489}

(b) If $G$ is a discrete (resp. compact) LCA group, one still can choose a
Haar measure, but there is additionally a canonical choice, namely $\mu
_{G}(\{g\}):=1$ for each element, i.e.%
\[
\mu_{G}(U):=\#U\text{\qquad(the counting measure)}%
\]
if $G$ is discrete, or the normalization $\mu_{G}(G)=1$ in the case if $G$ is
compact. So the cases of discrete or compact groups are special in the sense
that there is a canonical normalization.

(c) The exact sequence compatibility of the Haar measure amounts to the
following property: If%
\[%
\xymatrix{
G^{\prime} \ar@{^{(}->}[r] \ar[d]^{\sim}_{f} & G \ar@{->>}[r] \ar[d]^{\sim
}_{g} & G^{\prime\prime} \ar[d]^{\sim}_{h} \\
G^{\prime} \ar@{^{(}->}[r] & G \ar@{->>}[r] & G^{\prime\prime} \\
}%
\]
is a commutative diagram, then the modulus satisfies%
\begin{equation}
\left\vert g\right\vert =\left\vert f\right\vert \cdot\left\vert h\right\vert
\text{.}\label{lcey3b}%
\end{equation}
One can extract further laws describing the behaviour of Haar measures in
complexes, or more complicated commutative diagrams, e.g., a $(3\times
3)$-Lemma etc. Such rules have been studied to some extent by Oesterl\'{e}
\cite[Appendice 4]{MR762353} and de Smit \cite{MR1383444}.

This concludes our quick summary of the Haar measure, from the viewpoint of
picking an (inevitably non-canonical) choice of a normalization on each group.
However, there is also a completely different viewpoint of the same issue:

As was pointed out by Hoffmann and Spitzweck, the category $\mathsf{LCA}$ of
locally compact groups is an exact category \cite{MR2329311}. As a result, it
comes with its own universal determinant functor in the sense of Deligne
\cite{MR902592}:%
\[
\det:\mathsf{LCA}^{\times}\longrightarrow V(\mathsf{LCA})\text{,}%
\]
where $V(\mathsf{LCA})$ denotes the universal Picard groupoid of the category,
also known as Deligne's category of `virtual objects'.

Now, the properties of the Haar measure which we had recalled above, can be
used to construct a concrete determinant functor on $\mathsf{LCA}$ by hand: If
$A$ is an abelian group, recall that an $A$\emph{-torsor }$T$ is a simply
transitive $A$-set. Hence, for every $a\in A$ one gets an isomorphism of
$A$-sets $A\overset{\sim}{\rightarrow}T$, but there is no canonical choice for
such an identification. There is a tensor product%
\[
T\otimes T^{\prime}:=\{(t,t^{\prime})\in T\times T^{\prime}\}/\left\langle
(t,t^{\prime})\sim(s,s^{\prime}):\Leftrightarrow\exists a\in A\text{ such that
}t=a+s,t^{\prime}=a+s^{\prime}\right\rangle \text{.}%
\]
The $A$-torsors form a category $\mathsf{Tors}(A)$, and indeed a Picard
groupoid. The group $A$ itself with its natural $A$-set structure acts as the
unit $1_{\mathsf{Tors}(A)}$, and $T^{-1}$ is defined by letting $A$ act negatively.

Now, define the \emph{Haar functor}%
\begin{equation}
Ha(-):\mathsf{LCA}^{\times}\longrightarrow\mathsf{Tors}(\mathbb{R}%
_{>0}^{\times})\label{lcey4a}%
\end{equation}
by sending each LCA group $G$ to its set of Haar measures. Since the rescaling
by a positive scalar acts simply transitively on them, this is an
$\mathbb{R}_{>0}^{\times}$-torsor.

Let us recast the functoriality properties of the Haar measure in this
language: (a) If $f:G\overset{\sim}{\rightarrow}G$ denotes an automorphism,
then any trivialization of the Haar torsor, i.e. an isomorphism of
$\mathbb{R}_{>0}^{\times}$-sets $t:\mathbb{R}_{>0}^{\times}\overset{\sim
}{\rightarrow}Ha(G)$ gives a concrete measure $\mathrm{d}\mu_{G}\in Ha(G)$
attached to some real number, as does $f^{\ast}\mathrm{d}\mu_{G}$. The modulus
$\left\vert f\right\vert $ is the fraction of these numbers. This corresponds
to the canonical morphism%
\begin{equation}
\operatorname*{Aut}(G)\longrightarrow\pi_{1}(\mathsf{Tors}(\mathbb{R}%
_{>0}^{\times}))\cong\mathbb{R}_{>0}^{\times}\label{lcey4b}%
\end{equation}
coming from comparing trivializations. (c) For every short exact sequence
$G^{\prime}\hookrightarrow G\twoheadrightarrow G^{\prime\prime}$ the
compatibility of Equation \ref{lccac1} amounts to a canonical isomorphism of
torsors,%
\begin{equation}
Ha(G)\overset{\sim}{\longrightarrow}Ha(G^{\prime})\underset{\mathsf{Tors}%
(\mathbb{R}_{>0}^{\times})}{\otimes}Ha(G^{\prime\prime})\text{.}\label{lcey3a}%
\end{equation}
Taken altogether, this gives the fundamental datum of a determinant functor,
see Equation \ref{lcey0}.

So far, we have skipped over the analogue of (b). Here the story differs a
little: Write $\mathsf{LCA}_{D}$ (resp. $\mathsf{LCA}_{C}$) for the full
subcategories of $\mathsf{LCA}$ of discrete (resp. compact)\ LCA\ groups.
Restricted to these categories, the Haar torsor trivializes.

This can be seen by a variant of the Eilenberg swindle: Suppose $C\in
\mathsf{LCA}_{C}$ is compact. Then there is an exact sequence%
\begin{equation}
C\hookrightarrow\prod_{\mathbb{Z}}C\twoheadrightarrow\prod_{\mathbb{Z}%
}C\text{,}\label{lcey3d}%
\end{equation}
using that the product of compact groups is again compact (Tychonoff's
Theorem). Equation \ref{lcey3a} provides us with a canonical isomorphism
$Ha(\prod_{\mathbb{Z}}C)\overset{\sim}{\longrightarrow}Ha(C)\otimes
Ha(\prod_{\mathbb{Z}}C)$, and $\otimes$-multiplying with the inverse of
$Ha(\prod_{\mathbb{Z}}C)$, we get an isomorphism%
\[
t_{C}:1_{\mathsf{Tors}(A)}\overset{\sim}{\longrightarrow}Ha(C)\text{,}%
\]
i.e. a canonical trivialization of the torsor. This is also compatible with
exact sequences, i.e. if $C^{\prime}\hookrightarrow C\twoheadrightarrow
C^{\prime\prime}$ is an exact sequence of compact groups, the trivializations
$t_{(-)}$ for $C^{\prime},C,C^{\prime\prime}$ sit in a commutative diagram.

This argument also shows that the modulus is trivial: Firstly, the property of
Equation \ref{lcey3b} applied to%
\begin{equation}%
\xymatrix{
0 \ar@{^{(}->}[r] \ar[d]^{\sim}_{1} & {\prod_{\mathbb{Z}}C} \ar@{->>}%
[r]^{\sim} \ar[d]^{\sim}_{g} & {\prod_{\mathbb{Z}}C} \ar[d]^{\sim}_{h} \\
0 \ar@{^{(}->}[r] & {\prod_{\mathbb{Z}}C} \ar@{->>}[r]_{\sim} & {\prod
_{\mathbb{Z}}C} \\
}%
\label{lcey3c}%
\end{equation}
shows that $\left\vert g\right\vert =\left\vert h\right\vert $, i.e. the
modulus agrees when the automorphisms commute with an isomorphism. Every
automorphism $f:C\overset{\sim}{\rightarrow}C$ induces an automorphism
factor-wise to $\prod_{\mathbb{Z}}C$, call it $g$. Hence, applied to the
diagram%
\[%
\xymatrix{
C \ar@{^{(}->}[r] \ar[d]^{\sim}_{f} & {\prod_{\mathbb{Z}}C} \ar@{->>}%
[r] \ar[d]^{\sim}_{g} & {\prod_{\mathbb{Z}}C} \ar[d]^{\sim}_{h} \\
C \ar@{^{(}->}[r] & {\prod_{\mathbb{Z}}C} \ar@{->>}[r] & {\prod_{\mathbb{Z}}C}
\\
}%
\]
whose rows are made from Equation \ref{lcey3d}, the same property shows that
$\left\vert g\right\vert =\left\vert f\right\vert \cdot\left\vert h\right\vert
$, where $h$ is the induced map on the quotient. However, by our previous
observation from Equation \ref{lcey3c}, $\left\vert g\right\vert =\left\vert
h\right\vert $, so we deduce $\left\vert f\right\vert =1$. The automorphism
$f$ was arbitrary, and it follows that the modulus is always $1$ for every
compact group.

For discrete groups $D\in\mathsf{LCA}_{D}$ instead take the short exact
sequence $D\hookrightarrow\coprod_{\mathbb{Z}}D\twoheadrightarrow
\coprod_{\mathbb{Z}}D$ based on coproducts. These are still discrete. The same
argument goes through.

While the argument here might look a little different, this is the torsor
reformulation of property (b) saying that compact and discrete groups have a
canonical normalization.

We summarize our observations as follows.

\begin{lemma}
\label{lemma_ConstraintDetFunctor}Let $G\in\mathsf{LCA}$ be a compact group or
a discrete group. Then for every determinant functor $\mathcal{P}%
:\mathsf{C}^{\times}\longrightarrow\mathsf{P}$, with $\mathsf{P}$ some Picard
groupoid, the value $\mathcal{P}(f)\in\pi_{1}(\mathsf{P})$ of any automorphism
$f:G\overset{\sim}{\rightarrow}G$ is trivial, i.e. the neutral element of the group.
\end{lemma}

\begin{proof}
A careful reading of the previous arguments shows that they only use axioms of
a determinant functor, as in Definition \ref{def_DetFunctor}, so they hold in
general. The construction of the modulus translates into mapping the
isomorphism in the Picard groupoid $\mathsf{P}$ to its class in $\pi_{1}$.
\end{proof}

Now let us compare the usual determinant line with the Haar measure:

\begin{example}
\label{ex_DetA1}Firstly, we consider the category $\mathsf{Vect}%
_{f}(\mathbb{R})$ of finite-dimensional real vector spaces. Deligne's
universal determinant functor is%
\[
\det:\mathsf{Vect}_{f}(\mathbb{R})^{\times}\longrightarrow\mathsf{Pic}%
_{\mathbb{R}}^{\mathbb{Z}}%
\]
with values in the Picard groupoid of graded lines. In particular, for any
automorphism $f:\mathbb{R}^{n}\overset{\sim}{\rightarrow}\mathbb{R}^{n}$, the
value lies in $\pi_{1}(\mathsf{Pic}_{\mathbb{R}}^{\mathbb{Z}})=\mathbb{R}%
^{\times}$. This can be interpreted as a signed/oriented measure. Every
finite-dimensional real vector space can also be regarded as an LCA group; the
functor $\psi_{\infty}:\mathsf{Vect}_{f}(\mathbb{R})\rightarrow\mathsf{LCA}$
is exact. However, no determinant functor on $\mathsf{LCA}$ can see the
orientation. To see this, consider the map of multiplication with $-1$, which
induces a map of exact sequences%
\[%
\xymatrix{
\mathbb{Z} \ar@{^{(}->}[r] \ar[d]^{\sim}_{-1} & {\mathbb{R}} \ar@{->>}%
[r] \ar[d]^{\sim}_{-1} & \mathbb{T} \ar[d]^{\sim}_{-1} \\
\mathbb{Z} \ar@{^{(}->}[r] & {\mathbb{R}} \ar@{->>}[r] & \mathbb{T}
}%
\]
and by Equation \ref{lcey3b} (or rather its analogue for general Picard
groupoids) we obtain $\mathcal{P}(-1_{\mathbb{R}})=\mathcal{P}(-1_{\mathbb{Z}%
})\cdot\mathcal{P}(-1_{\mathbb{T}})$. However, $\mathbb{Z}$ is
discrete\ (resp. $\mathbb{T}$ compact), so by Lemma
\ref{lemma_ConstraintDetFunctor}, we must have $\mathcal{P}(-1_{\mathbb{Z}%
})=1$, $\mathcal{P}(-1_{\mathbb{T}})=1$. Thus, $\mathcal{P}(-1_{\mathbb{R}%
})=1$. A more careful computation shows that the functor $\psi_{\infty}$
induces the map $\pi_{1}(\mathsf{Pic}_{R}^{\mathbb{Z}})\rightarrow\pi
_{1}(V(\mathsf{LCA}))$, $\alpha\mapsto\left\vert \alpha\right\vert $, i.e. the
Haar measure differs from the real determinant line just by forgetting the
sign. We give a rigorous proof for this claim in Proposition \ref{prop_CC1}
based on $K$-theory.
\end{example}

\begin{example}
\label{ex_DetA2}The analogous effect over the $p$-adics is more drastic. For
any $\alpha\in\mathbb{Z}_{p}^{\times}$ consider the map of exact sequences%
\[%
\xymatrix{
{\mathbb{Z}}_p \ar@{^{(}->}[r] \ar[d]^{\sim}_{\cdot\alpha} & {\mathbb{Q}_p}
\ar@{->>}[r] \ar[d]^{\sim}_{\cdot\alpha} & {{\mathbb{Q}}_p}/{{\mathbb{Z}}_p}
\ar[d]^{\sim}_{\cdot\alpha} \\
{\mathbb{Z}}_p \ar@{^{(}->}[r] & {\mathbb{Q}_p} \ar@{->>}[r] & {{\mathbb{Q}%
}_p}/{{\mathbb{Z}}_p} \text{.}
}%
\]
This time $\mathbb{Z}_{p}$ is compact and $\mathbb{Q}_{p}/\mathbb{Z}_{p}$
discrete, but again the same Lemma \ref{lemma_ConstraintDetFunctor} shows that
we must have $\mathcal{P}(\alpha_{\mathbb{Z}_{p}})=1$, $\mathcal{P}%
(\alpha_{\mathbb{Q}_{p}/\mathbb{Z}_{p}})=1$, and therefore $\mathcal{P}%
(\alpha_{\mathbb{Q}_{p}})=1$. Indeed, the difference between the $p$-adic
determinant line and the Haar measure is taking the $p$-adic absolute value%
\[
\pi_{1}(\mathsf{Pic}_{\mathbb{Q}_{p}}^{\mathbb{Z}})\longrightarrow\pi
_{1}(V(\mathsf{LCA}))\text{,}\qquad\alpha\longmapsto\left\vert \alpha
\right\vert _{p}\text{.}%
\]
Again, we see that the Haar measure is a lot coarser than the determinant
line.\ See Proposition \ref{prop_CC1} for a rigorous proof of this claim.
\end{example}

The previous examples show that the Haar torsor, although related, is really a
bit of a different invariant than the standard determinant line of a ring. The
following result is more or less implicit in the work of Clausen
\cite{clausen}:

\begin{theoremannounce}
The Haar functor $Ha:\mathsf{LCA}^{\times}\rightarrow\mathsf{Tors}%
(\mathbb{R}_{>0}^{\times})$ is the universal determinant functor of the
category $\mathsf{LCA}$. In particular, Deligne's Picard groupoid of virtual
objects for $\mathsf{LCA}$ is isomorphic to the Picard groupoid of
$\mathbb{R}_{>0}^{\times}$-torsors.
\end{theoremannounce}

This will be Theorem \ref{thm_UnivOfHaarTorsor} below. This result motivates
to define an \textquotedblleft equivariant Haar measure\textquotedblright%
\ simply by replacing $\mathsf{LCA}$ by the category $\mathsf{LCA}%
_{\mathfrak{A}}$ of locally compact $\mathfrak{A}$-modules and taking the
universal determinant functor of this category. By the above theorem, it
follows that for the non-equivariant setting $\mathfrak{A}:=\mathbb{Z}$ we get
the classical Haar measure.

\begin{definition}
\label{def_EquivHaarMeasure}Let $\mathfrak{A}$ be an order in a
finite-dimensional semisimple $\mathbb{Q}$-algebra $A$. Define the
\emph{equivariant Haar measure} functor to be Deligne's universal determinant
functor of the exact category $\mathsf{LCA}_{\mathfrak{A}}$, i.e.%
\[
Ha_{\mathfrak{A}}:\mathsf{LCA}_{\mathfrak{A}}^{\times}\longrightarrow
V(\mathsf{LCA}_{\mathfrak{A}})\text{,}%
\]
where $V(\mathsf{LCA}_{\mathfrak{A}})$ denotes Deligne's virtual objects in
this category. We may also call $V(\mathsf{LCA}_{\mathfrak{A}})$ the Picard
groupoid of \emph{equivariant volumes}.
\end{definition}

Now, let us connect this back to the ETNC.

\begin{theoremannounce}
Let $\mathfrak{A}$ be a regular order in a finite-dimensional semisimple
$\mathbb{Q}$-algebra $A$. Then%
\[
\pi_{1}V(\mathsf{LCA}_{\mathfrak{A}})\cong K_{0}(\mathfrak{A},\mathbb{R}%
)\text{,}%
\]
and, stronger, the Picard groupoid of equivariant volumes $V(\mathsf{LCA}%
_{\mathfrak{A}})$ is equivalent to the $1$-truncation of the $(-1)$-shift of
the fiber of the morphism $K(\mathfrak{A})\longrightarrow K(A_{\mathbb{R}})$.
\end{theoremannounce}

See Theorem \ref{marker_ComputeVLCAA}. The corresponding universal Picard
categories $V(\mathsf{LCA}_{D})$ resp. $V(\mathsf{LCA}_{C})$ are the trivial
Picard groupoid with one object and no non-trivial automorphisms.

\subsection{\label{subsect_Nenashev}Nenashev's presentation}

A \emph{double (short) exact sequence} in $\mathsf{LCA}_{\mathfrak{A}}$ is the
datum of two short exact sequences%
\[
\mathrm{Yin}:A\overset{p}{\hookrightarrow}B\overset{r}{\twoheadrightarrow
}C\qquad\text{and}\qquad\mathrm{Yang}:A\overset{q}{\hookrightarrow}%
B\overset{s}{\twoheadrightarrow}C\text{,}%
\]
where (as we can see) only the maps may differ, but the three objects agree.
We write%
\[%
\xymatrix{
A \ar@<1ex>@{^{(}->}[r]^{p} \ar@<-1ex>@{^{(}->}[r]_{q} & B \ar@<1ex>@{->>}%
[r]^{r} \ar@<-1ex>@{->>}[r]_{s} & C
}%
\]
as a convenient shorthand. The study of this concept was pioneered by Nenashev.

\begin{theorem}
\label{thm_NenashevStylePresentation}Let $\mathfrak{A}$ be a regular order in
a finite-dimensional semisimple $\mathbb{Q}$-algebra. Then $K_{0}%
(\mathfrak{A},\mathbb{R})$, or equivalently $K_{1}(\mathsf{LCA}_{\mathfrak{A}%
})$, has the following presentation as an abelian group:

\begin{enumerate}
\item Attach a generator to each double exact sequence%
\[%
\xymatrix{
A \ar@<1ex>@{^{(}->}[r]^{p} \ar@<-1ex>@{^{(}->}[r]_{q} & B \ar@<1ex>@{->>}%
[r]^{r} \ar@<-1ex>@{->>}[r]_{s} & C
}%
\]
with $A,B,C$ locally compact right $\mathfrak{A}$-modules.

\item Whenever the yin and yang exact sequence agree, i.e.,%
\[%
\xymatrix{
A \ar@<1ex>@{^{(}->}[r]^{p}_{=} \ar@<-1ex>@{^{(}->}[r]_{p} & B \ar@
<1ex>@{->>}[r]^{r}_{=} \ar@<-1ex>@{->>}[r]_{r} & C\text{,}
}%
\]
we declare the class to be zero.

\item Suppose there is a (not necessarily commutative) diagram%
\[%
\xymatrix@W=0.3in@H=0.3in{
A \ar@<1ex>@{^{(}->}[r] \ar@<1ex>@{^{(}->}[d] \ar@<-1ex>@{^{(}->}%
[d] \ar@<-1ex>@{^{(}->}[r] & B \ar@<1ex>@{->>}[r] \ar@<-1ex>@{->>}%
[r] \ar@<1ex>@{^{(}->}[d] \ar@<-1ex>@{^{(}->}[d] & C \ar@<1ex>@{^{(}%
->}[d] \ar@<-1ex>@{^{(}->}[d] \\
D \ar@<1ex>@{^{(}->}[r] \ar@<-1ex>@{^{(}->}[r] \ar@<1ex>@{->>}[d] \ar@
<-1ex>@{->>}[d] & E \ar@<1ex>@{->>}[r] \ar@<-1ex>@{->>}[r] \ar@<1ex>@{->>}%
[d] \ar@<-1ex>@{->>}[d] & F \ar@<1ex>@{->>}[d] \ar@<-1ex>@{->>}[d] \\
G \ar@<1ex>@{^{(}->}[r] \ar@<-1ex>@{^{(}->}[r] & H \ar@<1ex>@{->>}%
[r] \ar@<-1ex>@{->>}[r] & I \\
}%
\]
whose rows $Row_{i}$ and columns $Col_{j}$ are double exact sequences. Suppose
after deleting all yin (resp. all yang) exact sequences, the remaining diagram
commutes. Then impose the relation%
\begin{equation}
Row_{1}-Row_{2}+Row_{3}=Col_{1}-Col_{2}+Col_{3}\text{.}\label{l_C_Nenashev}%
\end{equation}

\end{enumerate}
\end{theorem}

\begin{proof}
[Proof of Theorem \ref{thm_NenashevStylePresentation}]This is precisely
Nenashev's presentation of $K_{1}$, see \cite[Theorem]{MR1637539}. Hence, the
validity of this presentation follows by our main theorem, which expresses
$K_{0}(\mathfrak{A},\mathbb{R})$ as the $K_{1}$-group $K_{1}(\mathsf{LCA}%
_{\mathfrak{A}})$.
\end{proof}

Let us demonstrate how to work with this presentation in two key cases.

\begin{example}
\label{example_A1}The map%
\[
K_{1}(A_{\mathbb{R}})\longrightarrow K_{0}(\mathfrak{A},\mathbb{R})
\]
in Sequence \ref{l_Intro1} is given as follows: Write any matrix
$M\in\operatorname*{GL}(A_{\mathbb{R}})$ as an automorphism of
$X:=A_{\mathbb{R}}^{m}$ with $m$ sufficiently large. Then send%
\[
\left[  \varphi:X\overset{\sim}{\rightarrow}X\right]  \longmapsto\left[
\xymatrix{
0 \ar@<1ex>@{^{(}->}[r]^{0} \ar@<-1ex>@{^{(}->}[r]_{0} & X \ar@<1ex>@{->>}%
[r]^{\varphi} \ar@<-1ex>@{->>}[r]_{1} & X\text{,}
}%
\right]  \text{,}%
\]
where on the right $X$ is now interpreted as a locally compact $\mathfrak{A}%
$-module, equipped with the topology coming from the real vector space
structure. This is the map transforming automorphisms in $\operatorname*{PMod}%
(A_{\mathbb{R}})$ to elements in the Nenashev presentation.\ This fact is
standard and discussed in \cite[Ch. IV, Example 9.6.2]{MR3076731},
\cite[Equation (2.2)]{MR1637539} (in these sources translate Weibel's left
sequence to the bottom sequence in Nenashev's paper, as opposed to the top
one, to get the signs to match).
\end{example}

\begin{example}
From the exactness of Sequence \ref{l_Intro1} we expect that the composition%
\[
K_{1}(\mathfrak{A})\longrightarrow K_{1}(A_{\mathbb{R}})\longrightarrow
K_{0}(\mathfrak{A},\mathbb{R})
\]
is zero. Let us confirm this by a direct computation in the Nenashev
presentation. We work in the category $\mathsf{LCA}_{\mathfrak{A}}$. Let
$\mathfrak{X}\in\operatorname*{PMod}(\mathfrak{A})$, which we tacitly equip
with the discrete topology, while we equip $X_{\mathbb{R}}:=\mathfrak{X}%
\otimes_{\mathbb{Z}}\mathbb{R}$ with the locally compact topology which comes
from its real vector space structure. Define%
\[
T:=X_{\mathbb{R}}/\mathfrak{X}\qquad\text{in}\qquad\mathsf{LCA}_{\mathfrak{A}%
}\text{.}%
\]
Note that $X_{\mathbb{R}}/\mathfrak{X}$, as a cokernel in the category
$\mathsf{LCA}_{\mathfrak{A}}$, carries the quotient topology. Since
$\mathfrak{X}$ is a lattice of full rank in $X_{\mathbb{R}}$ (see
\S \ref{subsect_Basics}), $T$ is topologically a compact torus of dimension
$\dim_{\mathbb{R}}(X_{\mathbb{R}})$. We begin with the $(3\times3)$-diagram%
\[%
\xymatrix@W=0.3in@H=0.3in{
0 \ar@<1ex>@{^{(}->}[r] \ar@<1ex>@{^{(}->}[d] \ar@<-1ex>@{^{(}->}%
[d] \ar@<-1ex>@{^{(}->}[r] & {\mathfrak{X}} \ar@<1ex>@{->>}[r]^{\varphi}
\ar@<-1ex>@{->>}[r]_{1} \ar@<1ex>@{^{(}->}[d] \ar@<-1ex>@{^{(}->}%
[d] & {\mathfrak{X}} \ar@<1ex>@{^{(}->}[d] \ar@<-1ex>@{^{(}->}[d] \\
0 \ar@<1ex>@{^{(}->}[r] \ar@<-1ex>@{^{(}->}[r] \ar@<1ex>@{->>}[d] \ar@
<-1ex>@{->>}[d] & {X_{\mathbb{R}}} \ar@<1ex>@{->>}[r]^{\varphi\otimes
\mathbb{R} } \ar@<-1ex>@{->>}[r]_{1} \ar@<1ex>@{->>}[d] \ar@<-1ex>@{->>}%
[d] & {X_{\mathbb{R}}} \ar@<1ex>@{->>}[d] \ar@<-1ex>@{->>}[d] \\
0 \ar@<1ex>@{^{(}->}[r] \ar@<-1ex>@{^{(}->}[r] & T \ar@<1ex>@{->>}%
[r]^{\overline{\varphi\otimes\mathbb{R}}} \ar@<-1ex>@{->>}[r]_{1} & T\text{,}
\\
}%
\]
where we wrote $\overline{\varphi\otimes\mathbb{R}}$ to refer to what the map
$\varphi\otimes\mathbb{R}$ induces on the torus quotient. As we see from
Equation \ref{l_C_Nenashev}, this diagram induces the following relation.%
\begin{equation}
\left[
\xymatrix{
0 \ar@<1ex>@{^{(}->}[r] \ar@<-1ex>@{^{(}->}[r] & {\mathfrak{X}} \ar@
<1ex>@{->>}[r]^{\varphi} \ar@<-1ex>@{->>}[r]_{1} & {\mathfrak{X}}
}%
\right]  +\left[
\xymatrix{
0 \ar@<1ex>@{^{(}->}[r] \ar@<-1ex>@{^{(}->}[r] & T \ar@<1ex>@{->>}%
[r]^{\overline{\varphi\otimes\mathbb{R}}} \ar@<-1ex>@{->>}[r]_{1} & T
}%
\right]  =\left[
\xymatrix{
0 \ar@<1ex>@{^{(}->}[r] \ar@<-1ex>@{^{(}->}[r] & {X_{\mathbb{R}}}
\ar@<1ex>@{->>}[r]^{\varphi\otimes\mathbb{R} } \ar@<-1ex>@{->>}[r]_{1}
& {X_{\mathbb{R}}}
}%
\right] \label{lrr1}%
\end{equation}
Next, consider the $(3\times3)$-diagram%
\begin{equation}%
\xymatrix@W=0.3in@H=0.3in{
0 \ar@<1ex>@{^{(}->}[r] \ar@<1ex>@{^{(}->}[d] \ar@<-1ex>@{^{(}->}%
[d] \ar@<-1ex>@{^{(}->}[r] & {\mathfrak{X}} \ar@<1ex>@{->>}[r]^{\varphi}
\ar@<-1ex>@{->>}[r]_{1} \ar@<1ex>@{^{(}->}[d] \ar@<-1ex>@{^{(}->}%
[d] & {\mathfrak{X}} \ar@<1ex>@{^{(}->}[d] \ar@<-1ex>@{^{(}->}[d] \\
0 \ar@<1ex>@{^{(}->}[r] \ar@<1ex>@{->>}[d] \ar@<-1ex>@{->>}[d] \ar@
<-1ex>@{^{(}->}[r] & {\coprod_{\mathbb{N}} \mathfrak{X}} \ar@<1ex>@{->>}%
[r]^{\varphi} \ar@<-1ex>@{->>}[r]_{1} \ar@<1ex>@{->>}[d] \ar@<-1ex>@{->>}%
[d] & {\coprod_{\mathbb{N}} \mathfrak{X}} \ar@<1ex>@{->>}[d] \ar@
<-1ex>@{->>}[d] \\
0 \ar@<1ex>@{^{(}->}[r] \ar@<-1ex>@{^{(}->}[r] & {\coprod_{\mathbb{N}}
\mathfrak{X}} \ar@<1ex>@{->>}[r]^{\varphi} \ar@<-1ex>@{->>}[r]_{1}
& {\coprod_{\mathbb{N}} \mathfrak{X}}\text{,} \\
}%
\label{lrr0}%
\end{equation}
where the middle and right downward exact sequence stem from the injection
sending $m$ to the sequence $(m,0,0,0,\ldots)$ and the epic sending
$(m_{1},m_{2},\ldots)$ to $(m_{2},m_{3},\ldots)$. Here Equation
\ref{l_C_Nenashev} yields the relation%
\begin{equation}
\left[
\xymatrix{
0 \ar@<1ex>@{^{(}->}[r] \ar@<-1ex>@{^{(}->}[r] & {\mathfrak{X}} \ar@
<1ex>@{->>}[r]^{\varphi} \ar@<-1ex>@{->>}[r]_{1} & {\mathfrak{X}}
}%
\right]  =0\text{.}\label{lrr2}%
\end{equation}
The analogous game can be played when replacing $\mathfrak{X}$ and
$\coprod_{\mathbb{N}}\mathfrak{X}$ by the pair $T$ and $\prod_{\mathbb{N}}T$.
We get%
\begin{equation}
\left[
\xymatrix{
0 \ar@<1ex>@{^{(}->}[r] \ar@<-1ex>@{^{(}->}[r] & T \ar@<1ex>@{->>}%
[r]^{\overline{\varphi\otimes\mathbb{R}}} \ar@<-1ex>@{->>}[r]_{1} & T
}%
\right]  =0\text{.}\label{lrr3}%
\end{equation}
All three relations of Equations \ref{lrr1}, \ref{lrr2}, \ref{lrr3} can be
combined to give%
\[
\left[
\xymatrix{
0 \ar@<1ex>@{^{(}->}[r] \ar@<-1ex>@{^{(}->}[r] & {X_{\mathbb{R}}}
\ar@<1ex>@{->>}[r]^{\varphi\otimes\mathbb{R} } \ar@<-1ex>@{->>}[r]_{1}
& {X_{\mathbb{R}}}
}%
\right]  =0\text{.}%
\]
In view of Example \ref{example_A1}, this shows that if $\varphi$ comes from
an automorphism $\varphi:\mathfrak{X}\overset{\sim}{\rightarrow}\mathfrak{X}$,
then its class in $K_{1}(\mathsf{LCA}_{\mathfrak{A}})$ and thus in
$K_{0}(\mathfrak{A},\mathbb{R})$ vanishes. Of course this argument is just a
variation of the Eilenberg swindle, performed in this explicit presentation,
and corresponds under the bridge to the Haar measure in
\S \ref{subsect_HaarIsDetFunctor}\ to the fact that the Haar measure on
compact resp. discrete LCA\ groups admits a canonical normalization.
\end{example}

\begin{example}
We note that there is nothing like Diagram \ref{lrr0} with $A_{\mathbb{R}}$
instead of $\mathfrak{X}$, since both $\prod_{\mathbb{N}}A_{\mathbb{R}}$ as
well as $\coprod_{\mathbb{N}}A_{\mathbb{R}}$ fail to be locally compact.
\end{example}

\subsection{Strategy of proof}

In the special case $\mathfrak{A}=\mathbb{Z}$ and $A=\mathbb{Q}$, our main
result was proven by Clausen \cite{clausen}, using $\infty$-categorical
methods, and not in the context of the ETNC. We had already developed a
generalization of Clausen's result for maximal orders $\mathcal{O}$ (i.e. the
ring of integers) when $A$ is a number field \cite{obloc}, and in this paper
we try to adapt the proof given \textit{op. cit.} as far as possible. Let us
list some of the obstacles where things become harder:

\begin{itemize}
\item The rings are now \textit{non}-commutative and duality swaps left and
right modules.

\item We introduce a new method to create certain projective covers based on
covering space theory, see Lemma \ref{Lemma_RZOrRTHulls}. While this could
certainly be replaced by an argument \textquotedblleft by
hand\textquotedblright, this is a very efficient new tool.

\item We run into two separate issues when handling the compactly generated
modules, see \S \ref{sect_CGPiece}. We explain how to solve them there, and
this fix even works for orders $\mathfrak{A}$ of infinite global dimension.
This might be important for the issue to remove the restriction to regular
orders in our main theorem in the future. The principal trick is to
\textquotedblleft move\textquotedblright\ infinitely long resolutions into a
piece of the category whose $K$-theory gets killed by an Eilenberg swindle.
This way, that the resolution would not have been finite, does not need to
bother us anymore.

\item Most delicate: Ibid. all rings were Dedekind. Thus, if $M\in
\mathsf{LCA}_{\mathcal{O}}$ was for example a torus $\mathbb{T}^{n}$
topologically, then it followed that $M$ was divisible and thus an injective
$\mathcal{O}$-module (which is useful to split off direct summands). But over
a general order $\mathfrak{A}$, more precisely over the non-hereditary ones,
divisible modules need \textit{not} be injective. In particular, some nice
decomposition results in \cite{obloc} simply would be false in the setup of
the present paper. This complicates all proofs which previously were relying
on these decompositions.
\end{itemize}

\section{\label{sect_CatFramework}Categorical framework}

\subsection{Basics\label{subsect_Basics}}

We recall that a module is called \emph{semisimple} if it splits as a
(possibly infinite) direct sum of simple modules. A unital associative (not
necessarily commutative) ring $A$ is called \emph{semisimple} if one (then
all) of the following properties hold: (a) $A$ has trivial Jacobson radical,
(b) the category of left $A$-modules is split exact, (c) all left $A$-modules
are semisimple, (d) all left $A$-modules are injective, (e) all left
$A$-modules are projective, (f) its opposite algebra $A^{op}$ is semisimple.

In particular, by (f) the conditions (b)-(e) could equivalently be demanded to
hold for right $A$-modules.

Suppose $A$ is a finite-dimensional semisimple $\mathbb{Q}$-algebra. By finite
dimension we mean that $A$ is finite-dimensional as a $\mathbb{Q}$-vector
space. An \emph{order} $\mathfrak{A}\subseteq A$ is a subring of $A$ which is
a finitely generated $\mathbb{Z}$-module such that $\mathbb{Q}\cdot
\mathfrak{A}=A$. More accurately, this should be called a $\mathbb{Z}$-order,
but we will not consider any other types of orders.

As an abelian group, we have $\mathfrak{A}\simeq\mathbb{Z}^{n}$, where
$n:=\dim_{\mathbb{Q}}A$. To see this, tensor $\mathfrak{A}\hookrightarrow A$
over $\mathbb{Z}$ with $\mathbb{Q}$, giving the injection $\mathfrak{A}%
\otimes_{\mathbb{Z}}\mathbb{Q}\hookrightarrow A$, and the condition
$\mathbb{Q}\cdot\mathfrak{A}=A$ implies that this map must be surjective, and
thus an isomorphism. However, as $\mathfrak{A}\subseteq A$ forces
$\mathfrak{A}$ to be torsion-free, it can only be of the shape $\mathbb{Z}^{a}
$ for some $a$ anyway, and then $\mathfrak{A}\otimes_{\mathbb{Z}}%
\mathbb{Q}\simeq\mathbb{Q}^{a}$, proving $a=n$.

Moreover, we see that $\mathbb{Q}\cdot\mathfrak{A}\cong\mathbb{Q}%
\otimes_{\mathbb{Z}}\mathfrak{A}$ and $\mathfrak{A}$ is a full rank lattice in
the finite-dimensional real vector space $\mathbb{R}\cdot\mathfrak{A}%
\cong\mathbb{R}\otimes_{\mathbb{Q}}A$. That is, the quotient $(\mathbb{R}%
\cdot\mathfrak{A})/\mathfrak{A}$ is an $n$-dimensional real torus, topologically.

If $R$ denotes a ring, we speak of an \emph{algebraic }$R$\emph{-module} $M$
whenever we want to stress that we ignore any topological structures which are
present on both $R$ and $M$. Alternatively, regard both $R$ and $M$ as
equipped with the discrete topology.

\begin{example}
For every finite group $G$, the group ring $\mathbb{Z}[G]$ is an order in the
rational group algebra $\mathbb{Q}[G]$.
\end{example}

\subsection{Quasi-abelian categories}

We will work a lot with exact categories. We shall use the conventions of
B\"{u}hler \cite{MR2606234}. We write `$\hookrightarrow$' (resp.
`$\twoheadrightarrow$') to denote an admissible monic (resp. admissible epic).
The concept of a quasi-abelian category might be less well-known. It is also
explained \textit{loc. cit.}, but all we really need to know are the following
facts: (a)\ Quasi-abelian categories $\mathsf{C}$ are a particular type of
exact category, (b) In a quasi-abelian category every morphism has a kernel
and cokernel.

Unfortunately, the latter does not yet suffice to make $\mathsf{C}$ an abelian
category. By an insight of Hoffmann and Spitzweck \cite{MR2329311} the
category $\mathsf{LCA}$ of locally compact abelian (from now on: LCA) groups
is quasi-abelian. We shall recall this below in detail, but for the moment
just note that if $\mathbb{R}_{\delta}$ denotes the real numbers equipped with
the discrete topology, then the identity $f:\mathbb{R}_{\delta}\rightarrow
\mathbb{R}$ is a continuous group homomorphism with zero kernel and zero
cokernel, but $f$ still fails to be an isomorphism. Such a behaviour would be
impossible in an abelian category.

\subsection{Locally compact modules}

Suppose $\mathfrak{A}$ is an order inside a finite-dimensional semisimple
$\mathbb{Q}$-algebra $A$. Note that it is sufficient to specify $\mathfrak{A}$
since $A=\mathbb{Q}\otimes_{\mathbb{Z}}\mathfrak{A}$ is uniquely determined by it.

\begin{definition}
We define a category $\mathsf{LCA}_{\mathfrak{A}}$ such that

\begin{enumerate}
\item an object $M$ is a right $\mathfrak{A}$-module along with the datum of
an LCA group structure on its additive group $(M;+)$ such that right
multiplication by any $\alpha\in\mathfrak{A}$ is a continuous endomorphism
$(M;+)\overset{\cdot\alpha}{\longrightarrow}(M;+)$,

\item a morphism $M\rightarrow M^{\prime}$ is a continuous right
$\mathfrak{A}$-module homomorphism.
\end{enumerate}

We write $\left.  _{\mathfrak{A}}\mathsf{LCA}\right.  $ for the corresponding
category of left modules. We shall write $\mathsf{LCA}$ as a shorthand for
$\mathsf{LCA}_{\mathbb{Z}}$.
\end{definition}

As an equivalent alternative description, we may view the ring $\mathfrak{A}$
as equipped with the discrete topology and take locally compact topological
right $\mathfrak{A}$-modules as objects for $\mathsf{LCA}_{\mathfrak{A}}$.

\begin{remark}
\label{Rmk_OppositeOrder}Note that the opposite ring $\mathfrak{A}^{op}$ is an
order in the opposite algebra $A^{op}$ (which is also semisimple,
\S \ref{subsect_Basics}), so we have $\left.  _{\mathfrak{A}}\mathsf{LCA}%
\right.  =\left.  \mathsf{LCA}_{\mathfrak{A}^{op}}\right.  $. Hence, for most
considerations it will be sufficient to discuss them for right modules and the
left module case will be implied by switching to the opposite order in the
opposite algebra.
\end{remark}

Next, let $\mathbb{T}$ denote the unit circle (or equivalently $\mathbb{R}%
/\mathbb{Z}$) in the category $\mathsf{LCA}$. The following result naturally
extends the observation of Hoffmann and Spitzweck that $\mathsf{LCA}$ is a
quasi-abelian category.

\begin{proposition}
\label{Prop_LCAIsQuasiAbelianAndPontryaginDuality}The category $\mathsf{LCA}%
_{\mathfrak{A}}$ is a quasi-abelian exact category. There is an exact functor%
\begin{align*}
(-)^{\vee}:\mathsf{LCA}_{\mathfrak{A}}^{op}  & \longrightarrow\left.
_{\mathfrak{A}}\mathsf{LCA}\right. \\
M  & \longmapsto\operatorname*{Hom}(M,\mathbb{T})\text{,}%
\end{align*}
where the continuous right $\mathfrak{A}$-module homomorphism group
$\operatorname*{Hom}(M,\mathbb{T})$ is equipped with the compact-open topology
(that is: on the level of the underlying LCA group $(M;+)$ this is the
Pontryagin dual), and the left action%
\[
(\alpha\cdot\varphi)(m):=\varphi(m\cdot\alpha)\qquad\text{for all}\qquad
\alpha\in\mathfrak{A}\text{, }m\in M\text{.}%
\]
There is an analogous exact functor%
\[
(-)^{\vee}:\left.  _{\mathfrak{A}}\mathsf{LCA}^{op}\right.  \longrightarrow
\mathsf{LCA}_{\mathfrak{A}}%
\]
such that there is a natural equivalence of functors from the identity functor
to double dualization,%
\[
\eta:\operatorname*{id}\longrightarrow(-)^{\vee}\circ\left[  (-)^{\vee
}\right]  ^{op}\text{.}%
\]
In less technical terms: For every object $M\in\mathsf{LCA}_{\mathfrak{A}}$
there exists a reflexivity isomorphism $\eta(M):M\overset{\sim}%
{\longrightarrow}M^{\vee\vee}$, and the isomorphisms $\eta(M)$ are natural in
$M$. Under the forgetful functor $\mathsf{LCA}_{\mathfrak{A}}\rightarrow
\mathsf{LCA}$ (resp. $\left.  _{\mathfrak{A}}\mathsf{LCA}\right.
\rightarrow\mathsf{LCA}$), both duality functors restrict to ordinary
Pontryagin duality.
\end{proposition}

We warn the reader of the subtlety that $\operatorname*{Hom}(M,\mathbb{T})$
with the compact-open topology is again locally compact, but it is not at all
true that $\operatorname*{Hom}(M,N)$ will be locally compact for other choices
of $N$. See Moskowitz \cite[Theorem 4.3 (1$^{\prime}$)]{MR0215016} for a discussion.

\begin{proof}
(1) The proof of Hoffmann and Spitzweck for $\mathsf{LCA}$ carries over,
\cite[Proposition 1.2]{MR2329311}. It shows that $\mathsf{LCA}_{\mathfrak{A}}
$ is quasi-abelian and therefore carries a natural exact structure, by a
theorem of Schneiders, \cite[Proposition 4.4]{MR2606234} or the discussion
preceding \cite[Remark 1.1.11]{MR1779315}. (2)\ We quickly check that for
$\alpha,\beta\in\mathfrak{A}$ and $m\in M$ we have%
\[
(\beta\alpha\cdot\varphi)(m)=\varphi(m\cdot\beta\alpha)=\varphi((m\cdot
\beta)\cdot\alpha)=(\alpha\cdot\varphi)(m\cdot\beta)=(\beta\cdot(\alpha
\cdot\varphi))(m)
\]
and that $M^{\vee}=\operatorname*{Hom}(M,\mathbb{T})$ is an algebraic left
$\mathfrak{A}$-module follows literally from $M$ being an algebraic right
$\mathfrak{A}$-module. The continuity of the scalar action on $M^{\vee}$
follows correspondingly from the one of $M$. For the double dualization, note
that the standard reflexivity isomorphism of the underlying LCA group
$\eta:\operatorname*{id}\longrightarrow(-)^{\vee}\circ\left[  (-)^{\vee
}\right]  ^{op}$ has the required properties and additionally preserves the
$\mathfrak{A}$-module action. We refer to \cite{MR0442141} for a proof of the
duality statements for the underlying category $\mathsf{LCA}$.
\end{proof}

\begin{remark}
If $\mathfrak{A}$ happens to be commutative, we have $\mathfrak{A}%
^{op}=\mathfrak{A}$ and then we need not distinguish between left and right
modules over $\mathfrak{A}$. Then Proposition
\ref{Prop_LCAIsQuasiAbelianAndPontryaginDuality} extends to show that
$\mathsf{LCA}_{\mathfrak{A}}$ is an exact category with duality in the sense
of \cite[Definition 2.1]{MR2600285}. This remark applies for example if
$\mathfrak{A}$ is the ring of integers (or any order) in a number field. This
generalizes \cite[Proposition 2.4]{obloc}.
\end{remark}

\begin{lemma}
\label{lemma_CharacterizeAdmissibleMorphisms}A morphism $f:G^{\prime
}\rightarrow G$ in $\mathsf{LCA}_{\mathfrak{A}}$ is

\begin{enumerate}
\item an admissible monic if and only if it is injective and a closed map,

\item an admissible epic if and only if it is surjective and an open map.
\end{enumerate}
\end{lemma}

\begin{proof}
This fact also carries over from $\mathsf{LCA}$, see \cite{MR2329311}.
\end{proof}

We recall a few basic concepts around topological groups.

\begin{definition}
\label{def_TopProps}Suppose $G\in\mathsf{LCA}$ is an LCA group. A subset
$C\subseteq G$ is called \emph{symmetric} if $g\in C$ implies $-g\in C$.

\begin{enumerate}
\item We say $G$ has \emph{no small subgroups} if there exists a neighbourhood
$U$ of the zero element such that $U$ does not contain any non-trivial
subgroups of $G$. Write $\mathsf{LCA}_{nss}$ for the full subcategory of
$\mathsf{LCA}$ of these groups.

\item We say $G$ is \emph{compactly generated} if there exists a symmetric
compact subset $C\subseteq G$ such that $G=\bigcup_{n\geq0}C^{n}$. We write
$\mathsf{LCA}_{cg}$ for the full subcategory of these groups.

\item We call $G$ a \emph{vector group} if it admits an isomorphism
$G\overset{\sim}{\rightarrow}\mathbb{R}^{n}$ for some $n\in\mathbb{Z}_{\geq0}%
$. We write $\mathsf{LCA}_{\mathbb{R}}$ for these groups.
\end{enumerate}
\end{definition}

The structure of these groups is discussed in Moskowitz \cite{MR0215016}. We
recall that every $G\in\mathsf{LCA}_{cg}$ is of the shape $\mathbb{R}%
^{n}\oplus\mathbb{Z}^{m}\oplus C$ for $C$ compact and $n,m$ finite; while
every $G\in\mathsf{LCA}_{nss}$ is of the shape $\mathbb{R}^{n}\oplus
\mathbb{T}^{m}\oplus D$ with $D$ discrete and again $n,m$ finite. This is
proven as Theorem 2.5 resp. Theorem 2.4 \textit{op. cit.} respectively.

\begin{definition}
Extending Definition \ref{def_TopProps} we define full subcategories
$\mathsf{LCA}_{\mathfrak{A},cg}:=\mathsf{LCA}_{\mathfrak{A}}\cap
\mathsf{LCA}_{cg}$, resp. $\mathsf{LCA}_{\mathfrak{A},nss}:=\mathsf{LCA}%
_{\mathfrak{A}}\cap\mathsf{LCA}_{nss}$.
\end{definition}

We shall later show that these subcategories are fully exact subcategories and
in particular may be regarded as exact categories themselves (Remark
\ref{rmk_LCAcg_IsClosedUnderExtensions}). For the moment, we can only treat
them as full additive subcategories.

\begin{lemma}
\label{lemma_CGExchangesWithNSS}The functor $(-)^{\vee}:\mathsf{LCA}%
_{\mathfrak{A}}^{op}\longrightarrow\left.  _{\mathfrak{A}}\mathsf{LCA}\right.
$

\begin{enumerate}
\item sends the full subcategory $\mathsf{LCA}_{\mathfrak{A},cg}^{op}$ to
$\left.  _{nss,\mathfrak{A}}\mathsf{LCA}\right.  $ and conversely;

\item sends the full subcategory $\mathsf{LCA}_{\mathfrak{A},nss}^{op}$ to
$\left.  _{cg,\mathfrak{A}}\mathsf{LCA}\right.  $ and conversely;

\item sends the full subcategory $\mathsf{LCA}_{\mathfrak{A},\mathbb{R}}^{op}
$ to $\left.  _{\mathbb{R},\mathfrak{A}}\mathsf{LCA}\right.  $ and conversely;

\item sends projectives to injectives, and injectives to projectives.
\end{enumerate}
\end{lemma}

\begin{proof}
(1), (2), (3) depend only on qualities of the underlying LCA\ group, so they
follow from \cite{MR0215016}, Corollary 1 to Theorem 2.5 \textit{loc. cit.}
(4) follows since going to the opposite category transitions the universal
property of projectivity into the one for injectivity, and conversely.
\end{proof}

\section{Basic decompositions}

\begin{theorem}
\label{thm_FirstDecomp}Suppose $M\in\mathsf{LCA}_{\mathfrak{A}}$.

\begin{enumerate}
\item Then there exists a clopen right $\mathfrak{A}$-submodule $H$ and an
exact sequence%
\[
H\overset{i}{\hookrightarrow}M\twoheadrightarrow D
\]
in $\mathsf{LCA}_{\mathfrak{A}}$, where $H$ is compactly generated, $D$
discrete, and $i$ an open map.

\item Then there exists a compact right $\mathfrak{A}$-submodule $C$ and an
exact sequence%
\[
C\hookrightarrow M\overset{q}{\twoheadrightarrow}N
\]
in $\mathsf{LCA}_{\mathfrak{A}}$, where $N$ has no small subgroups and $q$ is
an open and closed map.
\end{enumerate}

The analogous claims for left modules hold for $M\in\left.  _{\mathfrak{A}%
}\mathsf{LCA}\right.  $.
\end{theorem}

\begin{proof}
(1) The first claim is proven fairly analogously to \cite[Lemma 2.14]{obloc}.
We work with right modules. If $M=0$, take $H:=0$. Otherwise, let $m\neq0$ be
any element of $M$. As discussed in \S \ref{subsect_Basics}, $\mathfrak{A}%
\simeq\mathbb{Z}^{n}$ as abelian groups, so we may pick generators $b_{i}$ and
write $\mathfrak{A}=\mathbb{Z}\left\langle b_{1},\ldots,b_{n}\right\rangle $.
Next, since $M$ is locally compact, we find compact neighbourhoods of both
$0\in M$ and $m$, so let $U_{0}\subseteq M$ be their union. Define%
\[
U_{1}:=\bigcup_{i=1}^{n}U_{0}\cdot b_{i}%
\]
and the $U_{2}:=U_{1}\cup(-U_{1})$. We observe that (1) $U_{2}$ is symmetric,
(2) $U_{2}$ is compact since it is a finite union of compacta, (3) $U_{2}$ is
open since it contains an open neighbourhood of $0\in M$. Define
$H:=\bigcup_{m\geq1}U_{2}^{m}$. Then $H$ is compactly generated, and open
since it is a union of opens. Thus, $H$ is even clopen in $M$. We claim that
$H$ is a right $\mathfrak{A}$-module: Suppose $h\in H$ and $\alpha
\in\mathfrak{A}$. Then $h\in U_{2}^{m}$ for some $m$, so $h$ is a finite sum
of terms of the shape $\pm u\cdot b_{i}$ with $u\in U_{0}$. For each such term
we have%
\[
(\pm u\cdot b_{i})\cdot\alpha=\pm u\cdot(b_{i}\cdot\alpha)=\pm u\cdot
\sum_{j=1}^{n}c_{j}b_{j}\qquad\text{with}\qquad c_{i}\in\mathbb{Z}%
\]
since the $b_{i}$ form a $\mathbb{Z}$-basis of $\mathfrak{A}$. Unravelling
each $c_{j}\in\mathbb{Z}$ as a finite sum of \textquotedblleft$\pm
1$\textquotedblright, we note that this expression again is a finite sum of
terms of the form $\pm u\cdot b_{j}$, i.e. lies in $H$. Thus, $H\in
\mathsf{LCA}_{\mathfrak{A}}$ and the quotient $D:=M/H$ is discrete since $H$
was open. The claim for left modules can be proven symmetrically. (2) We apply
(1) to the Pontryagin dual, giving us an exact sequence%
\[
D^{\vee}\hookrightarrow M\twoheadrightarrow H^{\vee}%
\]
and the dual $D^{\vee}$ of a discrete module is compact, giving $C$, and the
quotient map $q$ is open. By \cite[Proposition 11]{MR0442141} it follows that
$q$ is also a closed map. Finally, the dual of a compactly generated module
has no small subgroups, see Lemma \ref{lemma_CGExchangesWithNSS}.
\end{proof}

\section{\label{sect_VectorModules}Vector $\mathfrak{A}$-modules}

Let $A$ be a finite-dimensional semisimple $\mathbb{Q}$-algebra and
$\mathfrak{A}\subset A$ an order.

\begin{lemma}
\label{lemma_AFlatOverOrder}$A$ is a left (and right) flat $\mathfrak{A}$-algebra.
\end{lemma}

\begin{proof}
We have $\mathbb{Q}\mathfrak{A}=A$, so we can also write $A=\underrightarrow
{\operatorname*{colim}}\frac{1}{n}\mathfrak{A}$, where the colimit runs over
all integers, partially ordered by divisibility. This presents $A$ as a
filtering colimit of (clearly flat!) $\mathfrak{A}$-algebras since $\frac
{1}{n}\mathfrak{A}\cong\mathfrak{A}$ since there cannot be any non-trivial
torsion inside a $\mathbb{Q}$-vector space. However, a filtering colimit of
flat algebras is still flat. For this, see \cite[(4.4), Proposition]%
{MR1653294}; it just reduces to tensor products commuting with filtering
colimits. Since multiplication with $\frac{1}{n}$ is central, this argument
works both for the left and the right $\mathfrak{A}$-module structure.
\end{proof}

\begin{lemma}
\label{lemma_VectorModulesAreAlgebraicallyInjectiveAndProjective}Suppose
$M\in\mathsf{LCA}_{\mathfrak{A},\mathbb{R}}$. Then as an algebraic right
$\mathfrak{A}$-module, $M$ is injective and projective; i.e., $M$ is injective
and projective in $\mathsf{Mod}_{\mathfrak{A}}$. The corresponding statement
for left $\mathfrak{A}$-modules is also true.
\end{lemma}

\begin{proof}
(Step 1) Firstly, we show injectivity for right $\mathfrak{A}$-modules. We
define a map of right $\mathfrak{A}$-modules%
\[
\Phi:M\longrightarrow\operatorname*{Hom}\nolimits_{A}(A,M)\text{,}\qquad
m\longmapsto(a\mapsto ma)\text{,}%
\]
where $\operatorname*{Hom}\nolimits_{A}(A,M)$ carries the right $\mathfrak{A}%
$-module structure given by $(f\cdot r)(a):=f(r\cdot a)$ for $r\in
\mathfrak{A}$. It is easy to check that $\Phi$ respects the right
$\mathfrak{A}$-module structure and the map is an isomorphism since $f\mapsto
f(1_{A})$ is an inverse map. Now, $M$ is an injective right $\mathfrak{A}%
$-module if and only if the functor $N\mapsto\operatorname*{Hom}%
\nolimits_{\mathfrak{A}}(N,M)$ is exact. We express this functor in a
different way by using the functorial isomorphisms%
\[
\operatorname*{Hom}\nolimits_{\mathfrak{A}}(N,M)\underset{\Phi}{\cong%
}\operatorname*{Hom}\nolimits_{\mathfrak{A}}(N,\operatorname*{Hom}%
\nolimits_{A}(A,M))\cong\operatorname*{Hom}\nolimits_{A}(N\otimes
_{\mathfrak{A}}A,M)\text{,}%
\]
stemming from $\Phi$ and the Hom-tensor adjunction. Since $A$ is a flat left
$\mathfrak{A}$-module (Lemma \ref{lemma_AFlatOverOrder}), the functor
$N\mapsto N\otimes_{\mathfrak{A}}A$ is exact. Moreover, the functor
$\operatorname*{Hom}\nolimits_{A}(-,M)$ is exact since $A$ is semisimple, so
every right $A$-module is injective (see \S \ref{subsect_Basics}). Hence, the
composed functor is also exact and it follows that $M$ is an injective
algebraic right $\mathfrak{A}$-module. (Step 2) The argument to show
injectivity for algebraic left $\mathfrak{A}$-modules is analogous. We only
need to replace the isomorphism $\Phi$ by%
\[
\Phi^{\prime}:M\longrightarrow\operatorname*{Hom}\nolimits_{A}(A,M)\text{,}%
\qquad m\longmapsto(a\mapsto am)
\]
and the left $\mathfrak{A}$-module structure $(r\cdot f)(a):=f(a\cdot r)$.
(Step 3) Next, we show projectivity for right $\mathfrak{A}$-modules $M$. To
this end, note that the Pontryagin dual $M^{\vee}$ is a left $\mathfrak{A}%
$-module. As the Pontryagin dual of a vector group is again a vector group, we
get $M^{\vee}\in\left.  _{\mathbb{R},\mathfrak{A}}\mathsf{LCA}\right.  $. By
Step 2 it follows that $M^{\vee}$ is an injective algebraic left
$\mathfrak{A}$-module. We take Pontryagin duals again, transforming the
universal property of injectivity into the universal property for
projectivity. Hence, $M^{\vee\vee}\cong M$ is a projective algebraic right
$\mathfrak{A}$-module (this generalizes the corresponding argument of
Moskowitz, \cite[Theorem 3.1]{MR0215016}). (Step 4) Projectivity for left
$\mathfrak{A}$-modules now follows analogously from Step 1: The dual $M^{\vee
}$ is a topological right $\mathfrak{A}$-module, injective as an algebraic
right $\mathfrak{A}$-module by Step 1, and by double dualization $M^{\vee\vee
}\cong M$ is projective as an algebraic left $\mathfrak{A}$-module.
\end{proof}

\begin{lemma}
\label{Lemma_CompactPartOfCG}Every object $M\in\mathsf{LCA}_{\mathfrak{A},cg}$
sits in a canonically determined exact sequence%
\[
C\hookrightarrow M\twoheadrightarrow W
\]
such that $C$ is a compact submodule, and as an LCA\ group $W$ is isomorphic
to $\mathbb{R}^{n}\oplus\mathbb{Z}^{m}$ for suitable $n,m\in\mathbb{Z}_{\geq
0}$.
\end{lemma}

\begin{proof}
Firstly, we shall work with $M$ viewed as an object of the category
$\mathsf{LCA}$ alone. Being compactly generated, there exists an isomorphism
$M\simeq\mathbb{R}^{n}\oplus\mathbb{Z}^{m}\oplus C$ in $\mathsf{LCA}$ with $C$
compact, \cite[Theorem 2.5]{MR0215016}. For every $\alpha\in\mathfrak{A}$ the
right multiplication map is continuous, giving the solid arrows in the
commutative diagram%
\[%
\xymatrix{
C \ar@{^{(}->}[d] \ar@{-->}[r] & C \ar@{^{(}->}[d] \\
M \ar[r]^{\cdot\alpha} \ar@{->>}[d] & M \ar@{->>}[d] \\
M/C & M/C\text{.}
}%
\]
As the image of the upper left compactum $C$ in the lower right quotient $M/C
$ is compact, but $M/C\simeq\mathbb{R}^{n}\oplus\mathbb{Z}^{m}$ has no
non-trivial compact subgroups, the universal property of kernels exhibits the
dashed horizontal arrow. As this holds for all $\alpha\in\mathfrak{A}$, it
follows that $C$ is (1) closed in $M$ since $C$ is compact, and (2) closed
under the right action of $\mathfrak{A}$. In particular, $C$ is a closed
submodule. This also implies that the quotient $W:=M/C$ also makes sense in
$\mathsf{LCA}_{\mathfrak{A}}$, with the same underlying LCA group, proving our claim.
\end{proof}

\begin{lemma}
\label{Lemma_RZOrRTHulls}Let $V\in\mathsf{LCA}_{\mathfrak{A}}$ be a vector
$\mathfrak{A}$-module. Suppose $M\in\mathsf{LCA}_{\mathfrak{A}}$. If the
underlying LCA group of $M$ is

\begin{enumerate}
\item isomorphic to $\mathbb{R}^{n}\oplus\mathbb{Z}^{m}$, then there exists an
admissible monic $M\hookrightarrow\tilde{M}$, where $\tilde{M}$ is a vector
$\mathfrak{A}$-module and any morphism $f$ as in%
\[%
\xymatrix{
M \ar@{^{(}->}[rr] \ar[dr]_{f} & & \tilde{M} \ar@{-->}[dl]^{f^{\prime}} \\
& V
}%
\]
admits a lift $f^{\prime}$. The cokernel of $M\hookrightarrow\tilde{M}$ has
underlying LCA group $\mathbb{T}^{m}$.

\item isomorphic to $\mathbb{R}^{n}\oplus\mathbb{T}^{m}$, then there exists an
admissible epic $\tilde{M}\twoheadrightarrow M$, where $\tilde{M}$ is a vector
$\mathfrak{A}$-module and any morphism $f$ as in%
\[%
\xymatrix{
& V \ar@{-->}[dl]_{f^{\prime}} \ar[dr]^{f} \\
\tilde{M} \ar@{->>}[rr] & & M
}%
\]
admits a lift $f^{\prime}$. The kernel of $\tilde{M}\twoheadrightarrow M$ has
underlying LCA group $\mathbb{Z}^{m}$.
\end{enumerate}
\end{lemma}

We prove this by using covering space theory. This might appear like overkill,
but it gives a quick and clean solution to the problem.

\begin{proof}
We prove (2) for right $\mathfrak{A}$-modules: We first work in the category
$\mathsf{LCA}$ alone. Since $M$ is connected, locally path-connected and
semi-locally simply connected, $M$ admits a universal covering space
$\tilde{M}$, see e.g. \cite[(6.7) Theorem]{MR643101}. Take $0\in M$ to regard
it as a pointed space. The fundamental group $\pi_{1}(M,0_{M})\simeq
\mathbb{Z}^{m}$ acts properly discontinuously on $\tilde{M}$ via deck
transformations. In particular, the covering map $q:\tilde{M}\rightarrow M$ is
a topological quotient map. Now, we shall use the lifting property of
continuous maps along covering maps, \cite[(6.1) Theorem]{MR643101}. Firstly,
lift addition to $\tilde{M}$ by lifting the map%
\[
\tilde{M}\times\tilde{M}\overset{q\times q}{\longrightarrow}M\times
M\overset{+}{\longrightarrow}M
\]
to $\tilde{M}$. Similarly, lift the negation map. All such lifts are unique,
\cite[(5.1) Theorem]{MR643101}. A standard computation confirms that this
equips $\tilde{M}$ with the structure of a topological abelian group, see for
example \cite[(6.11) Theorem]{MR643101}. Moreover, $\tilde{M}$ is Hausdorff
and locally compact since the covering map $q$ is a local homeomorphism (or
just use that we know that $\tilde{M}\simeq\mathbb{R}^{n+m}$). It follows that
$\tilde{M}\in\mathsf{LCA}$ and $q$, being surjective and open, is an
admissible epic. Next, for each $\alpha\in\mathfrak{A}$ lift the scalar action
of $M$ by lifting%
\[
\tilde{M}\overset{q}{\longrightarrow}M\overset{\cdot\alpha}{\longrightarrow
}M\text{,}%
\]
again using \cite[(6.1) Theorem]{MR643101}. Since the lifts of these maps are
again continuous, we obtain $\tilde{M}\in\mathsf{LCA}_{\mathfrak{A}}$, and
since these were lifts of the scalar action on $M$, the morphism $q$ is even a
morphism in $\mathsf{LCA}_{\mathfrak{A}}$. Thus, we have constructed an
admissible epic $\tilde{M}\twoheadrightarrow M$ in $\mathsf{LCA}%
_{\mathfrak{A}}$, and now, finally, using that $\tilde{M}\simeq\mathbb{R}%
^{n+m}$, we see that $\tilde{M}$ is indeed a vector $\mathfrak{A}$-module. The
lifting property from $f$ to $f^{\prime}$ again is nothing but the topological
lifting, \cite[(6.1) Theorem]{MR643101}. We quickly check that $f^{\prime}$ is
a right $\mathfrak{A}$-module homomorphism: Since $q$ is a right
$\mathfrak{A}$-module homomorphism, we compute%
\[
q(f^{\prime}(v)\cdot_{\tilde{M}}\alpha)=q(f^{\prime}(v))\cdot_{M}%
\alpha=f(v)\cdot_{M}\alpha=f(v\cdot_{V}\alpha)\qquad\qquad\text{for all }v\in
V\text{,}%
\]
since $f$ is a right $\mathfrak{A}$-module homomorphism, but of course
$qf^{\prime}=f$ also implies%
\[
q(f^{\prime}(v\cdot_{V}\alpha))=f(v\cdot_{V}\alpha)
\]
and therefore $f^{\prime}(-)\cdot_{\tilde{M}}\alpha$ and $f^{\prime}%
(-\cdot_{V}\alpha)$ lift the same map, so by the uniqueness of lifts they must
agree. But this is just the statement that $f^{\prime}$ is a module
homomorphism. It is clear that a symmetric argument works for left
$\mathfrak{A}$-modules. Now, we prove (1) for right modules by observing that
the Pontryagin dual of $\mathbb{R}^{n}\oplus\mathbb{Z}^{m}$ is $\mathbb{R}%
^{n}\oplus\mathbb{T}^{m}$, so by part (2) we get an epic $\tilde
{M}\twoheadrightarrow M^{\vee}$ in $\left.  _{\mathfrak{A}}\mathsf{LCA}%
\right.  $. Dualizing again, this becomes an admissible monic
$M\hookrightarrow\tilde{M}^{\vee}$, but since the dual of a vector
$\mathfrak{A}$-module is again a vector $\mathfrak{A}$-module, this proves the
claim. Analogously for left modules $M$.
\end{proof}

\begin{remark}
The previous lemma can also be proved without using any covering spaces.
Nonetheless, we would hope that this type of argument could be extended to a
much wider class of topological groups. Even for the purposes of this text it
would be nice to handle all groups of the shape $G=\mathbb{R}^{n}%
\oplus\mathbb{T}^{m}\oplus\mathbb{Z}^{d}$. The first problem is that classical
covering space theory is only developed for connected base spaces. However,
regarding applications to topological groups, a formalism for covering spaces
of non-connected groups was developed by Taylor \cite{MR0087028} and extended
by Brown and Mucuk \cite{MR1253285}. The first most drastic difference of this
situation in comparison to the connected case is that the existence of a
universal covering space hinges not just on local constraints (like local
path-connectedness), but on a global topological obstruction, a class in
$H^{3}(\pi_{0}(G),\pi_{1}(G,e))$, for $e$ the neutral element of the group. In
the case at hand, this causes no concern, this obstruction class vanishes if
there is a section for the map $G\twoheadrightarrow\pi_{0}(G)$; use
\cite[(7.1)]{MR0087028} or \cite[Theorem 3 and Proposition 4]{MR0419643}%
.\ Such a section exists in our situation,%
\[
\mathbb{R}^{n}\oplus\mathbb{T}^{m}\oplus\mathbb{Z}^{d}\leftrightarrows
\mathbb{Z}^{d}\text{,}%
\]
but as we shall see in Example \ref{example_CyclicGroupCannotSplitOffTorus}
such sections cannot be expected to exist as $\mathfrak{A}$-module
homomorphisms in the case of non-hereditary orders. Besides this issue, the
lifting properties of module structures along such coverings appear not to be
completely understood. We thank R. Brown for interesting correspondence.
\end{remark}

\begin{remark}
Extending the previous remark, one could of course dream to work with even
more general LCA groups $G$. Besides the complications from lack of
connectedness, there are also issues around local path-connectedness. A
theorem of Dixmier shows that $G\in\mathsf{LCA}$ is locally path-connected if
and only if $G\cong\mathbb{R}^{n}\oplus C\oplus D$ with $D$ discrete and $C$
compact and path-connected, \cite[(8.38)]{MR637201}. This restricts the
consideration to compact groups, and for those Dixmier showed that a compact
$G\in\mathsf{LCA}$ is path-connected if and only if its dual $G^{\vee}$, which
is discrete, satisfies $\operatorname*{Ext}_{\mathbb{Z}}^{1}(G^{\vee
},\mathbb{Z})=0,$ \cite[(8.25) Remarks, (a)]{MR637201}. Such abelian groups
$G^{\vee}$ are called `Whitehead groups'\footnote{Of course some people just
call $K_{1}(R)$ the Whitehead group. What these groups have in common is that
they are both abelian groups, but that's about it.}. Free abelian groups
clearly have this property, so $G=\mathbb{T}^{\omega}$ is path-connected for
all cardinals $\omega$. Any other path-connected example $G$ would yield a
non-free Whitehead group $G^{\vee}$. It was shown by Shelah that it is
undecidable in ZFC set theory whether such groups exist or not. If we restrict
to second countable LCA groups, the $\mathbb{T}^{\omega}$ are the only compact
groups which are path-connected. We are now remote from the subject of this
paper, but we feel it might be worth to see that the assumptions of Lemma
\ref{Lemma_RZOrRTHulls} are clearly much more restrictive than required, yet
formulating a good generalization will be delicate.
\end{remark}

\begin{lemma}
\label{Lemma_LiftFromLattice}Suppose $G^{\prime}\in\mathsf{LCA}_{\mathfrak{A}%
}$ such that its underlying LCA group is isomorphic to $\mathbb{R}^{n}%
\oplus\mathbb{Z}^{m}$ for suitable $n,m\in\mathbb{Z}_{\geq0}$. Suppose $G$ is
a second such object, for possibly different values of $n$ and $m$. Suppose we
are given the solid arrows in the diagram%
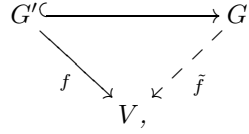
\[%
\xymatrix{
G^{\prime} \ar@{^{(}->}[rr] \ar[dr]_{f} & & G \ar@{-->}[dl]^{\tilde{f}} \\
& V\text{,}
}%
\]
where $V$ is a vector $\mathfrak{A}$-module. Then a lift $\tilde{f}$ exists.
\end{lemma}

\begin{proof}
We follow the strategy of Moskowitz \cite{MR0215016}. (Step 1) We work with
right modules. Suppose $G$ is a vector $\mathfrak{A}$-module. Then $G^{\prime
}$ is a closed subgroup of a vector group $\mathbb{R}^{n}$ on the level of
$\mathsf{LCA}$. It follows that there exists a basis $b_{1},\ldots,b_{n}$ of
$G$ as a real vector space such that%
\[
G^{\prime}=\mathbb{Z}\left\langle b_{1},\ldots,b_{r}\right\rangle
\oplus\mathbb{R}\left\langle b_{r+1},\ldots,b_{n}\right\rangle
\]
for some $0\leq r\leq n$, cf. \cite[Theorem 6]{MR0442141}. As $f$ is an
abelian group morphism, we have%
\[
f\left(  \sum a_{i}b_{i}\right)  =\sum a_{i}f(b_{i})
\]
for all $a_{1},\ldots,a_{r}\in\mathbb{Z}$ and $a_{r+1},\ldots,a_{n}%
\in\mathbb{Q}$. By continuity, the same holds if instead $a_{r+1},\ldots
,a_{n}\in\mathbb{R}$. Thus, if we define $\tilde{f}(\sum a_{i}b_{i}%
):=a_{i}f(b_{i})$ for all $a_{i}\in\mathbb{R}$, this indeed extends $f$. Being
a linear map on a finite-dimensional vector space, $\tilde{f}$ is necessarily
continuous. Moreover, it respects the $\mathfrak{A} $-module structure:\ As
$\tilde{f}$ is $\mathbb{R}$-linear, we can check this on the basis elements
$b_{i}$, however they all lie in the subgroup $G^{\prime}$, and so this
follows since $f$ respects the $\mathfrak{A}$-module structure. This proves
our claim. (Step 2) Now, let $G$ be arbitrary as in our claim. By Lemma
\ref{Lemma_RZOrRTHulls}, (1), we get%
\[
G^{\prime}\hookrightarrow G\overset{i}{\hookrightarrow}\tilde{G}\text{,}%
\]
where $\tilde{G}$ is a vector $\mathfrak{A}$-module. By Step 1 we can lift
$f:G^{\prime}\rightarrow V$ to a map $\tilde{G}\rightarrow V$, and restricting
to $G$, the map $f^{\prime}\circ i$ proves our claim. Finally, it is clear
that all these steps also work for left modules.
\end{proof}

\begin{lemma}
\label{Lemma_LiftCG}Suppose we are given the solid arrows in the diagram%
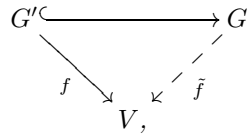
\[%
\xymatrix{
G^{\prime} \ar@{^{(}->}[rr] \ar[dr]_{f} & & G \ar@{-->}[dl]^{\tilde{f}} \\
& V\text{,}
}%
\]
where $V$ is a vector $\mathfrak{A}$-module and $G^{\prime},G\in
\mathsf{LCA}_{\mathfrak{A},cg}$. Then a lift $\tilde{f}$ exists.
\end{lemma}

\begin{proof}
We right away work in $\mathsf{LCA}_{\mathfrak{A}}$. Apply Lemma
\ref{Lemma_CompactPartOfCG} to both $G^{\prime}$ and $G$. We obtain a
commutative diagram%
\[%
\xymatrix{
C^{\prime} \ar@{^{(}->}[d]  & C \ar@{^{(}->}[d] \\
G^{\prime} \ar@{^{(}->}[r]^{i} \ar@{->>}[d]_{\pi^{\prime}} & G \ar@
{->>}[d]^{\pi} \\
W^{\prime} & W\text{.}
}%
\]
The map from $C^{\prime}$ on the upper left to $W$ on the lower right maps the
compactum $C^{\prime}$ to a compact subgroup of $W$. But $W$ is isomorphic to
$\mathbb{R}^{n}\oplus\mathbb{Z}^{m}$ as an LCA group, so this morphism must be
zero. The universal property of kernels yields the lift $i^{\sharp}$, and then
as a result the morphism $j$ gets induced to the quotients:%
\begin{equation}%
\xymatrix{
C^{\prime} \ar@{^{(}->}[d] \ar@{-->}[r]^{i^{\sharp}} & C \ar@{^{(}->}[d] \\
G^{\prime} \ar@{^{(}->}[r]^{i} \ar@{->>}[d]_{\pi^{\prime}} & G \ar@
{->>}[d]^{\pi} \\
W^{\prime} \ar@{-->}[r]_{j} & W\text{.}
}%
\label{l_xA1}%
\end{equation}
Observe that both $\pi,\pi^{\prime}$ are open morphisms (as they are
admissible epics), but they are additionally closed morphisms since
$C^{\prime},C$ are compact by \cite[Proposition 11]{MR0442141} (beware: not
every admissible epic in $\mathsf{LCA}$ is a closed map, see Example
\ref{example_NotEveryQuotientMapIsClosed}, neither does all open sets being
clopen imply that all open maps are closed; not at all).\newline(Step 1) We
claim that $j$ is injective: Suppose $w^{\prime}\in W^{\prime}$ and
$j(w^{\prime})=0$. Since $\pi^{\prime}$ is surjective, we find some
$g^{\prime}\in G^{\prime}$ such that $w^{\prime}=\pi^{\prime}(g^{\prime})$ and
thus $j\pi^{\prime}(g^{\prime})=0$, i.e. $\pi i(g^{\prime})=0$. But $i$ is
injective, so we may just as well regard $G^{\prime}$ as a subset of $G$. We
see that the kernel of $\pi i$ is just $G^{\prime}\cap C$; it is those
elements in $C$ which come from the subset $G^{\prime}$. However, $C$ is
compact, so $G^{\prime}\cap C$ is a compact subgroup of $G^{\prime}$. Its
image $\pi^{\prime}(G^{\prime}\cap C)\subseteq W^{\prime}$ must also be
compact, but since $W^{\prime}$ has no compact subgroups, we conclude that
$G^{\prime}\cap C\subseteq C^{\prime}$ (a different way to say this:
$C^{\prime}$ is the maximal compact subgroup of $G^{\prime}$). Thus, our
$g^{\prime}$ with $\pi i(g^{\prime})=0$ must satisfy $g^{\prime}\in C^{\prime
}$. Since $g^{\prime}$ was a preimage of $w^{\prime}$ in $G^{\prime}$, it
follows that $w^{\prime}=0$.\newline(Step 2) We claim that $j$ is a closed
map. Let $T\subseteq W^{\prime}$ be a closed set. As $\pi^{\prime}$ is
continuous, the preimage%
\[
(\pi^{\prime})^{-1}(T)=\{g^{\prime}\in G^{\prime}\mid\pi^{\prime}(g^{\prime
})\in T\}
\]
is closed in $G^{\prime}$. Now, note that $i$ is a closed map, and $\pi$ is
also a closed map (as we had recalled above), so $(\pi\circ i)(\pi^{\prime
})^{-1}(T)$ is closed in $W$. However, $\pi i=j\pi^{\prime}$, i.e.%
\[
(\pi\circ i)\left(  (\pi^{\prime})^{-1}(T)\right)  =(j\circ\pi^{\prime
})\left(  (\pi^{\prime})^{-1}(T)\right)
\]
and $\pi^{\prime}(\pi^{\prime})^{-1}(T)=T$ (we always have `$\subseteq$', and
since $\pi^{\prime}$ is surjective, we even have equality). Thus, this set
agrees with $j(T)$. As $T$ was an arbitrary closed subset of $W^{\prime}$, it
follows that $j$ maps closed sets to closed sets.\newline(Step 3) However,
combining the previous steps, it follows that $j$ is an admissible monic. We
obtain the solid arrows of the diagram%
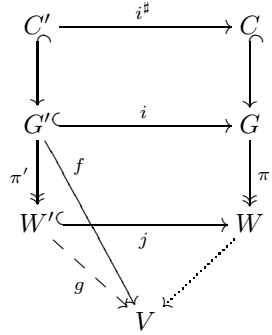
\[%
\xymatrix{
C^{\prime} \ar@{^{(}->}[d] \ar[rr]^{i^{\sharp}} & & C \ar@{^{(}->}[d] \\
G^{\prime} \ar@{^{(}->}[rr]^{i} \ar[ddr]^<<<<<<{f} \ar@{->>}[d]_{\pi^{\prime}}
& & G \ar@{->>}[d]^{\pi} \\
W^{\prime} \ar@{^{(}->}[rr]_{j} \ar@{-->}[dr]_{g} & & W \ar@{..>}[dl] \\
& V
}%
\]
and since $V$ has no compact subgroups, but $C^{\prime}$ has compact image
under $f$, we get the dashed lift $g$ by the universal property of cokernels.
The morphism $j$ has all the properties required to apply Lemma
\ref{Lemma_LiftFromLattice}, so we get the lift depicted by a dotted arrow.
Composing with $\pi$, this yields the required lift and finishes the proof.
\end{proof}

\begin{remark}
At the beginning of the proof we have shown that $C^{\prime}$ maps to zero in
$W$, so $C^{\prime}\subseteq C$ and trivially $C^{\prime}\subseteq G^{\prime
}\cap C$, and in Step 1 we saw that $G^{\prime}\cap C\subseteq C^{\prime}$. It
follows that the upper square in Diagram \ref{l_xA1} is a pullback diagram. It
remains a pullback in the category of abelian groups and then applying the
Snake Lemma yields a different proof that $j$ is injective. We cannot use the
Snake Lemma right away in $\mathsf{LCA}$ itself, since at this step of the
proof we do not yet know that $j$ is an admissible morphism, compare
\cite[Corollary 8.13]{MR2606234}.
\end{remark}

\begin{example}
\label{example_NotEveryQuotientMapIsClosed}The parabola $\{(x,y)\in
\mathbb{R}^{2}\mid xy=1\}$ is closed, but gets mapped to the open set
$\mathbb{R}\setminus\{0\}$ under the projection to (either) factor. These
projection maps are therefore open, but not closed.
\end{example}

We recall the following elementary facts from topology.

\begin{lemma}
\label{lemma_CheckNearZero}If $f:G^{\prime}\rightarrow G$ is an abelian group
homomorphism of topological groups (not necessarily continuous), then $f$ is
continuous if and only if it is continuous in a neighbourhood of the neutral
element of $G^{\prime}$.
\end{lemma}

\begin{lemma}
\label{lemma_SumOpenWithArbitraryIsOpen}Suppose $G$ is a topological group. If
$H$ is an open subgroup and $E$ an arbitrary subgroup, then $H+E$ is an open subgroup.
\end{lemma}

\begin{proof}
Just add an open neighbourhood of zero inside $H$ to an arbitrary element of
$H+E$ to see this.
\end{proof}

\begin{theorem}
\label{thm_VectorModulesAreInjectiveAndBijective}Vector $\mathfrak{A}$-modules
are injective and projective objects in $\mathsf{LCA}_{\mathfrak{A}} $, as
well as in $\left.  _{\mathfrak{A}}\mathsf{LCA}\right.  $.
\end{theorem}

\begin{proof}
We demonstrate injectivity in $\mathsf{LCA}_{\mathfrak{A}}$. This means we
have to show the following: Suppose we are given the solid arrows in the
diagram%
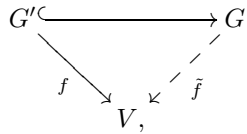
\[%
\xymatrix{
G^{\prime} \ar@{^{(}->}[rr] \ar[dr]_{f} & & G \ar@{-->}[dl]^{\tilde{f}} \\
& V\text{,}
}%
\]
where $V$ is a vector $\mathfrak{A}$-module and $G^{\prime},G\in
\mathsf{LCA}_{\mathfrak{A}}$. Then a lift $\tilde{f}$ exists. We use
essentially the same reduction as Moskowitz \cite[Proposition 3.4]{MR0215016},
but in order to provide a complete proof, we repeat the argument: By Theorem
\ref{thm_FirstDecomp} we find a clopen compactly generated $H\hookrightarrow
G$ with $H\in\mathsf{LCA}_{\mathfrak{A}}$. We obtain the commutative diagram%
\[%
\xymatrix{
G^{\prime} \cap H \ar@{^{(}->}[d] \ar@{^{(}->}[r]^{s} & H \ar@{^{(}->}[d] \\
G^{\prime} \ar@{^{(}->}[r]_{i} & G
}%
\]
and $s$ is an admissible monic. Since $H$ is compactly generated, so is its
closed subgroup $G^{\prime}\cap H$ by \cite[Theorem 2.6]{MR0215016}. We lift
the restriction $f\mid_{G^{\prime}\cap H}:G^{\prime}\cap H\rightarrow V$ to
$H$ using Lemma \ref{Lemma_LiftCG}, call this $f^{\prime} $. Thus, we have
constructed a morphism $(f,f^{\prime}):G^{\prime}\oplus H\rightarrow V$. Since
$f$ and $f^{\prime}$ agree on $G^{\prime}\cap H$, we see that $(f,f^{\prime})$
restricts to the zero map on this intersection, so by the universal property
of cokernels applied to the exact sequence%
\[
G^{\prime}\cap H\overset{\triangle}{\hookrightarrow}G^{\prime}\oplus
H\overset{+}{\twoheadrightarrow}G^{\prime}+H\text{,}%
\]
where $\triangle x:=(x,-x)$, the map descends to a morphism $f^{\prime\prime
}:G^{\prime}+H\rightarrow V$. We note that on $G^{\prime}$ this map agrees
with $f$, so $f^{\prime\prime}$ is a partial lift. Next, since $H$ is open in
$G$, the sum $G^{\prime}+H$ is also open in $G$ (Lemma
\ref{lemma_SumOpenWithArbitraryIsOpen}). In particular, it is a closed
submodule, so we arrive at the solid arrows in a diagram%
\[%
\xymatrix{
G^{\prime} + H \ar[dr]_{f^{\prime\prime}} \ar@{^{(}->}[r] & G \ar@
{-->}[d] \ar@{->>}[r] & D \\
& V\text{,}
}%
\]
where $D$ is discrete since $G^{\prime}+H$ was additionally open. Since $V$ is
injective as an algebraic right $\mathfrak{A}$-module by Lemma
\ref{lemma_VectorModulesAreAlgebraicallyInjectiveAndProjective}, we get the
dashed arrow, as an algebraic $\mathfrak{A}$-module morphism. We need to check
that it is continuous. However, for set-theoretic maps of topological groups
it suffices to check continuity in a neighbourhood of the neutral element
(Lemma \ref{lemma_CheckNearZero}) and since $G^{\prime}+H$ is open, it
suffices to check it on $G^{\prime}+H$, but there $f$ agrees with
$f^{\prime\prime}$, which we know to be continuous. Thus, the dashed arrow
defines a morphism in $\mathsf{LCA}_{\mathfrak{A}}$. The argument for left
modules is analogous. Finally, projectivity follows by Pontryagin duality, see
Lemma \ref{lemma_CGExchangesWithNSS}.
\end{proof}

\section{\label{sect_SplitOffVectorSummand}Splitting off vector summands}

\begin{lemma}
\label{Lemma_KeySplittingLemma}Let $M\in\mathsf{LCA}_{\mathfrak{A}}$ be given.
Suppose for the underlying LCA group there exists a direct sum decomposition%
\[
M\cong V\oplus D\qquad\text{in}\qquad\mathsf{LCA}%
\]
for $V,D\in\mathsf{LCA}$, where $D$ is discrete.

\begin{enumerate}
\item If $V$ is a vector group, then it additionally has a canonical structure
as a topological right vector $\mathfrak{A}$-module. Hence, $V\in
\mathsf{LCA}_{\mathfrak{A}}$, and moreover there exists a direct sum
decomposition%
\begin{equation}
M\cong V\oplus D\label{l_xA2}%
\end{equation}
also in $\mathsf{LCA}_{\mathfrak{A}}$, where $D\in\mathsf{LCA}_{\mathfrak{A}}$
carries the discrete topology.

\item If $\mathfrak{A}$ is a hereditary order and $V\simeq\mathbb{R}^{n}%
\oplus\mathbb{T}^{m}$, then $V$ additionally has a canonical structure as an
object in $\mathsf{LCA}_{\mathfrak{A}}$. Moreover, the same direct sum
decomposition of Equation \ref{l_xA2} holds.
\end{enumerate}
\end{lemma}

This is the analogue of \cite[Lemma 2.11]{obloc}. Note that we only obtain a
weaker result if $\mathfrak{A}$ fails to be hereditary. Indeed, the
impossibility to split off torus summands in $\mathsf{LCA}_{\mathfrak{A}}$ is
a key reason why the theory is more complicated than in \cite{obloc}.

\begin{proof}
(Claim 1) We note that $V\subseteq M$ can alternatively be described as the
connected component of the neutral element since $V$ is connected and $D$
discrete. As the action of $\mathfrak{A}$ is by continuous maps, the image of
a connected set is connected, so it necessarily maps $V$ into itself. It
follows that $V$ is an algebraic right $\mathfrak{A}$-submodule. Being a
subgroup, the action of $\mathfrak{A}$ is clearly still continuous, so it is
also a topological right $\mathfrak{A}$-module, i.e. $V\in\mathsf{LCA}%
_{\mathfrak{A}}$. In $\mathsf{LCA}$, we know that the quotient $M/V$ is
discrete, so $V$ is open (it is the preimage of the open $\{0\}$ under the
quotient map) in $M$ and thus clopen. Hence, we get a subobject%
\[
V\hookrightarrow M
\]
in $\mathsf{LCA}_{\mathfrak{A}}$ and the inclusion is an admissible monic.
From Lemma \ref{lemma_VectorModulesAreAlgebraicallyInjectiveAndProjective} we
learn that $V$ is algebraically an injective right $\mathfrak{A}$-module.
Thus, as algebraic $\mathfrak{A}$-module maps, the injection%
\[
V\hookrightarrow M\qquad\qquad\text{(in }\mathsf{Mod}_{\mathfrak{A}}\text{)}%
\]
splits and we get a direct sum decomposition $M\simeq V\oplus D^{\prime}$ for
some complement $D^{\prime}$ in $\mathsf{Mod}_{\mathfrak{A}}$. Quotienting out
$H$, we see that $D^{\prime}$ is algebraically isomorphic to $D$, but since
both carry the discrete topology, this is equivalent to them being isomorphic
in $\mathsf{LCA}_{\mathfrak{A}}$. Similarly, the algebraic right
$\mathfrak{A}$-module section $D^{\prime}\hookrightarrow M$ is tautologically
continuous as the preimage of any open is trivially open in the discrete
topology. Thus, we also get a section in $\mathsf{LCA}_{\mathfrak{A}}$, i.e.
we have%
\[
M\simeq H\oplus D
\]
in $\mathsf{LCA}_{\mathfrak{A}}$. This proves the first claim.\newline(Claim
2) We adapt the proof of Claim 1. Again, $V$ can be uniquely characterized as
the connected component. The same argument shows $V\in\mathsf{LCA}%
_{\mathfrak{A}}$ and again we obtain that $V\hookrightarrow M$ is a subobject
in $\mathsf{LCA}_{\mathfrak{A}}$. Next, apply Lemma \ref{Lemma_RZOrRTHulls} to
$V$. We obtain a canonical admissible epic%
\[
q:\tilde{V}\twoheadrightarrow V\text{,}%
\]
where $V$ is a vector $\mathfrak{A}$-module. By Lemma
\ref{lemma_VectorModulesAreAlgebraicallyInjectiveAndProjective} we deduce that
$\tilde{V}$ is algebraically an injective $\mathfrak{A}$-module. However,
being an epic in $\mathsf{LCA}_{\mathfrak{A}}$, $q$ is also an epic in
$\mathsf{Mod}_{\mathfrak{A}}$. Since $\mathfrak{A}$ is hereditary, all
quotients of an injective module are again injective, cf.
\cite[(3.22)\ Theorem]{MR1653294}. Thus, $V$ is also injective as an algebraic
right module. Now, proceed as in the proof of Claim 1.
\end{proof}

\begin{example}
\label{example_CyclicGroupCannotSplitOffTorus}Let us demonstrate that the
assumption of the order $\mathfrak{A}$ being hereditary is truly necessary.
Let $G\simeq\mathbb{Z}/n$ be a cyclic group. We work with right $\mathbb{Z}%
[G]$-modules, where $\mathbb{Z}[G]\simeq\mathbb{Z}[t]/(t^{n}-1)$ is the group
ring. We recall that the trivial $G$-module $\mathbb{Z}$ has the standard
projective resolution%
\begin{equation}
\left[  \cdots\overset{N}{\longrightarrow}\mathbb{Z}[G]\overset{1-t}%
{\longrightarrow}\mathbb{Z}[G]\overset{N}{\longrightarrow}\mathbb{Z}%
[G]\overset{1-t}{\longrightarrow}\mathbb{Z}[G]\right]  _{-\infty,0}%
\overset{\Sigma}{\longrightarrow}\mathbb{Z}\text{,}\label{lta_1}%
\end{equation}
where $N:=1+t+\cdots+t^{n-1}$ is the norm operator of $G$ and $\Sigma$ sends
each $g\in G$ to $1$. Write $N_{G}:=\operatorname*{im}(N)\subseteq
\mathbb{Z}[G]$ for the norm ideal, and $I_{G}:=\operatorname*{im}%
(1-t)\subseteq\mathbb{Z}[G] $ for the augmentation ideal. We have the short
exact sequence%
\[
N_{G}\hookrightarrow\mathbb{R}[G]\twoheadrightarrow\mathbb{R}[G]/N_{G}%
\]
inducing the long exact sequence%
\[
\operatorname*{Ext}\nolimits_{\mathbb{Z}[G]}^{i}(\mathbb{Z},\mathbb{R}%
[G])\longrightarrow\operatorname*{Ext}\nolimits_{\mathbb{Z}[G]}^{i}%
(\mathbb{Z},\mathbb{R}[G]/N_{G})\overset{\partial}{\longrightarrow
}\operatorname*{Ext}\nolimits_{\mathbb{Z}[G]}^{i+1}(\mathbb{Z},N_{G}%
)\longrightarrow\operatorname*{Ext}\nolimits_{\mathbb{Z}[G]}^{i+1}%
(\mathbb{Z},\mathbb{R}[G])
\]
and the two outer terms vanish since $\mathbb{R}[G]$ is an injective module
(Lemma \ref{lemma_VectorModulesAreAlgebraicallyInjectiveAndProjective}). Thus,
we obtain an isomorphism%
\[
\operatorname*{Ext}\nolimits_{\mathbb{Z}[G]}^{1}(\mathbb{Z},\mathbb{R}%
[G]/N_{G})\underset{\cong}{\overset{\partial}{\longrightarrow}}%
\operatorname*{Ext}\nolimits_{\mathbb{Z}[G]}^{2}(\mathbb{Z},N_{G})
\]
Next, observe that $\mathbb{Z}\overset{\sim}{\longrightarrow}N_{G}$, $1\mapsto
N$, is an isomorphism of $\mathbb{Z}[G]$-modules. Hence, $\operatorname*{Ext}%
\nolimits_{\mathbb{Z}[G]}^{2}(\mathbb{Z},N_{G})\cong H^{2}(G,\mathbb{Z})$ and
the latter is known to be isomorphic to $\mathbb{Z}/n$ (this computation can
easily be done using the resolution in Equation \ref{lta_1}. Instead of
unwinding $\partial$, let us describe the $\operatorname*{Ext}\nolimits^{1}%
$-group on the left explicitly. Using $\operatorname*{Ext}%
\nolimits_{\mathbb{Z}[G]}^{1}(\mathbb{Z},\mathbb{R}[G]/N_{G})\cong%
\mathbf{R}\operatorname*{Hom}_{\mathbb{Z}[G]}(\mathbb{Z},\mathbb{R}%
[G]/N_{G}[1])$, and the standard resolution, we learn that giving an element
of $\operatorname*{Ext}\nolimits^{1}$ is equivalent to giving a morphism of
complexes%
\[%
\xymatrix{
\cdots\ar[r]^{N} & \mathbb{Z}[G] \ar[r]^{1-t} \ar[d] & \mathbb{Z}%
[G] \ar[r]^{N} \ar[d] & \mathbb{Z}[G] \ar[r]^{1-t} \ar[d]^{\alpha}
& \mathbb{Z}[G] \ar[d] \\
\cdots\ar[r] & 0 \ar[r]  & 0 \ar[r] & \mathbb{R}[G]/{N_G} \ar[r] & 0
}%
\]
and since a $\mathbb{Z}[G]$-module morphism is defined on $\mathbb{Z}[G]$
itself by assigning a value to $1_{G}$, all possible maps $\alpha$ can only be
multiplication $1_{G}\mapsto1_{G}\cdot\alpha_{0}$ with $\alpha_{0}%
\in\mathbb{Z}/n\mathbb{Z}$. These maps are precisely those corresponding under
$\partial$ to $H^{2}(G,\mathbb{Z})\cong\mathbb{Z}/n$. Next, let us concretely
construct the extension corresponding to the element of the
$\operatorname*{Ext}\nolimits^{1}$-group. To this end, we need to form the
pushout $M_{\alpha}$ of%
\[%
\xymatrix{
\mathbb{Z}[G] \ar[r]^{1-t} \ar[d]  & \mathbb{Z}[G] \ar[d] \ar@{->>}%
[r] & \mathbb{Z} \ar@{=}[d] \\
\mathbb{R}[G]/{N_G} \ar[r] &  M_{\alpha} \ar@{->>}[r] & \mathbb{Z}
}%
\]
Concretely, that is%
\[
M_{\alpha}:=\frac{\left\{  (a,b)\mid a\in\mathbb{Z}[G]\text{, }b\in
\mathbb{R}[G]/N_{G}\right\}  }{\text{all pairs }(x(1-t),\overline{x}%
\alpha)\text{ with }x\in\mathbb{Z}[G]}\text{.}%
\]
We then get a short exact sequence%
\[
\mathbb{R}[G]/N_{G}\hookrightarrow M_{\alpha}\twoheadrightarrow\mathbb{Z}%
\]
in $\mathsf{LCA}_{\mathbb{Z}[G]}$; the maps are $b\mapsto(0,b)$ and
$(a,b)\mapsto\Sigma a$. On the level of the underlying LCA\ groups it looks
like $\mathbb{Z}^{n-1}\oplus\mathbb{T}\hookrightarrow M_{\alpha}%
\twoheadrightarrow\mathbb{Z}$ in $\mathsf{LCA}$. If this sequence splits in
$\mathsf{LCA}_{\mathbb{Z}[G]}$, then so it does in $\mathsf{Mod}%
_{\mathbb{Z}[G]}$. However, unless $\alpha=[0]$ in $H^{2}(G,\mathbb{Z})$, this
is impossible.
\end{example}

\begin{example}
The same computation can be adapted to many other finite groups. For example,
for $G:=D_{2n}$ the dihedral group of $2n$ elements, $H^{2}(D_{2n}%
,\mathbb{Z})\cong\mathbb{Z}/2$ also gives rise to an analogous example.
\end{example}

\begin{theorem}
\label{thm_DirectSumSplittingsForHereditaryOrders}Suppose $\mathfrak{A}$ is a
hereditary order.

\begin{enumerate}
\item Every $M\in\mathsf{LCA}_{\mathfrak{A},nss}$ admits an isomorphism
$M\cong V\oplus T\oplus D$, where $V$ is a vector right $\mathfrak{A}$-module,
$T$ a right $\mathfrak{A}$-module and finite-dimensional real torus, and $D$ a
discrete right $\mathfrak{A}$-module.

\item Every $M\in\mathsf{LCA}_{\mathfrak{A},cg}$ admits an isomorphism $M\cong
V\oplus G\oplus C$, where $V$ is a vector right $\mathfrak{A}$-module, $G$ a
right $\mathfrak{A}$-module and finite rank free $\mathbb{Z}$-module, and $C$
a compact right $\mathfrak{A}$-module.
\end{enumerate}
\end{theorem}

The analogous claims hold for left modules. For $\mathfrak{A}$ the maximal
order of a Dedekind domain this is due to Levin \cite{MR0310125}. For
non-hereditary orders we cannot expect such direct sum decompositions to exist.

\begin{proof}
(1) We work with right modules. The proof of \cite[Proposition 2.14]{obloc}
can be adapted with the following change: As $M$ has no small subgroups, its
underlying\ LCA group is isomorphic to $\mathbb{R}^{n}\oplus\mathbb{T}%
^{m}\oplus D$ for $D$ some discrete group. Now by Lemma
\ref{Lemma_KeySplittingLemma} (2) this can be strengthened to a direct sum
decomposition $M\cong E\oplus D$ in $\mathsf{LCA}_{\mathfrak{A}}$ with
$D\in\mathsf{LCA}_{\mathfrak{A}}$ discrete and $E\in\mathsf{LCA}%
_{\mathfrak{A}}$ having underlying LCA group $\mathbb{R}^{n}\oplus
\mathbb{T}^{m}$ (as witnessed by Example
\ref{example_CyclicGroupCannotSplitOffTorus} this step required $\mathfrak{A}$
to be hereditary). Now, $E^{\vee}\in\left.  _{\mathfrak{A}}\mathsf{LCA}%
\right.  $ has underlying LCA group $\mathbb{R}^{n}\oplus\mathbb{Z}^{m}$. Use
the left module version of Lemma \ref{Lemma_KeySplittingLemma} to promote this
to $E^{\vee}\cong V\oplus G$ in $\left.  _{\mathfrak{A}}\mathsf{LCA}\right.  $
for $V$ a left vector module and $G$ a discrete left module with underlying
group $\mathbb{Z}^{n}$. Dualizing again, we get $E\cong V^{\vee}\oplus
G^{\vee}$ in $\mathsf{LCA}_{\mathfrak{A}}$, and $G^{\vee}$ is a torus as
required. (2) This argument is Pontryagin dual, using that duality exchanges
compactly generated modules with those without small subgroups, Lemma
\ref{lemma_CGExchangesWithNSS}.
\end{proof}

Combining all these results, we obtain some further flexibility with regards
to decompositions in the style of Theorem \ref{thm_FirstDecomp}, even without
having to assume that $\mathfrak{A}$ be hereditary.

\begin{lemma}
\label{Lemma_CompactSubobjectDecomp}Suppose $M\in\mathsf{LCA}_{\mathfrak{A}}
$. Then there exist exact sequences%
\begin{equation}
C\hookrightarrow M\twoheadrightarrow V\oplus D\qquad\text{and}\qquad
V^{\prime}\oplus C^{\prime}\hookrightarrow M\twoheadrightarrow D^{\prime
}\label{l_xA3}%
\end{equation}
with $C,C^{\prime}$ compact, $V,V^{\prime}$ vector $\mathfrak{A}$-modules and
$D,D^{\prime}$ discrete.
\end{lemma}

\begin{proof}
We construct the sequence on the left first: By Theorem \ref{thm_FirstDecomp}
(1) there exists an exact sequence $H\hookrightarrow M\twoheadrightarrow
D^{\prime}$ with $H\in\mathsf{LCA}_{\mathfrak{A},cg}$ and $D^{\prime}$
discrete. Further, by Lemma \ref{Lemma_CompactPartOfCG} there is an exact
sequence $C\hookrightarrow H\twoheadrightarrow W$ with $C$ compact and
$W\in\mathsf{LCA}_{\mathfrak{A}}$ having underlying LCA group isomorphic to
$\mathbb{R}^{n}\oplus\mathbb{Z}^{m}$ for suitable $n,m$. This yields a
filtration $C\hookrightarrow H\hookrightarrow M$ and Noether's Lemma for exact
categories (\cite[Lemma 3.5]{MR2606234}) tells us that%
\[
H/C\hookrightarrow M/C\twoheadrightarrow M/H
\]
is exact in $\mathsf{LCA}_{\mathfrak{A}}$. Unravelling these quotients, we
learn that%
\[
W\hookrightarrow M/C\twoheadrightarrow D^{\prime}%
\]
is exact. Now, by Lemma \ref{Lemma_KeySplittingLemma} (1) the module $W$
splits as a direct sum $W\cong V\oplus\tilde{D}$ with $V$ a vector
$\mathfrak{A}$-module and $\tilde{D}$ discrete in $\mathsf{LCA}_{\mathfrak{A}}
$. Thus, by the above exact sequence, $V$ is a subobject of $M/C$. However, by
Theorem \ref{thm_VectorModulesAreInjectiveAndBijective} vector modules are
injective objects, so $V$ is a direct summand of $M/C$. The complementary
summand, call it $J$, then is an extension $\tilde{D}\hookrightarrow
J\twoheadrightarrow D^{\prime}$ of discrete modules, so $J$ is itself
discrete; and $M/C\cong V\oplus J$. Combining these results, we arrive at the
exact sequence $C\hookrightarrow M\twoheadrightarrow V\oplus J$ with $J$
discrete. Letting $D:=J$ proves our first claim. All arguments so far work for
both left and right $\mathfrak{A}$-modules. Thus, we obtain the second exact
sequence, i.e. Equation \ref{l_xA3} on the right, by using the first exact
sequence for the left module $M^{\vee}$ and dualizing back.
\end{proof}

We record an elementary fact:

\begin{lemma}
\label{Lemma_Elem1}A module $M\in\mathsf{Mod}_{\mathfrak{A}}$ is finitely
generated over $\mathfrak{A}$ if and only if it is finitely generated as a
$\mathbb{Z}$-module.
\end{lemma}

\begin{proof}
This follows since $\mathfrak{A}\simeq\mathbb{Z}^{n}$ for $n:=\dim
_{\mathbb{Q}}A$, see \S \ref{subsect_Basics}. In detail: If $M$ is finitely
generated over $\mathfrak{A}$, then there exists a surjection $\mathfrak{A}%
^{m}\twoheadrightarrow M$ for some $m\in\mathbb{Z}_{\geq0}$, and thus a
surjection by $\mathbb{Z}^{nm}$. Conversely, any generators of $M$ over
$\mathbb{Z}$ of course also generate $M$ as an $\mathfrak{A}$-module.
\end{proof}

\begin{lemma}
\label{Lemma_ProjCovers}Let $\mathfrak{A}$ be an arbitrary order in $A$.

\begin{enumerate}
\item Every discrete module $M\in\mathsf{LCA}_{\mathfrak{A}}$ has a discrete
projective cover $P$ in $\mathsf{LCA}_{\mathfrak{A}}$. If $M$ is compactly
generated, $P$ is also compactly generated.

\item Every compact module $M\in\mathsf{LCA}_{\mathfrak{A}}$ has a compact
injective hull $I$ in $\mathsf{LCA}_{\mathfrak{A}}$. If $M$ has no small
subgroups, $I$ also has no small subgroups.
\end{enumerate}
\end{lemma}

\begin{proof}
(1) The category $\mathsf{Mod}_{\mathfrak{A}}$ has enough projectives, so we
find a projective cover $P\twoheadrightarrow M$. Equipping $P$ with the
discrete topology, the epic is tautologically continuous and even an
admissible epic in $\mathsf{LCA}_{\mathfrak{A}}$ because in the case of the
discrete topology there is nothing to check topologically. For the second
statement, note that a discrete module is compactly generated if and only if
its underlying additive group is finitely generated by \cite[Theorem
2.5]{MR0215016}, so $M$ is a finitely generated $\mathfrak{A}$-module by Lemma
\ref{Lemma_Elem1}, and thus its projective cover is also finitely generated
over $\mathfrak{A}$, and using the converse of Lemma \ref{Lemma_Elem1}, and
\cite[Theorem 2.5]{MR0215016} we deduce that $P$ is compactly generated.
(2)\ Pontryagin dual to (1), using that duality swaps compactly generated
modules with those without small subgroups, Lemma
\ref{lemma_CGExchangesWithNSS}.
\end{proof}

\begin{corollary}
\label{cor_Aux_GlobalDimension}Let $\mathfrak{A}$ be an arbitrary order in $A
$.

\begin{enumerate}
\item Every discrete module $M\in\mathsf{LCA}_{\mathfrak{A}}$ has a projective
resolution by discrete modules in $\mathsf{LCA}_{\mathfrak{A}}$ of length at
most the projective dimension of $M$ in $\mathsf{Mod}_{\mathfrak{A}}$.

\item Every compact module $M\in\mathsf{LCA}_{\mathfrak{A}}$ has an injective
resolution by compact modules in $\mathsf{LCA}_{\mathfrak{A}}$ of length at
most the projective dimension of $M^{\vee}$ in $\left.  _{\mathfrak{A}%
}\mathsf{Mod}\right.  $.
\end{enumerate}
\end{corollary}

\begin{proof}
Inductive usage of Lemma \ref{Lemma_ProjCovers}.
\end{proof}

\begin{example}
\label{example_FinGroup_InfGlobalDim}For any finite group $G\neq1$, the global
dimension of $\mathbb{Z}[G]$ is infinite. To see this, use that $G$ contains a
non-trivial cyclic subgroup $C\simeq\mathbb{Z}/m$ with $m\geq2$. The group
cohomology of a cyclic group is periodic in the strictly positive degree range
(compare Example \ref{example_CyclicGroupCannotSplitOffTorus}, where we recall
this) with $H^{2n}(C,\mathbb{Z})\cong\mathbb{Z}/m$ for all $n\geq1$. Thus,%
\[
\operatorname*{Ext}\nolimits_{\mathbb{Z}[G]}^{2n}(\mathbb{Z}%
,\mathrm{\operatorname*{Ind}}_{C}^{G}(\mathbb{Z}))=H^{2n}%
(G,\mathsf{\operatorname*{Ind}}_{C}^{G}(\mathbb{Z}))\cong H^{2n}%
(C,\mathbb{Z})\cong\mathbb{Z}/m\neq0\text{.}%
\]

\end{example}

\begin{conjecture}
It is not possible to find strictly shorter projective resolutions of
discretes in $\mathsf{LCA}_{\mathfrak{A}}$ than in $\mathsf{Mod}%
_{\mathfrak{A}}$.
\end{conjecture}

The following result complements Corollary \ref{cor_Aux_GlobalDimension}
regarding resolutions of opposite nature.

\begin{corollary}
\label{cor_Aux_FiniteToriHaveProjResolutions}Let $\mathfrak{A}$ be an
arbitrary order in $A$.

\begin{enumerate}
\item Suppose $T\in\mathsf{LCA}_{\mathfrak{A}}$ has underlying LCA group
$\mathbb{T}^{n}$ for some $n\in\mathbb{Z}_{\geq0}$. Then it admits a (possibly
infinitely long) projective resolution in $\mathsf{LCA}_{\mathfrak{A}}$
starting in a vector $\mathfrak{A}$-module $V$,%
\[
\cdots\longrightarrow P^{2}\longrightarrow P^{1}\longrightarrow
V\twoheadrightarrow T
\]
and if $\mathfrak{A}$ has finite global dimension $s$, then the length of this
resolution is at most $s+1$.

\item Suppose $G\in\mathsf{LCA}_{\mathfrak{A}}$ has underlying LCA group
$\mathbb{Z}^{n}$ for some $n\in\mathbb{Z}_{\geq0}$. Then it admits a (possibly
infinitely long) injective resolution in $\mathsf{LCA}_{\mathfrak{A}}$
starting in a vector $\mathfrak{A}$-module $V$,%
\[
G\hookrightarrow V\longrightarrow I^{1}\longrightarrow I^{2}\longrightarrow
\cdots\text{,}%
\]
and if $\mathfrak{A}$ has finite global dimension $s$, then the length of this
resolution is at most $s+1$.
\end{enumerate}
\end{corollary}

\begin{proof}
(1) By Lemma \ref{Lemma_RZOrRTHulls} (2) there exists an exact sequence
$D\hookrightarrow P\twoheadrightarrow T$ with $D$ discrete. By Lemma
\ref{Lemma_ProjCovers} discrete modules permit projective covers in
$\mathsf{LCA}_{\mathfrak{A}}$. Inductively, this produces a projective
resolution of $T$, the initial step being a vector module (which is projective
by Theorem \ref{thm_VectorModulesAreInjectiveAndBijective}) and all further
terms discrete. In particular, since the projective resolution of $D$ can be
carried out in $\mathsf{Mod}_{\mathfrak{A}}$, and then just be imported to
$\mathsf{LCA}_{\mathfrak{A}}$ by equipping every term with the discrete
topology, the claim about the projective dimension follows. (2) Pontryagin
dual to (1).
\end{proof}

\begin{question}
Besides the upper bound, is there a general formula for the injective
dimension of such a $G$ in $\mathsf{LCA}_{\mathfrak{A}}$ in terms of $G$,
regarded as an object in $\mathsf{Mod}_{\mathfrak{A}}$ alone?
\end{question}

\begin{example}
If $T$ has the torus $\mathbb{T}^{\omega}$ for some strictly infinite cardinal
$\omega$ as its underlying LCA group, then $\mathbb{T}^{\omega}$ has small
subgroups (since its Pontryagin dual is not compactly generated), and
Moskowitz proves that $T$ does not even admit a projective resolution only in
$\mathsf{LCA}$, \cite[Theorem 3.6 (2)]{MR0215016}. Hence, the
finite-dimensionality is crucial in Corollary
\ref{cor_Aux_FiniteToriHaveProjResolutions}. Similarly, if $D$ is any discrete
module in $\mathsf{LCA}_{\mathfrak{A}}$, then if we take an injective
resolution of $D$ in $\mathsf{Mod}_{\mathfrak{A}}$, say $D\hookrightarrow
I^{\bullet}$ and view this as an exact complex in $\mathsf{LCA}_{\mathfrak{A}%
}$, each $I^{n}$ equipped with the discrete topology, then in general there is
no hope that the modules $I^{n}$ are injective objects also in the category
$\mathsf{LCA}_{\mathfrak{A}}$.
\end{example}

We remind the reader that all global dimensions, including $+\infty$, can
occur, see Remark \ref{rmk_on_global_dims}.

\section{Structure theorems}

Below, we shall make frequent use of Schlichting's concepts of left and right
$s$-filtering categories, see \cite[Definition 3.1]{obloc} or \cite[Appendix
A]{MR3510209}.

\begin{proposition}
\label{Prop_CGIsLeftSFiltering}The full subcategory $\mathsf{LCA}%
_{\mathfrak{A},cg}\hookrightarrow\mathsf{LCA}_{\mathfrak{A}}$ is left $s$-filtering.
\end{proposition}

\begin{proof}
Firstly, it is left filtering. This amounts to the purely topological fact
that the closure of a compactly generated LCA group inside a locally compact
group is again compactly generated. The proof of \cite[Lemma 3.2]{obloc} works
verbatim. It remains to prove that the inclusion $\mathsf{LCA}_{\mathfrak{A}%
,cg}\hookrightarrow\mathsf{LCA}_{\mathfrak{A}}$ is left special. If
$G^{\prime\prime}\hookrightarrow G\overset{q}{\twoheadrightarrow}G^{\prime}$
is an exact sequence with $G^{\prime}\in\mathsf{LCA}_{\mathfrak{A},cg}$, apply
Lemma \ref{Lemma_CompactSubobjectDecomp} to $G$. We arrive at the solid arrows
in the following commutative diagram%
\begin{equation}%
\xymatrix{
C \ar@{^{(}->}[d]_{i} \ar[dr]_{h} \ar[drr]^{0} \\
G \ar@{->>}[r]_{q} \ar@{->>}[d] & G^{\prime} \ar@{->>}[r]_-{r} & G^{\prime
}/{\operatorname{im}(h)} \\
V \oplus D \ar@{-->}[urr]_{w} \text{.}
}%
\label{l_xA5}%
\end{equation}
We explain how to set this up: (1) The map $h$ is defined as the composition
$q\circ i$. (2) Since $C$ is compact, its set-theoretic image
$\operatorname*{im}_{\mathsf{Set}}(h)$ is also compact, and thus in particular
closed in $G^{\prime}$. As a result, the set-theoretic image agrees with the
image in the sense of the category $\mathsf{LCA}_{\mathfrak{A}}$ (this is a
subtle issue with terminology, see \cite[\S 2, Notation]{obloc}), and moreover
the quotient $G^{\prime}/\operatorname*{im}(h)$ exists in $\mathsf{LCA}%
_{\mathfrak{A}} $. (3) Proceeding along the resulting quotient map in the
diagram above, the dashed arrow $w$ exists by the universal property of
cokernels. Since $\mathsf{LCA}_{\mathfrak{A}}$ is quasi-abelian, it is
idempotent complete and thus by \cite[Proposition 7.6]{MR2606234} the morphism
$w$ must be an admissible epic. Let us try to characterize the underlying
LCA\ group of $X:=G^{\prime}/\operatorname*{im}(h)$. As $G^{\prime}$ is
compactly generated, so are all its quotients by \cite[Theorem 2.6]%
{MR0215016}. Thus, as an LCA group, we have $X\simeq\mathbb{R}^{A}%
\oplus\mathbb{Z}^{B}\oplus C$ for suitable $A,B\in\mathbb{Z}_{\geq0}$ and $C$
compact. On the other hand, $V\oplus D$ has the underlying LCA group
$\mathbb{R}^{n}\oplus D$ with $D$ discrete and by \cite[Corollary 2 to Theorem
7]{MR0442141} all closed subgroups can be classified, and as a result all
admissible quotients can only have the underlying LCA group $\mathbb{R}%
^{N}\oplus\mathbb{T}^{M}\oplus D^{\prime}$ for suitable $N,M\in\mathbb{Z}%
_{\geq0}$ and $D^{\prime}$ discrete. Having these two contrasting structural
results, we learn that $X$ can only have $\mathbb{R}^{N}\oplus\mathbb{T}%
^{M}\oplus\mathbb{Z}^{B}$ as its underlying LCA group for suitable
$N,M,B\in\mathbb{Z}_{\geq0}$ (a very clean way to see this is to use the
canonical filtration of \cite[Proposition 2.2]{MR2329311} in $\mathsf{LCA}$
since it canonically identifies the three direct summands as graded pieces of
the filtration). Apply Lemma \ref{Lemma_CompactPartOfCG} to $X$, giving us an
exact sequence%
\[
T\hookrightarrow X\twoheadrightarrow W\qquad\text{in}\qquad\mathsf{LCA}%
_{\mathfrak{A}}%
\]
such that the underlying LCA groups of $T$ are $\mathbb{T}^{M}$, and
$\mathbb{R}^{N}\oplus\mathbb{Z}^{B}$ for $W$. Apply Lemma
\ref{Lemma_KeySplittingLemma} (1) to $W$ to promote this exact sequence to%
\[
T\hookrightarrow X\twoheadrightarrow\tilde{V}\oplus\tilde{D}\qquad
\text{in}\qquad\mathsf{LCA}_{\mathfrak{A}}%
\]
with $\tilde{V}$ a vector $\mathfrak{A}$-module and $\tilde{D}$ having
underlying LCA group $\mathbb{Z}^{B}$.\ We pick admissible epics from
projective objects $P_{i}\in\mathsf{LCA}_{\mathfrak{A}}$ to these objects as
follows: (1) $c_{1}:P_{1}\twoheadrightarrow T$ can be produced by Lemma
\ref{Lemma_RZOrRTHulls} (2), in this case $P_{1}$ is a vector $\mathfrak{A}%
$-module; (2) $c_{2}:P_{2}\twoheadrightarrow\tilde{V}$ is just the identity
since by Theorem \ref{thm_VectorModulesAreInjectiveAndBijective} the object
$\tilde{V}$ is itself already projective, and (3) $c_{3}:P_{3}%
\twoheadrightarrow\tilde{D}$ is taken as a projective cover, which exists by
Lemma \ref{Lemma_ProjCovers}; in this case $P_{3}$ is a compactly generated
discrete $\mathfrak{A}$-module. Since morphisms from projectives can be lifted
along admissible epics, we obtain the lift $\tilde{c}_{23}$ of%
\begin{equation}%
\xymatrix{
P_1 \ar@{->>}[d]_{c_1} \ar[dr] & & P_2 \oplus P_3 \ar@{->>}[d]^{c_2 + c_3}
\ar@{-->}[dl]^{\tilde{c}_{23}} \\
T \ar@{^{(}->}[r] & X \ar@{->>}[r] & \tilde{V} \oplus\tilde{D}\text{.}
}%
\label{l_xA6}%
\end{equation}
Recalling that $X=G^{\prime}/\operatorname*{im}(h)$, take this morphism
$P_{1}\oplus P_{2}\oplus P_{3}\longrightarrow G^{\prime}/\operatorname*{im}%
(h)$, which arises as the sum $c_{1}+\tilde{c}_{23}$, and lift it also along
the admissible epic $r$ in Diagram \ref{l_xA5}. Call this lift $\hat{c}$. Then
we arrive at a commutative diagram%
\[%
\xymatrix{
C \oplus P_1 \oplus P_2 \oplus P_3 \ar[d]_{i + {\hat{c}}^{\prime}}
\ar[dr]^{h + \hat{c}} \\
G \ar@{->>}[r]_{q} & G^{\prime}\text{,}
}%
\]
where $\hat{c}^{\prime}$ is a further lift of the morphism $\hat{c}$ along the
admissible epic $q$. Let us study the morphism $h+\hat{c}$: Firstly, it is a
morphism in $\mathsf{LCA}_{\mathfrak{A}}$. Next, tracing through the
construction, it is surjective: all elements in $\operatorname*{im}%
(h)\subseteq G^{\prime}$ are of course surjected on by $C$, and $P_{1}\oplus
P_{2}\oplus P_{3}$ were constructed just in such a way to surject onto the
quotient $X=G^{\prime}/\operatorname*{im}(h)$; see Diagram \ref{l_xA6}.
Finally, since $C$ is compact, $P_{1},P_{2}$ are vector $\mathfrak{A}%
$-modules, and $P_{3}$ is compactly generated, $C\oplus P_{1}\oplus
P_{2}\oplus P_{3}$ is compactly generated in total. In particular, it is
$\sigma$-compact (i.e. a countable union of compact sets), and thus by
Pontryagin's Open Mapping Theorem, $h+\hat{c}$ is an open map, see
\cite[Theorem 3]{MR0442141}. Being surjective and open, $h+\hat{c}$ is an
admissible epic in $\mathsf{LCA}_{\mathfrak{A}}$, Lemma
\ref{lemma_CharacterizeAdmissibleMorphisms}. Thus, we can set up the following
commutative diagram%
\[%
\xymatrix{
K \ar@{^{(}->}[r] \ar@{-->}[d] & C \oplus P_1 \oplus P_2 \oplus P_3 \ar
[d] \ar@{->>}[r]^-{h + \hat{c}} & G^{\prime} \ar@{=}[d] \\
G^{\prime\prime} \ar@{^{(}->}[r] & G \ar@{->>}[r]_{q} & G^{\prime}\text{,}
}%
\]
where $K:=\ker(h+\hat{c})$, and the dashed arrow stems from the universal
property of kernels. Since $C\oplus P_{1}\oplus P_{2}\oplus P_{3}$ is
compactly generated, so is its subobject $K$ by \cite[Theorem 2.6
(2)]{MR0215016}, and thus the entire top row is an exact sequence with objects
in $\mathsf{LCA}_{\mathfrak{A},cg}$. This confirms that $\mathsf{LCA}%
_{\mathfrak{A},cg}\hookrightarrow\mathsf{LCA}_{\mathfrak{A}}$ is left special.
\end{proof}

\begin{corollary}
\label{Cor_Ext}For an exact sequence $G^{\prime}\hookrightarrow
G\twoheadrightarrow G^{\prime\prime}$ of objects in $\mathsf{LCA}%
_{\mathfrak{A}}$ we have $G\in\mathsf{LCA}_{\mathfrak{A},cg}$ if and only if
$G^{\prime},G^{\prime\prime}\in\mathsf{LCA}_{\mathfrak{A},cg}$.
\end{corollary}

\begin{proof}
\cite[Proposition A.2 and Definition A.1 (1)]{MR3510209}.
\end{proof}

\begin{remark}
\label{rmk_LCAcg_IsClosedUnderExtensions}In particular, $\mathsf{LCA}%
_{\mathfrak{A},cg}\hookrightarrow\mathsf{LCA}_{\mathfrak{A}}$ is closed under
extensions and thus carries a natural exact structure itself, \cite[Lemma
10.20]{MR2606234}.
\end{remark}

\begin{remark}
\label{rmk_ProofsByDualization}By duality (Proposition
\ref{Prop_LCAIsQuasiAbelianAndPontryaginDuality}) and the exchange properties
of Lemma \ref{lemma_CGExchangesWithNSS}, Proposition
\ref{Prop_CGIsLeftSFiltering} literally yields that%
\[
\left.  _{nss,\mathfrak{A}}\mathsf{LCA}^{op}\right.  \hookrightarrow\left.
_{\mathfrak{A}}\mathsf{LCA}^{op}\right.
\]
is left $s$-filtering. However, by the symmetry of left and right
$s$-filtering, this just means that $\left.  _{nss,\mathfrak{A}}%
\mathsf{LCA}\right.  \hookrightarrow\left.  _{\mathfrak{A}}\mathsf{LCA}%
\right.  $ is right $s$-filtering. Applying this observation to the opposite
order $\mathfrak{A}$, we get that $\mathsf{LCA}_{\mathfrak{A},nss}%
\hookrightarrow\mathsf{LCA}_{\mathfrak{A}}$ is also right $s$-filtering. Using
duality on this statement and Lemma \ref{lemma_CGExchangesWithNSS} once more,
we learn that $\left.  _{cg,\mathfrak{A}}\mathsf{LCA}^{op}\right.
\hookrightarrow\left.  _{\mathfrak{A}}\mathsf{LCA}^{op}\right.  $ is right
$s$-filtering, and thus $\left.  _{cg,\mathfrak{A}}\mathsf{LCA}\right.
\hookrightarrow\left.  _{\mathfrak{A}}\mathsf{LCA}\right.  $ is left
$s$-filtering. In particular, Corollary \ref{Cor_Ext} generalizes to
$\mathsf{LCA}_{\mathfrak{A},nss}$, and all full subcategories%
\[
\mathsf{LCA}_{\mathfrak{A},nss}\hookrightarrow\mathsf{LCA}_{\mathfrak{A}%
}\text{,}\qquad\mathsf{LCA}_{\mathfrak{A},cg}\hookrightarrow\mathsf{LCA}%
_{\mathfrak{A}}\text{,}\qquad\left.  _{cg,\mathfrak{A}}\mathsf{LCA}\right.
\hookrightarrow\left.  _{\mathfrak{A}}\mathsf{LCA}\right.  \text{,}%
\qquad\left.  _{nss,\mathfrak{A}}\mathsf{LCA}\right.  \hookrightarrow\left.
_{\mathfrak{A}}\mathsf{LCA}\right.
\]
are fully exact subcategories and hence carry the structure of an exact
category themselves.
\end{remark}

\section{Injectives and projectives}

We determine the projective objects in $\mathsf{LCA}_{\mathfrak{A}}$. To this
end, we follow the strategy of Moskowitz in \cite{MR0215016}, who deals with
the case of $\mathsf{LCA}$ itself.

\begin{proposition}
\label{prop_ClassifyInjectivesAndProjectives}Let $\mathfrak{A}$ be an
arbitrary order.

\begin{enumerate}
\item Every injective object in $\mathsf{LCA}_{\mathfrak{A}}$ is isomorphic to
$V\oplus I$, where $V$ is a vector $\mathfrak{A}$-module and $I$ compact
connected such that $I^{\vee}$ is projective as an algebraic left
$\mathfrak{A}$-module.

\item Every projective object in $\mathsf{LCA}_{\mathfrak{A}}$ is isomorphic
to $V\oplus P$, where $V$ is a vector $\mathfrak{A}$-module and $P$ discrete
such that $P$ is projective as an algebraic right $\mathfrak{A}$-module.
\end{enumerate}

Conversely, all objects of this shape are injective resp. projective.
\end{proposition}

\begin{proof}
(1) Suppose $G\in\mathsf{LCA}_{\mathfrak{A}}$ is injective. Firstly, we claim
that $G$ must be injective in $\mathsf{Mod}_{\mathfrak{A}}$ as well. To see
this, let $M^{\prime}\hookrightarrow M\twoheadrightarrow M^{\prime\prime}$ be
any exact sequence in $\mathsf{Mod}_{\mathfrak{A}}$. Equipped with the
discrete topology, this is exact in $\mathsf{LCA}_{\mathfrak{A}}$. Let
$f:M\rightarrow G$ be any morphism in $\mathsf{Mod}_{\mathfrak{A}}$. As $M$
carries the discrete topology, $f$ is actually continuous, tautologically.
Since $G$ is injective in $\mathsf{LCA}_{\mathfrak{A}}$, there exists a lift
$f^{\prime}$ as in%
\[%
\xymatrix{
M^{\prime} \ar@{^{(}->}[rr] \ar[dr]_{f} & & M \ar@{-->}[dl]^{f^{\prime}} \\
& G
}%
\]
and if we forget the topology, the underlying algebraic morphism of
$f^{\prime}$ is of course also a lift of $f$ in $\mathsf{Mod}_{\mathfrak{A}}$.
Thus, $G$ is injective in $\mathsf{Mod}_{\mathfrak{A}}$. Next, we claim that
$G$ must be connected. To this end, consider the exact sequence $\mathfrak{A}%
\hookrightarrow A\otimes_{\mathbb{Q}}\mathbb{R}\twoheadrightarrow T$, where
the middle term is viewed with the standard real topology, and $T$ is plainly
defined as the corresponding quotient. For any element $x\in G$ we define a
morphism $f:\mathfrak{A}\rightarrow G$ by $f(\alpha):=x\cdot\alpha$. As $G$ is
injective, we can lift $f$ along%
\begin{equation}%
\xymatrix{
\mathfrak{A} \ar@{^{(}->}[rr] \ar[dr]_{f} & & {A \otimes_{\mathbb{Q}}
\mathbb{R}} \ar@{-->}[dl]^{\tilde{f} } \\
& G
}%
\label{l_xA7}%
\end{equation}
and (following Moskowitz's idea) note that since $A\otimes_{\mathbb{Q}%
}\mathbb{R}$ is connected (it has underlying LCA group $\mathbb{R}^{d}$ for
$d:=\dim_{\mathbb{Q}}A$), the image of the continuous map $\tilde{f}$ must
also be connected. However, $\tilde{f}(1_{\mathfrak{A}})=f(1_{\mathfrak{A}%
})=x$ was arbitrary, so we deduce that all elements of $G$ lie in the
connected component of the neutral element. Thus, $G$ is connected. Next,
apply Lemma \ref{Lemma_CompactSubobjectDecomp} to $G$, giving an exact
sequence%
\[
V\oplus C\hookrightarrow G\overset{q}{\twoheadrightarrow}D\qquad
\text{in}\qquad\mathsf{LCA}_{\mathfrak{A}}%
\]
with $V$ a vector $\mathfrak{A}$-module, $C$ compact and $D$ discrete. As $G$
is connected, its image under the quotient map $q$ is also connected, but $D$
is discrete, so we must have $D=0$ as $q$ is also surjective. Next, let
$C^{0}$ denote the connected component of $C$; at first this only makes sense
in $\mathsf{LCA}$. However, for every $\alpha\in\mathfrak{A}$ we get the solid
arrows of the commutative diagram (of LCA groups)%
\[%
\xymatrix{
C^0 \ar@{^{(}->}[d] \ar@{-->}[r] & C^0 \ar@{^{(}->}[d] \\
C \ar[r]^{\cdot\alpha} \ar@{->>}[d] & C \ar@{->>}[d] \\
C/{C^0} & C/{C^0}\text{.}
}%
\]
The composition of maps from the connected $C^{0}$ on the upper left to the
totally disconnected $C/C^{0}$ on the lower right is necessarily zero by
continuity. Thus, the dashed arrow exists by the universal property of
kernels. As this holds for all $\alpha\in\mathfrak{A}$, it follows that the
closed subgroup $C^{0}$ actually defines a subobject of $C$ in $\mathsf{LCA}%
_{\mathfrak{A}}$, and the quotient $C/C^{0}$ makes sense in $\mathsf{LCA}%
_{\mathfrak{A}}$ accordingly. Thus, we obtain a new exact sequence%
\[
V\oplus C^{0}\hookrightarrow\underset{=G}{V\oplus C}\overset{q^{\prime}%
}{\twoheadrightarrow}C/C^{0}\text{.}%
\]
Again, since $G$ is connected, but $C/C^{0}$ totally disconnected, we must
have $C/C^{0}=0$ because $q^{\prime}$ is surjective. It follows that $G\cong
V\oplus C^{0}$, where $V$ is a vector $\mathfrak{A}$-module. Recall that if a
product of objects is injective, then so are its factors. Hence, $C^{0}$ is
compact and injective in $\mathsf{LCA}_{\mathfrak{A}}$. By Pontryagin duality,
it follows that $C^{0\vee}$ is a discrete projective left $\mathfrak{A}%
$-module. This proves (1)\ for right modules. Note that the same argument
works for left modules: just change that $f$ needs to be defined as
$f(\alpha):=\alpha\cdot x$ in the context of Diagram \ref{l_xA7}. Finally, (2)
follows by Pontryagin duality, noting that it exchanges injectives with
projectives, compact with discrete, and duals of vector $\mathfrak{A}$-modules
remain vector $\mathfrak{A}$-modules. (3)\ We conclude by showing that all
such objects are indeed injective resp. projective. For the vector
$\mathfrak{A}$-module summand, this is just Theorem
\ref{thm_VectorModulesAreInjectiveAndBijective}. Next, if $P$ is projective in
$\mathsf{Mod}_{\mathfrak{A}}$ and discrete, then any lift as algebraic module
maps is continuous, because any map originating from a discrete space is
continuous. For injectives, argue by duality.
\end{proof}

\section{Isolating the real part}

We recall from \cite{obloc} that $\mathsf{LCA}_{\mathbb{R}C}$ is our
short-hand for $\mathsf{LCA}_{\mathbb{Z},\mathbb{R}C}$, i.e. the category of
LCA groups which admit an isomorphism to $\mathbb{R}^{n}\oplus C$ for some $n
$ and some compact $C$. As was shown \textit{loc. cit.} this is an idempotent
complete exact category, \cite[Lemma 3.10]{obloc}. These results turn out to
generalize to topological $\mathfrak{A}$-modules. Let us quickly set this up
to the extent which we shall need later:

\begin{definition}
\label{def_RC}Let $\mathsf{LCA}_{\mathfrak{A},\mathbb{R}C}$ denote the full
subcategory of $\mathsf{LCA}_{\mathfrak{A}}$ whose objects have underlying LCA
group in $\mathsf{LCA}_{\mathbb{R}C}$.
\end{definition}

So far, this only equips $\mathsf{LCA}_{\mathfrak{A},\mathbb{R}C}$ with the
structure of an additive category. Better though, let us show that the
definition agrees with another ostensibly more restrictive definition:

\begin{lemma}
\label{lemma_RC1prep}Equivalently, $\mathsf{LCA}_{\mathfrak{A},\mathbb{R}C}$
can be defined as the full subcategory of $\mathsf{LCA}_{\mathfrak{A}}$ whose
objects admit a direct sum decomposition $G\cong V\oplus C$ with $V$ a vector
$\mathfrak{A}$-module and $C$ compact in $\mathsf{LCA}_{\mathfrak{A}}$.
\end{lemma}

\begin{proof}
Suppose $G\in\mathsf{LCA}_{\mathfrak{A}}$ is as in\ Definition \ref{def_RC}.
The dual $G^{\vee}\in\left.  _{\mathfrak{A}}\mathsf{LCA}\right.  $ has
underlying LCA group of the shape $\mathbb{R}^{n}\oplus\tilde{D}$ for
$\tilde{D}$ discrete, so Lemma \ref{Lemma_KeySplittingLemma} (1) applies,
allowing us to decompose $G^{\vee}$ in $\left.  _{\mathfrak{A}}\mathsf{LCA}%
\right.  $ as a direct sum of a vector left $\mathfrak{A}$-module and a
discrete module, and dualizing back, we get an isomorphism $G\cong V\oplus C$
with $V$ a vector module and $C$ compact in $\mathsf{LCA}_{\mathfrak{A}}$. The
converse is clear.
\end{proof}

\begin{proposition}
\label{prop_RC1}Let $\mathfrak{A}$ be an arbitrary order.

\begin{enumerate}
\item The full subcategory $\mathsf{LCA}_{\mathfrak{A},\mathbb{R}%
C}\hookrightarrow\mathsf{LCA}_{\mathfrak{A}}$ is closed under extensions.

\item The full subcategory $\mathsf{LCA}_{\mathfrak{A},\mathbb{R}%
C}\hookrightarrow\mathsf{LCA}_{\mathfrak{A}}$ is left filtering; indeed
filtering by a clopen subobject.

\item The category $\mathsf{LCA}_{\mathfrak{A},\mathbb{R}C}$ is idempotent complete.
\end{enumerate}
\end{proposition}

\begin{proof}
This can be proven largely as in \cite{obloc}, just by replacing some
ingredients by their analogues in our setting. We give a sketch: (1) By our
definition of $\mathsf{LCA}_{\mathfrak{A},\mathbb{R}C}$, this is only a
condition on the underlying LCA groups. In particular, it follows from
$\mathsf{LCA}_{\mathbb{R}C}$ being closed under extensions in $\mathsf{LCA}$,
which is \cite[Lemma 3.8]{obloc} for $\mathcal{O}=\mathbb{Z}$. (2) Suppose
$f:G^{\prime}\rightarrow G$ is any morphism with $G^{\prime}\in\mathsf{LCA}%
_{\mathfrak{A},\mathbb{R}C}$. Apply Lemma \ref{lemma_RC1prep} to $G^{\prime}
$. Next, applying the right hand side sequence of Lemma
\ref{Lemma_CompactSubobjectDecomp} to $G$, we then arrive at the diagram below
on the left:%
\[%
\xymatrix{
& V \oplus C \ar@{^{(}->}[d] \\
V^{\prime} \oplus C^{\prime} \ar[r]_{f} \ar[dr]_{h} & G \ar@{->>}[d] \\
& D.
}%
\qquad\qquad%
\xymatrix{
& K \ar@{^{(}->}[d] \\
V^{\prime} \oplus C^{\prime} \ar@{-->}[ur] \ar[r]_{f} \ar[dr]_{0}
& G \ar@{->>}[d] \\
& D/{\operatorname{im}(h)}.
}%
\]
As $h$ is continuous, it necessarily maps the connected $V^{\prime}$ to zero
in the discrete $D$, and maps the compact $C^{\prime}$ to a compact subgroup
of $D$, which is necessarily finite. Now, we obtain from this the diagram
above on the right, where $K$ is simply defined as the kernel. As
$D/\operatorname*{im}(h)$ is discrete, $K$ is open. The dashed lift exists by
the universal property of kernels, and moreover $K$ can be unraveled to arise
as the extension $V\oplus C\hookrightarrow K\twoheadrightarrow
\operatorname*{im}(h)$. Since $\operatorname*{im}(h)$ is finite discrete, it
is trivially compact, and thus by part (1) of our claim, $K\in\mathsf{LCA}%
_{\mathfrak{A},\mathbb{R}C}$ and the dashed arrow is the required
factorization. For more details, see the proof of \cite[Lemma 3.9]{obloc},
which uses the same pattern. (3) Just as \cite[Lemma 3.10]{obloc} this follows
from that in part (2) we actually showed that any morphism from $\mathsf{LCA}%
_{\mathfrak{A},\mathbb{R}C}$ to an object $G\in\mathsf{LCA}_{\mathfrak{A}}$
factors over a clopen subobject $K\in\mathsf{LCA}_{\mathfrak{A},\mathbb{R}C}$,
i.e. one with discrete quotient. The argument \textit{loc. cit.} then also
works in our generality.
\end{proof}

\begin{corollary}
$\mathsf{LCA}_{\mathfrak{A},\mathbb{R}C}$ is an idempotent complete fully
exact subcategory of $\mathsf{LCA}_{\mathfrak{A}}$.
\end{corollary}

\begin{proof}
By Proposition \ref{prop_RC1} (1) it is an extension-closed subcategory of
$\mathsf{LCA}_{\mathfrak{A}}$, so by \cite[Lemma 10.20]{MR2606234} this
induces a canonical exact structure on $\mathsf{LCA}_{\mathfrak{A}%
,\mathbb{R}C}$, and moreover renders it a fully exact subcategory
\cite[Definition 10.21]{MR2606234}. Proposition \ref{prop_RC1} (3) settles
being idempotent complete.
\end{proof}

From now on, we may therefore regard $\mathsf{LCA}_{\mathfrak{A},\mathbb{R}C}
$ as an exact category.

\begin{definition}
Let $\mathsf{LCA}_{\mathfrak{A},\mathbb{R}D}$ denote the full subcategory of
$\mathsf{LCA}_{\mathfrak{A}}$ whose objects have underlying LCA group
$\mathbb{R}^{n}\oplus D$ for some $n$ and some discrete $D$.
\end{definition}

We will only work with this category in passing later (namely in the
formulation of Lemma \ref{lemma_P1}), but note that under duality it gets sent
to $\left.  _{\mathbb{R}C,\mathfrak{A}}\mathsf{LCA}^{op}\right.  $, so the
above results can all be transported to $\mathsf{LCA}_{\mathfrak{A}%
,\mathbb{R}D}$ by dualization, in the spirit of Remark
\ref{rmk_ProofsByDualization}. In particular, it is an idempotent complete
fully exact subcategory of $\mathsf{LCA}_{\mathfrak{A}}$, and the full
subcategory $\mathsf{LCA}_{\mathfrak{A},\mathbb{R}C}\hookrightarrow
\mathsf{LCA}_{\mathfrak{A}}$ is right filtering. We have no use for these
facts however, so we leave the details to the reader.

\begin{definition}
\label{def_LCAAC}Let $\mathsf{LCA}_{\mathfrak{A},C}$ (resp. $\mathsf{LCA}%
_{\mathfrak{A},D}$) be the full subcategory of $\mathsf{LCA}_{\mathfrak{A}}$
of compact (resp. discrete) objects.
\end{definition}

All these considerations carry over to left modules with the obvious
modifications. Clearly $\mathsf{LCA}_{\mathfrak{A},D}\overset{\sim
}{\longrightarrow}\mathsf{Mod}_{\mathfrak{A}}$ is an equivalence of
categories, by simply forgetting the topology and conversely equipping all
objects with the discrete topology. In particular, $\mathsf{LCA}%
_{\mathfrak{A},D}$ is an abelian category. Pontryagin duality induces an
equivalence $\mathsf{LCA}_{\mathfrak{A},C}^{op}\overset{\sim}{\longrightarrow
}\left.  _{D,\mathfrak{A}}\mathsf{LCA}\right.  $, so the compact counterpart
is also naturally an abelian category.

\begin{example}
The category $\mathsf{LCA}_{\mathfrak{A},\mathbb{R}C}$ is not abelian. Indeed,
already in $\mathsf{LCA}_{\mathbb{R}C}$ the morphism $\mathbb{R}%
\twoheadrightarrow\mathbb{T}$ has no kernel. It is an admissible epic in
$\mathsf{LCA}$, but not an admissible epic in $\mathsf{LCA}_{\mathbb{R}C}$.
\end{example}

\begin{proposition}
\label{prop_CLeftSFiltInRC}The full subcategory $\mathsf{LCA}_{\mathfrak{A}%
,C}\hookrightarrow\mathsf{LCA}_{\mathfrak{A},\mathbb{R}C}$ is left $s$-filtering.
\end{proposition}

\begin{proof}
This is proven exactly as in \cite[Lemma 3.11]{obloc}. Left filtering: Given
$V\oplus C\in\mathsf{LCA}_{\mathfrak{A},\mathbb{R}C}$ with $V$ a vector module
and $C$ compact and $a:C^{\prime}\rightarrow V\oplus C$ is any morphism with
$C^{\prime}$ compact, then the set-theoretic image $\operatorname*{im}%
\nolimits_{\mathsf{Set}}(a)$ is again a compact right $\mathfrak{A}$-module.
Since $V$ has no non-trivial compact subgroups, we obtain $\operatorname*{im}%
\nolimits_{\mathsf{Set}}(a)\subseteq C$ and therefore the exact sequence%
\[
\operatorname*{im}\nolimits_{\mathsf{Set}}(a)\hookrightarrow V\oplus
C\twoheadrightarrow V\oplus C/\operatorname*{im}\nolimits_{\mathsf{Set}}%
\]
in $\mathsf{LCA}_{\mathfrak{A}}$. As all objects in this sequence lie in
$\mathsf{LCA}_{\mathfrak{A},\mathbb{R}C}$, it follows that it is also exact in
$\mathsf{LCA}_{\mathfrak{A},\mathbb{R}C}$ since $\mathsf{LCA}_{\mathfrak{A}%
,\mathbb{R}C}$ is fully exact in $\mathsf{LCA}_{\mathfrak{A}}$. Thus, the
initial arrow $\operatorname*{im}\nolimits_{\mathsf{Set}}(a)\hookrightarrow
V\oplus C$ is an admissible monic in the category $\mathsf{LCA}_{\mathfrak{A}%
,\mathbb{R}C}$. Thus,%
\[
C\longrightarrow\operatorname*{im}\nolimits_{\mathsf{Set}}(a)\hookrightarrow
V\oplus C
\]
is the required factorization which shows that $\mathsf{LCA}_{\mathfrak{A}%
,\mathbb{R}C}\hookrightarrow\mathsf{LCA}_{\mathfrak{A}}$ is left filtering.
Left special: Briefly, if $V\oplus C\twoheadrightarrow C^{\prime}$ is an
admissible epic with $C^{\prime}\in\mathsf{LCA}_{\mathfrak{A},C}$, then we
obtain a commutative diagram%
\[%
\xymatrix{
C \ar@{^{(}->}[d] \ar[dr]_{h} \ar[drr]^{0} \\
V \oplus C \ar@{->>}[d] \ar@{->>}[r]_-{f} & C^{\prime} \ar@{->>}%
[r] & C^{\prime}/{\operatorname{im}(h)} \\
V \ar@{-->}[urr]_{b}
}%
\]
and since $\mathsf{LCA}_{\mathfrak{A},\mathbb{R}C}$ is idempotent complete,
\cite[Proposition 7.6]{MR2606234} implies that $b$ must be an admissible epic.
A topological consideration (relying only on the underlying LCA groups) can
now be carried out, exactly as in \cite[Lemma 3.11]{obloc}, contrasting that
$V\simeq\mathbb{R}^{n}$ as an LCA group, but $C^{\prime}/\operatorname*{im}%
(h)$ being compact. The idea is that $C^{\prime}/\operatorname*{im}(h) $ can
only be some torus, but this again is only possible if the kernel of $b $
contains some summand of the shape $\mathbb{Z}^{i}$ with $i>0$, but this is
impossible in $\mathsf{LCA}_{\mathbb{R}C}$. As in \textit{loc. cit.}, we
deduce that $h$ must have been an admissible epic to start with. This proves
left special.
\end{proof}

\begin{proposition}
\label{Prop_KRCEqualsKR}Let $\mathfrak{A}$ be an order in a finite-dimensional
semisimple $\mathbb{Q}$-algebra $A$. Define $A_{\mathbb{R}}:=A\otimes
_{\mathbb{Z}}\mathbb{R}$. For every localizing invariant
$K:\operatorname*{Cat}_{\infty}^{\operatorname*{ex}}\rightarrow\mathsf{A}$,
there is a canonical equivalence $K(\mathsf{LCA}_{\mathfrak{A},\mathbb{R}%
C})\overset{\sim}{\longrightarrow}K(A_{\mathbb{R}})$.
\end{proposition}

\begin{proof}
Again, we follow the pattern of \cite{obloc}. The steps are as follows: (1)
Since $\mathsf{LCA}_{\mathfrak{A},C}\hookrightarrow\mathsf{LCA}_{\mathfrak{A}%
,\mathbb{R}C}$ is left $s$-filtering by\ Proposition \ref{prop_CLeftSFiltInRC}%
, we get an exact sequence of exact categories%
\[
\mathsf{LCA}_{\mathfrak{A},C}\hookrightarrow\mathsf{LCA}_{\mathfrak{A}%
,\mathbb{R}C}\twoheadrightarrow\mathsf{LCA}_{\mathfrak{A},\mathbb{R}%
C}/\mathsf{LCA}_{\mathfrak{A},C}\text{.}%
\]
We recall that the quotient exact category $\mathsf{LCA}_{\mathfrak{A}%
,\mathbb{R}C}/\mathsf{LCA}_{\mathfrak{A},C}$ arises as the localization
$\mathsf{LCA}_{\mathfrak{A},\mathbb{R}C}[\Sigma_{e}^{-1}]$, where $\Sigma
_{\in}$ is the collection of admissible epics with kernel in the subcategory
$\mathsf{LCA}_{\mathfrak{A},C}$. Following Schlichting \cite{MR2079996} this
is again an exact category.\newline(2) We claim that there is an exact
equivalence of exact categories%
\[
\Psi:\mathsf{Mod}_{A_{\mathbb{R}},fg}\longrightarrow\mathsf{LCA}%
_{\mathfrak{A},\mathbb{R}C}/\mathsf{LCA}_{\mathfrak{A},C}\text{.}%
\]
We define the functor $\Psi$ by sending a finitely generated $A_{\mathbb{R}}
$-module to its underlying real vector space with its natural real vector
space topology. To this end, note that $A$ is a finite-dimensional
$\mathbb{Q}$-algebra, so $A_{\mathbb{R}}$ is a finite-dimensional $\mathbb{R}%
$-algebra, and hence all finitely generated modules over it are also
finite-dimensional over the reals. Any such vector space has a canonical real
topology and it is of course locally compact. Moreover, every $\mathbb{R}%
$-linear map is automatically continuous with respect to this topology. This
shows that $\Psi$ is a functor. It is now easy to check that it is additive
and moreover exact. Next, we claim that $\Psi$ is fully faithful. This is also
easy: The essential image in $\mathsf{LCA}_{\mathfrak{A},\mathbb{R}C}$ (i.e.
without quotienting out $\mathsf{LCA}_{\mathfrak{A},C}$) consists only of
vector modules, so any continuous morphism between them is (a) an abelian
group map, (b) by divisibility then also a $\mathbb{Q}$-vector space map, and
(c) since it is also continuous, a density argument shows that it is even an
$\mathbb{R}$-vector space map. Thus, the functor $\mathsf{Mod}_{A_{\mathbb{R}%
},fg}\longrightarrow\mathsf{LCA}_{\mathfrak{A},\mathbb{R}C}$ is fully
faithful. Quotienting out $\mathsf{LCA}_{\mathfrak{A},C}$ does not harm this
as we invert admissible epics with kernel in $\mathsf{LCA}_{\mathfrak{A},C}$,
but a real vector space has no non-trivial compact subgroups. Hence, $\Psi$ is
also fully faithful. Finally, we claim that $\Psi$ is essentially surjective:
Given any $G\in\mathsf{LCA}_{\mathfrak{A},\mathbb{R}C}$, Lemma
\ref{lemma_RC1prep} yields the exact sequence, $C\hookrightarrow G\overset
{q}{\twoheadrightarrow}V$, and after inverting $\Sigma_{e}$, $q$ becomes an
isomorphism, and $V$ lies in the image of $\Psi$, so $G$ lies in the essential
image. Being exact, fully faithful and essentially surjective, $\Psi$ is an
exact equivalence of exact categories.\newline(3) As $K$ is localizing,
Schlichting's Localization Theorem (in the formulation of \cite[Theorem
4.1]{obloc}) yields a fiber sequence in $\mathsf{A}$ of the shape%
\[
K(\mathsf{LCA}_{\mathfrak{A},C})\longrightarrow K(\mathsf{LCA}_{\mathfrak{A}%
,\mathbb{R}C})\overset{a}{\longrightarrow}K(\mathsf{Mod}_{A_{\mathbb{R}}%
,fg})\text{,}%
\]
where we have used Step (2) to identify the third term. On the other hand, by
the Eilenberg swindle, $K(\mathsf{LCA}_{\mathfrak{A},C})=0$, cf. \cite[Lemma
4.2]{obloc}, since infinite products of compact spaces are compact by
Tychonoff's Theorem. Thus, $a$ must be an equivalence in $\mathsf{A}$.
Finally, since $A_{\mathbb{R}}$ is a semisimple algebra, every right
$A_{\mathbb{R}}$-module is projective (see \S \ref{subsect_Basics}). Hence,
$K(\mathsf{Mod}_{A_{\mathbb{R}},fg})\overset{\sim}{\longrightarrow
}K(A_{\mathbb{R}})$. This finishes the proof.
\end{proof}

\section{\label{sect_CGPiece}The compactly generated piece}

\begin{proposition}
\label{prop_CGToRC}For every localizing invariant $K:\operatorname*{Cat}%
_{\infty}^{\operatorname*{ex}}\rightarrow\mathsf{A}$, there is a canonical
equivalence $K(\mathsf{LCA}_{\mathfrak{A},\mathbb{R}C})\overset{\sim
}{\longrightarrow}K(\mathsf{LCA}_{\mathfrak{A},cg})$.
\end{proposition}

This generalizes a corresponding result for $\mathfrak{A}:=\mathcal{O}$ the
maximal order in a number field, which appears as the first step in the proof
of \cite[Lemma 4.3]{obloc}. The proof idea\textit{\ loc. cit.} runs as
follows: Given any $G\in\mathsf{LCA}_{\mathcal{O},cg}$, write it as a direct
sum of a vector module, something with underlying LCA group $\mathbb{Z}^{m}$,
and a compact summand. Then use that for all these three types of summands,
there exists an injective resolution of length at most $1$, and thus for all
of $G$, some length $1$ injective resolution%
\[
G\overset{\sim}{\longrightarrow}\left[  I^{0}\longrightarrow I^{1}\right]
_{0,1}\text{.}%
\]
Then, by the classification of injectives, $I^{0}$ and $I^{1}$ would live in
the full subcategory $\mathsf{LCA}_{\mathcal{O},\mathbb{R}C}$. In our present
setting, this kind of argument fails in two ways:\ Firstly, decomposing
$G\in\mathsf{LCA}_{\mathfrak{A},cg}$ into a direct sum as above is
\textit{not} generally possible (Example
\ref{example_CyclicGroupCannotSplitOffTorus}). But even if we could circumvent
this issue, there is a further problem. Suppose $\mathfrak{A}:=\mathbb{Z}%
[C_{n}]$, where $C_{n}$ is the cyclic group of order $n\geq2$. In
$\mathsf{Mod}_{\mathfrak{A}}$ the trivial $C_{n}$-module $\mathbb{Z}$ has
infinite projective dimension (indeed, Equation \ref{lta_1} describes a
resolution and from its periodicity one gets the periodicity of
$\operatorname*{Ext}_{\mathbb{Z}[C_{n}]}^{i}(\mathbb{Z},\mathbb{Z})$ in $i$
for $i\geq1$, but if $\mathbb{Z}$ admitted a finite projective resolution,
this would be impossible). This does not strictly rule out the possibility
that in the larger category $\mathsf{LCA}_{\mathfrak{A}}$ a finite resolution
by projectives might exist. However, we see no evidence for this. By duality,
its dual $\mathbb{Z}^{\vee}$, which lies in $\mathsf{LCA}_{\mathfrak{A},cg}$,
should \textit{not} be expected to possess a finite injective resolution.

We can fix both problems:

\begin{itemize}
\item A more careful study of the topology shows that we can circumvent
splitting $G$ as direct sum entirely.

\item The possible lack of finite injective resolutions can be salvaged by a
little trick: homologically, it suffices to resolve objects of $\mathsf{LCA}%
_{\mathfrak{A},cg}$ by any resolutions coming from $\mathsf{LCA}%
_{\mathfrak{A},\mathbb{R}C}$, so these resolutions need \textit{not} consist
of injectives. So we will drag the pieces of possibly infinite injective
dimension into a compact module, as all compact modules lie in $\mathsf{LCA}%
_{\mathfrak{A},\mathbb{R}C}$, injective or not.
\end{itemize}

Note that we have already seen that the value of $K(\mathsf{LCA}%
_{\mathfrak{A},\mathbb{R}C})$ is not affected by the part coming from compact
modules since it can be killed off by an Eilenberg swindle (Proposition
\ref{Prop_KRCEqualsKR}), so in a sense the above procedure drags these
unwanted pieces of too big injective dimension to zero.

\begin{lemma}
\label{lemma_P1}Suppose $\mathfrak{A}$ is an arbitrary order.

\begin{enumerate}
\item Suppose $G\in\mathsf{LCA}_{\mathfrak{A},nss}$. Then there exists an
exact sequence%
\[
Y_{2}\hookrightarrow Y_{1}\twoheadrightarrow G
\]
with $Y_{1},Y_{2}\in\mathsf{LCA}_{\mathfrak{A},\mathbb{R}D}$.

\item Suppose%
\[
X\hookrightarrow Y\twoheadrightarrow Y^{\prime}%
\]
is an exact sequence in $\mathsf{LCA}_{\mathfrak{A},nss}$ with $Y,Y^{\prime
}\in\mathsf{LCA}_{\mathfrak{A},\mathbb{R}D}$. Then it follows that
$X\in\mathsf{LCA}_{\mathfrak{A},\mathbb{R}D}$.
\end{enumerate}
\end{lemma}

Note that since in (1) all terms of the sequence lie in $\mathsf{LCA}%
_{\mathfrak{A},nss}$, it makes no difference whether we speak of the sequence
being exact in the category $\mathsf{LCA}_{\mathfrak{A}}$ or in $\mathsf{LCA}%
_{\mathfrak{A},nss}$; both concepts agree since $\mathsf{LCA}_{\mathfrak{A}%
,nss}$ is fully exact in $\mathsf{LCA}_{\mathfrak{A}}$.

\begin{proof}
(Step 1) We begin by addressing the first claim. Since $G\in\mathsf{LCA}%
_{\mathfrak{A},nss}$, its dual is compactly generated, i.e. $G^{\vee}%
\in\left.  _{cg,\mathfrak{A}}\mathsf{LCA}\right.  $, Lemma
\ref{lemma_CGExchangesWithNSS}. By Lemma \ref{Lemma_CompactPartOfCG} there is
an exact sequence $C\hookrightarrow G^{\vee}\twoheadrightarrow W$ in $\left.
_{\mathfrak{A}}\mathsf{LCA}\right.  $ with $C$ compact and $W$ having
underlying LCA group $\mathbb{R}^{n}\oplus\mathbb{Z}^{m}$. Dualizing back, we
get an exact sequence $H\hookrightarrow G\twoheadrightarrow D$ with $H$ having
underlying LCA group $\mathbb{R}^{n}\oplus\mathbb{T}^{m}$ and $D$ discrete. We
will set up a commutative diagram%
\begin{equation}%
\xymatrix{
V \ar@{->>}[d]_{q} & & P \ar@{-->}[dl]_{f^{\prime}} \ar@{->>}[d]^{f} \\
H \ar@{^{(}->}[r]_{j} & G \ar@{->>}[r]_{w} & D
}%
\label{l_B6}%
\end{equation}
as follows: (1) Since $D$ is discrete, we may pick a projective cover $P$,
which is still discrete, by Lemma \ref{Lemma_ProjCovers}, so $P\in
\mathsf{LCA}_{\mathfrak{A},\mathbb{R}D}$. Being projective, the lift
$f^{\prime}$ along the admissible epic $w$ exists. (2) As $H$ has underlying
LCA group $\mathbb{R}^{n}\oplus\mathbb{T}^{m}$, by Lemma
\ref{Lemma_RZOrRTHulls} (2) there exists a vector $\mathfrak{A}$-module $V$
mapping to $H$ through the admissible epic $q$. This sets up the diagram.

From this, we obtain a morphism $h:=f^{\prime}+jq$,%
\begin{equation}
V\oplus P\overset{h}{\longrightarrow}G\text{.}\label{l_B7}%
\end{equation}
(Step 2) A simple diagram chase shows that $h$ is a surjective map. Since
$f^{\prime}$ and $jq$ are continuous $\mathfrak{A}$-module homomorphisms, so
is their sum. Next, we claim that $h$ is an open map. The underlying
topological space of $V\oplus P$ is just the product $V\times P$ in
$\mathsf{Top}$, so the Cartesian opens $\{U_{1}\times U_{2}\}_{U_{1},U_{2}}$
with $U_{1}\subseteq V$ open and $U_{2}\subseteq P$ open form a basis of the
topology. Hence, if $U\subseteq V\times P$ is an arbitrary open, we may write
it as%
\[
U=\bigcup_{i\in\mathcal{I}}U_{1,i}\times U_{2,i}\qquad\text{for suitable
opens}\qquad U_{1,i}\subseteq V\text{, }U_{2,i}\subset P\text{,}%
\]
and $\mathcal{I}$ some index set. Then%
\[
h\left(  U\right)  =h\left(  \bigcup_{i\in\mathcal{I}}U_{1,i}\times
U_{2,i}\right)  =\bigcup_{i\in\mathcal{I}}h(U_{1,i}\times U_{2,i}%
)=\bigcup_{i\in\mathcal{I}}\left(  jq(U_{1,i})+f^{\prime}(U_{2,i})\right)
\text{.}%
\]
Now, the map $j$ is open (since the quotient $D$ by $j$ in Diagram \ref{l_B6}
is discrete, see \cite[Proposition 14]{MR0442141}) and $q$ is an admissible
epic and therefore also open. Thus, $jq$ is an open map, and thus
$jq(U_{1,i})$ is an open set. It follows from Lemma
\ref{lemma_SumOpenWithArbitraryIsOpen} that all sets $jq(U_{1,i})+f^{\prime
}(U_{2,i})$ are open, and thus $h(U)$, being a union of open sets. As $U$ was
arbitrary, it follows that $h$ is an open map. As $h=f^{\prime}+jq$ is open
and surjective, it is an admissible epic. Hence, we can promote Equation
\ref{l_B7} to an exact sequence%
\[
K\hookrightarrow V\oplus P\overset{h}{\twoheadrightarrow}G
\]
in $\mathsf{LCA}_{\mathfrak{A}}$, where $K$ is plainly defined as the kernel,
so $K\in\mathsf{LCA}_{\mathfrak{A}}$. As the underlying LCA group of $V\oplus
P$ is $\mathbb{R}^{\ell}\oplus D^{\prime}$ for some $\ell\in\mathbb{Z}_{\geq
0}$ and $D^{\prime}$ discrete (i.e. $V\oplus P\in\mathsf{LCA}_{\mathfrak{A}%
,\mathbb{R}D}$), the closed subgroup $K$ must also lie in $\mathsf{LCA}%
_{\mathfrak{A},\mathbb{R}D}$ by \cite{MR0442141}, Corollary 2 to Theorem 7,
combined with the following Remark \textit{loc. cit.} to exclude the
possibility of a torus factor. Hence, letting $Y_{2}:=K$ and $Y_{1}:=V\oplus
P$ proves our claim. Note that the same argument which showed $K\in
\mathsf{LCA}_{\mathfrak{A},\mathbb{R}D}$ also settles our second claim.
\end{proof}

We actually need the dual formulation:

\begin{lemma}
\label{lemma_P2}Suppose $\mathfrak{A}$ is an arbitrary order.

\begin{enumerate}
\item Suppose $G\in\mathsf{LCA}_{\mathfrak{A},cg}$. Then there exists an exact
sequence%
\[
G\hookrightarrow Y^{1}\twoheadrightarrow Y^{2}%
\]
with $Y^{1},Y^{2}\in\mathsf{LCA}_{\mathfrak{A},\mathbb{R}C}$.

\item Suppose%
\[
Y^{\prime}\hookrightarrow Y\twoheadrightarrow X
\]
is an exact sequence in $\mathsf{LCA}_{\mathfrak{A},cg}$ with $Y,Y^{\prime}%
\in\mathsf{LCA}_{\mathfrak{A},\mathbb{R}C}$. Then it follows that
$X\in\mathsf{LCA}_{\mathfrak{A},\mathbb{R}C}$. In other words: The full
subcategory $\mathsf{LCA}_{\mathfrak{A},\mathbb{R}C}$ is closed under
cokernels in $\mathsf{LCA}_{\mathfrak{A},cg}$.
\end{enumerate}
\end{lemma}

\begin{proof}
For $G\in\mathsf{LCA}_{\mathfrak{A},cg}$ we have $G^{\vee}\in\mathsf{LCA}%
_{\mathfrak{A}^{op},nss}$. Apply Lemma \ref{lemma_P1} to $G^{\vee}$, then
dualize back and use Lemma \ref{lemma_CGExchangesWithNSS}.
\end{proof}

\begin{elaboration}
The above proof is quick, but of course one can also prove this directly by
dualizing the proof of Lemma \ref{lemma_P1}. The topological argument is a
little different in this case, and in some ways more involved. We sketch
it:\ Firstly, given $G\in\mathsf{LCA}_{\mathfrak{A},cg}$, by Lemma
\ref{Lemma_CompactPartOfCG} there is an exact sequence $C\hookrightarrow
G\twoheadrightarrow W$ in $\mathsf{LCA}_{\mathfrak{A}}$ with compact and $W$
having underlying LCA group of the shape $\mathbb{R}^{n}\oplus\mathbb{Z}^{m}
$. By Lemma \ref{Lemma_KeySplittingLemma} (1) one may isolate a vector
$\mathfrak{A}$-module summand $V$ in $W$, and being projective (Theorem
\ref{thm_VectorModulesAreInjectiveAndBijective}), this quotient object of $G$
splits off as a direct summand. Thus, we may write $G\cong G^{\prime}\oplus
V$. This decomposes $G^{\prime}$ as $C\hookrightarrow G^{\prime}%
\twoheadrightarrow W^{\prime}$ with $W^{\prime}$ having underlying group
$\mathbb{Z}^{m}$. By Lemma \ref{Lemma_ProjCovers} the compact $C$ has a
compact injective hull $I$, and by Lemma \ref{Lemma_RZOrRTHulls} there exists
an admissible monic $W^{\prime}\hookrightarrow V^{\prime}$ with $V^{\prime}$ a
vector module. We obtain the diagram%
\[%
\xymatrix{
C \ar@{^{(}->}[r] \ar@{^{(}->}[d]_{i} & G^{\prime} \ar@{->>}[r]^{q} \ar@
{-->}[dl]^{e} & W^{\prime} \ar@{^{(}->}[d]^{f} \\
I & & V^{\prime}\text{,}
}%
\]
analogous to Diagram \ref{l_B6}. This time we wish to prove that the sum
$h:=e+fq:G^{\prime}\rightarrow I\oplus V^{\prime}$ is a closed map. To this
end, note that, as a topological space, $G^{\prime}$ is the disjoint union
$G^{\prime}=\coprod_{w\in W^{\prime}}(C+\tilde{w})$, where $\tilde{w}$ denotes
an arbitrary (but fixed) lift of $w\in W^{\prime}$ to $G^{\prime}$. Thus, any
closed subset $T\subseteq G^{\prime}$ has the shape $T=\coprod\limits_{w}%
(T_{w}+\tilde{w})$ with $T_{w}\subseteq C$ closed. Hence,%
\begin{align*}
h(T)  & =\bigcup_{w\in W^{\prime}}h(T_{w}+\tilde{w})=\bigcup_{w}%
e(T_{w})+e(\tilde{w})+fq(T_{w})+fq(\tilde{w})\\
& =\bigcup_{w\in W^{\prime}}e(T_{w})+e(\tilde{w})+f(w)=\coprod_{w\in
W^{\prime}}\left(  e(T_{w}+\tilde{w})\right)  \times\{f(w)\}
\end{align*}
We have $T_{w}\subseteq C$, but $e\mid_{C}=i$, which is a closed map since it
is an admissible monic, so $e(T_{w})$ is closed, and therefore also
$e(T_{w}+\tilde{w})$ as translation is a homeomorphism. Note that,
topologically, $f$ embeds the lattice $W^{\prime}\simeq\mathbb{Z}^{m}$ into a
$\mathbb{R}^{m}$, so each value of $f(w)$ is attained only once, explaining
why the possibly countable union in the second line is a \textit{disjoint}
union. Each $e(T_{w}+\tilde{w})$ being closed, this means that $h(T)$ is
closed in $\coprod_{w\in W^{\prime}}I\times\{w\}=I\times W^{\prime}$, which in
turn is closed in $I\times V^{\prime}$, which is the underlying topological
space of $I\oplus V^{\prime}$.
\end{elaboration}

Armed with the previous lemma, we can prove the proposition.

\begin{proof}
[Proof of Proposition \ref{prop_CGToRC}]We use \cite[Theorem 12.1]{MR1421815}.
We note that the conditions C1 and C2 \textit{loc. cit.} hold (or more
specifically the conditions stated right below the axioms \textit{loc. cit.}),
thanks to Lemma \ref{lemma_P2}. Moreover, since the resolution provided by our
lemma is finite, this gives a triangulated equivalence of the bounded derived
categories $\mathcal{D}^{b}(\mathsf{LCA}_{\mathfrak{A},\mathbb{R}C}%
)\overset{\sim}{\longrightarrow}\mathcal{D}^{b}(\mathsf{LCA}_{\mathfrak{A}%
,cg})$. This suffices to know that the inclusion functor $\mathsf{LCA}%
_{\mathfrak{A},\mathbb{R}C}\hookrightarrow\mathsf{LCA}_{\mathfrak{A},cg}$
induces an equivalence of stable $\infty$-categories, \cite[Corollary
5.11]{MR3070515}.
\end{proof}

\begin{remark}
\label{rmk_on_global_dims}Since the topic of complications arising from high
global dimension has come up, let us quickly recapitulate the situation:

\begin{enumerate}
\item Every hereditary order $\mathfrak{A}$ has global dimension $\leq1$, so
these are regular.

\item Every maximal order $\mathfrak{A}$ is hereditary, so these are covered
by (1). See \cite[(26.12)\ Theorem]{MR1038525} for a proof.

\item Group rings $\mathbb{Z}[G]$ unfortunately have infinite global dimension
unless $G=1$. See Example \ref{example_FinGroup_InfGlobalDim}.

\item Between global dimension $\leq1$ and $\infty$, all other values are also
possible. Jategaonkar \cite[Theorem 2, and Remark]{MR0309988} gives (a) a
construction of $\mathbb{Z}_{(p)}$-orders in the matrix ring $M_{n}%
(\mathbb{Q})$ of global dimension $n$ for any prescribed $n\geq2$, and (b) on
the last pages of the paper gives an example of orders $\mathfrak{A}%
_{1},\mathfrak{A}_{2}\subset M_{2n+1}(\mathbb{Q})$ with%
\[
\operatorname*{gldim}\mathfrak{A}_{1}-\operatorname*{gldim}\mathfrak{A}_{2}=n
\]
for any chosen $n\geq1$. As a little elaboration: Firstly, observe that if
$\mathfrak{A}\subset A$ is an arbitrary order, then%
\begin{equation}
\operatorname*{gldim}\mathfrak{A}=\sup_{p}\left(  \operatorname*{gldim}%
\mathfrak{A}\otimes\mathbb{Z}_{(p)}\right)  \text{,}\label{l_Ja1}%
\end{equation}
where $p$ runs through all primes, \cite[Corollary to Proposition
2.6]{MR0117252}. For example, take%
\[
\Gamma_{3}:=%
\begin{pmatrix}
\mathbb{Z} & \mathbb{Z} & \mathbb{Z}\\
5\mathbb{Z} & \mathbb{Z} & \mathbb{Z}\\
5^{2}\mathbb{Z} & 5\mathbb{Z} & \mathbb{Z}%
\end{pmatrix}
\subset M_{3}(\mathbb{Q})\text{.}%
\]
We have $\operatorname*{gldim}\Gamma_{3}\otimes\mathbb{Z}_{(5)}=2$ (this is
Jategaonkar's example, \cite[Example after Theorem 2]{MR0309988}, over the
discrete valuation ring $\mathbb{Z}_{(5)}$).\ For all primes $p\neq5$, we have
$\Gamma_{3}\otimes\mathbb{Z}_{(p)}=M_{3}(\mathbb{Z}_{(p)})$. By Morita
invariance, the module category is equivalent to the one of $\mathbb{Z}_{(p)}%
$, so $\operatorname*{gldim}\Gamma_{3}\otimes\mathbb{Z}_{(p)}=1$. Finally, by
Equation \ref{l_Ja1}, this shows that $\operatorname*{gldim}\mathfrak{A}=2$.
In a similar way, all the examples in Jategaonkar's paper can be promoted from
the DVR situation to $\mathbb{Z}$-orders.

\item All orders of finite global dimension are of course regular, so the
methods of this paper apply to them. However, we must admit that we are not
aware of any \textquotedblleft real life\textquotedblright\ case study of the
ETNC where orders of finite global dimension $>1$ have played a r\^{o}le.
\end{enumerate}
\end{remark}

\section{Proof of the main theorem}

\begin{proposition}
\label{prop_PrepareMain}Let $A$ be any finite-dimensional semisimple
$\mathbb{Q}$-algebra and $\mathfrak{A}\subseteq A$ an order (not necessarily
regular). Suppose $K:\operatorname*{Cat}_{\infty}^{\operatorname*{ex}%
}\rightarrow\mathsf{A}$ is a localizing invariant with values in $\mathsf{A}$.
Then there is a fiber sequence%
\[
K(\mathsf{Mod}_{\mathfrak{A},fg})\overset{g}{\longrightarrow}K(\mathsf{LCA}%
_{\mathfrak{A},cg})\overset{h}{\longrightarrow}K(\mathsf{LCA}_{\mathfrak{A}})
\]
in $\mathsf{A}$. Here the map $g$ is induced from the exact functor sending a
finitely generated right $\mathfrak{A}$-module to itself, equipped with the
discrete topology. The map $h$ is induced from the inclusion $\mathsf{LCA}%
_{\mathfrak{A},cg}\hookrightarrow\mathsf{LCA}_{\mathfrak{A}}$.
\end{proposition}

\begin{proof}
(Step 1) We construct the following commutative diagram
\begin{equation}%
\xymatrix{
K({\mathsf{Mod}_{{\mathfrak{A}},fg}}) \ar[r] \ar[d]_{g} & K({\mathsf
{Mod}_{\mathfrak{A}}}) \ar[r] \ar[d] & K({{\mathsf{Mod}_{\mathfrak{A}}}%
}/{{\mathsf{Mod}_{{\mathfrak{A}},fg}}}) \ar[d]^{\Phi} \\
K(\mathsf{LCA}_{\mathfrak{A},cg}) \ar[r] & K(\mathsf{LCA}_{\mathfrak{A}}%
) \ar[r] & K({\mathsf{LCA}_{\mathfrak{A}}}/{\mathsf{LCA}_{\mathfrak{A},cg}})
}%
\label{l_B3}%
\end{equation}
as follows: (a) We leave it to the reader to check that $\mathsf{Mod}%
_{\mathfrak{A},fg}$ is left $s$-filtering in $\mathsf{Mod}_{\mathfrak{A}}$.
This is indeed true for all associative unital rings, and amounts to showing
that the image of each finitely generated right module is again finitely
generated, as well as that every surjection onto a finitely generated module
can be restricted to a surjection originating from a finitely generated
submodule of the source. (b) Being left $s$-filtering, we get the exact
sequence of exact categories%
\begin{equation}
\mathsf{Mod}_{\mathfrak{A},fg}\hookrightarrow\mathsf{Mod}_{\mathfrak{A}%
}\twoheadrightarrow\mathsf{Mod}_{\mathfrak{A}}/\mathsf{Mod}_{\mathfrak{A}%
,fg}\text{.}\label{l_B1}%
\end{equation}
(c) By Proposition \ref{Prop_CGIsLeftSFiltering} we also know that compactly
generated $\mathfrak{A}$-modules are left $s$-filtering in $\mathsf{LCA}%
_{\mathfrak{A}}$. Correspondingly, we get an exact sequence of exact
categories%
\begin{equation}
\mathsf{LCA}_{\mathfrak{A},cg}\hookrightarrow\mathsf{LCA}_{\mathfrak{A}%
}\twoheadrightarrow\mathsf{LCA}_{\mathfrak{A}}/\mathsf{LCA}_{\mathfrak{A}%
,cg}\text{.}\label{l_B2}%
\end{equation}
Next, we construct a morphism from the sequence in\ Equation \ref{l_B1} to the
one in Equation \ref{l_B2}, given by exact functors. Firstly, $\mathsf{Mod}%
_{\mathfrak{A},fg}\rightarrow\mathsf{LCA}_{\mathfrak{A},cg}$ just sends a
finitely generated $\mathfrak{A}$-module to itself with the discrete topology.
By Lemma \ref{Lemma_Elem1} the underlying abelian group is then finitely
generated, and thus indeed compactly generated as an LCA group. Since
everything carries the discrete topology, we do not need to worry about
continuity, and moreover the functor is clearly exact (in $\mathsf{LCA} $ a
sequence of discrete groups is exact if and only it is exact plainly as
abelian groups). Secondly, $\mathsf{Mod}_{\mathfrak{A}}\rightarrow
\mathsf{LCA}_{\mathfrak{A}}$ is defined the same way, just without the finite
generation condition. Thirdly, we get an exact functor induced to the quotient
exact categories%
\begin{equation}
\Phi:\mathsf{Mod}_{\mathfrak{A}}/\mathsf{Mod}_{\mathfrak{A},fg}\longrightarrow
\mathsf{LCA}_{\mathfrak{A}}/\mathsf{LCA}_{\mathfrak{A},cg}\text{.}\label{l_B4}%
\end{equation}
It is exact by construction. Now, we apply the invariant $K$ and arrive at
Diagram \ref{l_B3} as desired. Both rows are fiber sequences by Schlichting's
Localization Theorem, see \cite[Theorem 4.1]{obloc}.\newline(Step 2) Next, we
claim that $\Phi$ is an exact equivalence of exact categories. The functor is
essentially surjective: Given any $G$ in $\mathsf{LCA}_{\mathfrak{A}}$, Lemma
\ref{Lemma_CompactSubobjectDecomp} produces an exact sequence%
\[
V\oplus C\hookrightarrow G\overset{q}{\twoheadrightarrow}D\qquad
\text{in}\qquad\mathsf{LCA}_{\mathfrak{A}}%
\]
with $V$ a vector $\mathfrak{A}$-module, $C$ compact and $D$ discrete. We note
that $V\oplus C$ is compactly generated, so in the quotient category
$\mathsf{LCA}_{\mathfrak{A}}/\mathsf{LCA}_{\mathfrak{A},cg}$ the morphism $q$
becomes an isomorphism. However, $D$ clearly lies in the strict image of
$\Phi$ since it is discrete. Next, since all modules in the strict image of
the functor are discrete, there is no difference between continuous or just
algebraic right $\mathfrak{A}$-module homomorphisms. In particular, since a
discrete module in $\mathsf{LCA}_{\mathfrak{A}}$ is compactly generated if and
only if it is finitely generated as an abelian group, and thus by Lemma
\ref{Lemma_Elem1} if and only if it is finitely generated over $\mathfrak{A}$,
we observe that the categorical quotients on the left and right in Equation
\ref{l_B4} invert the same morphisms. Thus, $\Phi$ is fully faithful. Being
exact, essentially surjective, and fully faithful, $\Phi$ is an exact
equivalence of exact categories. It follows that $K(\Phi)$ is an equivalence
in $\mathsf{A}$. This shows that the right arrow in Diagram \ref{l_B3} is an
equivalence and thus the left square is bi-Cartesian in $\mathsf{A}$%
.\newline(Step 3) Next, $\mathsf{Mod}_{\mathfrak{A}}$ is closed under
countable coproducts, so $K(\mathsf{Mod}_{\mathfrak{A}})=0$ by the\ Eilenberg
swindle, \cite[Lemma 4.2]{obloc}. Thus, the left bi-Cartesian square in
Diagram \ref{l_B3} itself pins down a contraction, and the lower left three
terms give a fiber sequence%
\begin{equation}
K(\mathsf{Mod}_{\mathfrak{A},fg})\overset{g}{\longrightarrow}K(\mathsf{LCA}%
_{\mathfrak{A},cg})\longrightarrow K(\mathsf{LCA}_{\mathfrak{A}}%
)\text{,}\label{l_Bb1}%
\end{equation}
where $g$ sends a finitely generated $\mathfrak{A}$-module to itself with the
discrete topology. This is exactly the statement which we had claimed.
\end{proof}

We are ready to prove the main result of the paper.

\begin{theorem}
\label{thm_Main_b}Let $A$ be any finite-dimensional semisimple $\mathbb{Q}%
$-algebra and $\mathfrak{A}\subseteq A$ a regular order. Write $A_{\mathbb{R}%
}:=A\otimes_{\mathbb{Q}}\mathbb{R}$. Suppose $K:\operatorname*{Cat}_{\infty
}^{\operatorname*{ex}}\rightarrow\mathsf{A}$ is a localizing invariant with
values in $\mathsf{A}$. Then there is a fiber sequence%
\[
K(\mathfrak{A})\overset{i}{\longrightarrow}K(A_{\mathbb{R}})\overset
{j}{\longrightarrow}K(\mathsf{LCA}_{\mathfrak{A}})
\]
in $\mathsf{A}$. The map $i$ is induced from the exact functor $M\mapsto
M\otimes_{\mathfrak{A}}A_{\mathbb{R}}$, and the map $j$ is induced from the
functor interpreting a finitely generated projective right $A_{\mathbb{R}}%
$-module as its underlying LCA group (using the real vector space topology),
along with the right action coming from $\mathfrak{A}\subset A_{\mathbb{R}}$.
\end{theorem}

\begin{proof}
(Step 1) As in \cite{MR1884523}, we write $\operatorname*{PMod}(R)$ for the
exact category of finitely generated projective right $R$-modules. Note that
$A_{\mathbb{R}}$ is a projective $\mathfrak{A}$-module (Lemma
\ref{lemma_VectorModulesAreAlgebraicallyInjectiveAndProjective}), and thus in
particular flat over $\mathfrak{A}$. Hence, the tensor functor%
\[
\operatorname*{PMod}(\mathfrak{A})\rightarrow\operatorname*{PMod}%
(A_{\mathbb{R}}),\qquad\qquad M\mapsto M\otimes_{\mathfrak{A}}A_{\mathbb{R}}%
\]
is indeed exact. Moreover, the output is a finitely generated $A_{\mathbb{R}}
$-module, but since $A$ is semisimple over $\mathbb{Q}$, so is its base change
$A_{\mathbb{R}}$. Hence, all modules are projective, so the functor really
goes to $\operatorname*{PMod}(A_{\mathbb{R}})$. It follows that the morphism
$K(\mathfrak{A})\overset{i}{\longrightarrow}K(A_{\mathbb{R}})$ exists. The
inclusion $\operatorname*{PMod}(R)\subseteq\mathsf{Mod}_{\mathfrak{A},fg}$ is
an exact functor, and the induced map%
\begin{equation}
K(\mathfrak{A})=_{\operatorname*{def}}\left.  K(\operatorname*{PMod}%
(\mathfrak{A}))\right.  \overset{\sim}{\longrightarrow}K(\mathsf{Mod}%
_{\mathfrak{A},fg})\label{l_B3a1}%
\end{equation}
is an equivalence. This can be seen as follows: By assumption $\mathfrak{A}$
is a regular ring, i.e. every finitely generated right $\mathfrak{A}$-module
admits a finite projective resolution. Thus, we may apply (the categorical
dual of) \cite[Theorem 12.1]{MR1421815}. The required Conditions C1, C2
\textit{loc. cit.} are satisfied, see (the categorical dual of) \cite[Example
12.2]{MR1421815} and $\mathsf{Mod}_{\mathfrak{A},fg}$ clearly has enough
projectives. Invoking this theorem, we deduce that $\mathcal{D}^{b}%
(\operatorname*{PMod}(\mathfrak{A}))\overset{\sim}{\longrightarrow}%
\mathcal{D}^{b}(\mathsf{Mod}_{\mathfrak{A},fg})$ is a triangulated
equivalence, and thus the functor induces an equivalence of stable $\infty
$-categories, \cite[Corollary 5.11]{MR3070515}. This confirms that the map in
Equation \ref{l_B3a1} is indeed an equivalence. (Step 2) Next, we invoke
Proposition \ref{prop_PrepareMain} so that we have the fiber sequence%
\[
K(\mathsf{Mod}_{\mathfrak{A},fg})\overset{g}{\longrightarrow}K(\mathsf{LCA}%
_{\mathfrak{A},cg})\longrightarrow K(\mathsf{LCA}_{\mathfrak{A}})
\]
in $\mathsf{A}$. For the next step, we shall need an alternative description
of $g$. For an exact category $\mathsf{C}$, we write $\mathcal{E}\mathsf{C}$
for the exact category of exact sequences, \cite[Exercise 3.9]{MR2606234}. We
set up a functor%
\begin{equation}
p:\operatorname*{PMod}(\mathfrak{A})\longrightarrow\mathcal{E}\mathsf{LCA}%
_{\mathfrak{A},cg}\label{l_Ba1}%
\end{equation}
by%
\[
M\mapsto\left[  M\hookrightarrow M\otimes_{\mathfrak{A}}A_{\mathbb{R}%
}\twoheadrightarrow(M\otimes_{\mathfrak{A}}A_{\mathbb{R}})/M\right]  \text{,}%
\]
where (a) on the left $M$ is regarded with the discrete topology, (b)
$M\otimes_{\mathfrak{A}}A_{\mathbb{R}}$ is equipped with the locally compact
topology coming from the finite-dimensional real vector space structure, (c)
$(M\otimes_{\mathfrak{A}}A_{\mathbb{R}})/M$ is equipped with the quotient
topology making the sequence exact. The functor $p$ is exact. To see this,
observe that $\operatorname*{PMod}(\mathfrak{A})$ is split exact, so the
obvious additivity of the functor suffices for its exactness. We write
$p_{i}:\operatorname*{PMod}(\mathfrak{A})\rightarrow\mathsf{LCA}%
_{\mathfrak{A},cg}$ with $i=1,2,3$ for the three individual functors to the
left, middle resp. right term of the exact sequence. Note that for
$M:=\mathfrak{A}$, we have $p_{1}(M)\simeq\mathbb{Z}^{n}$ on the level of the
underlying LCA group (see \S \ref{subsect_Basics}), moreover $p_{2}%
(M)\simeq\mathbb{R}^{n}$ and $p_{3}(M)\simeq\mathbb{T}^{n}$, and on the level
of the underlying LCA groups the sequence $p(M)$ is just $\mathbb{Z}%
^{n}\hookrightarrow\mathbb{R}^{n}\twoheadrightarrow\mathbb{T}^{n}$. Since
every localizing invariant is also additive in the sense of \cite{MR3070515},
the exact functor $p$ induces a relation among the induced maps $p_{i\ast
}:K(\mathfrak{A})\rightarrow K(\mathsf{LCA}_{\mathfrak{A},cg})$, namely
$p_{2\ast}=p_{1\ast}+p_{3\ast}$. The functor $p_{3}$ admits a factorization%
\[
p_{3}:\operatorname*{PMod}(\mathfrak{A})\longrightarrow\mathsf{LCA}%
_{\mathfrak{A},C}\longrightarrow\mathsf{LCA}_{\mathfrak{A},cg}\text{,}%
\]
with $\mathsf{LCA}_{\mathfrak{A},C}$ being the full subcategory of compact
modules (Definition \ref{def_LCAAC}), because we have just seen that $p_{3}$
maps every $M\in\operatorname*{PMod}(\mathfrak{A})$ to a compact module. By
the Eilenberg swindle, $K(\mathsf{LCA}_{\mathfrak{A},C})=0$, see \cite[Lemma
4.2]{obloc}. Thus, $p_{3\ast}$ must be the zero map. We conclude that
$p_{1\ast}=p_{2\ast}$. The diagram%
\[%
\xymatrix{
K(\operatorname{PMod}(\mathfrak{A})) \ar[r]^{\sim} \ar[d]_{p_1} & K(\mathsf
{Mod}_{\mathfrak{A},fg}) \ar[dl]^{g} \\
K(\mathsf{LCA}_{\mathfrak{A},cg})
}%
\]
commutes, where the top horizontal map is the equivalence of Equation
\ref{l_B3a1}, and $g$ is the map of Equation \ref{l_Bb1}. Combining these two
facts, we deduce that upon (equivalently) replacing $K(\mathsf{Mod}%
_{\mathfrak{A},fg})$ by $K(\mathfrak{A})$ in Equation \ref{l_Bb1}, we obtain%
\begin{equation}
K(\mathfrak{A})\overset{p_{2}}{\longrightarrow}K(\mathsf{LCA}_{\mathfrak{A}%
,cg})\longrightarrow K(\mathsf{LCA}_{\mathfrak{A}})\text{.}\label{l_Bb2}%
\end{equation}
(Step 3) Using Proposition \ref{prop_CGToRC}, we have $K(\mathsf{LCA}%
_{\mathfrak{A},\mathbb{R}C})\overset{\sim}{\longrightarrow}K(\mathsf{LCA}%
_{\mathfrak{A},cg})$, and by Proposition \ref{Prop_KRCEqualsKR} moreover
$K(\mathsf{LCA}_{\mathfrak{A},\mathbb{R}C})\overset{\sim}{\longrightarrow
}K(A_{\mathbb{R}})$. Combining both equivalences, we may replace the middle
term of the sequence in Equation \ref{l_Bb2} by $K(A_{\mathbb{R}})$. Even
better, note from the construction of the maps in the cited propositions that
this equivalence stems from the inclusion of categories $\operatorname*{PMod}%
(A_{\mathbb{R}})\subseteq\mathsf{LCA}_{\mathfrak{A},cg}$ (where each module
$M\in\operatorname*{PMod}(A_{\mathbb{R}})$ is regarded as equipped with the
locally compact topology coming from its real vector space structure), and the
image of the functor $p_{2}$ indeed lies in this subcategory, so we get a new
fiber sequence%
\[
K(\mathfrak{A})\overset{p_{2}}{\longrightarrow}K(A_{\mathbb{R}})\overset
{j}{\longrightarrow}K(\mathsf{LCA}_{\mathfrak{A}})\text{,}%
\]
where the first map is still induced from the exact functor $p_{2}$. Finally,
note that the functor underlying $p_{2}$ is exactly the map $i$ which we are
discussing in the statement of the theorem, so this correctly identifies the
first map. Similarly, the second map arises as the composition of
$\operatorname*{PMod}(A_{\mathbb{R}})\subseteq\mathsf{LCA}_{\mathfrak{A},cg}$
and the inclusion $\mathsf{LCA}_{\mathfrak{A},cg}\hookrightarrow
\mathsf{LCA}_{\mathfrak{A}}$, see Equation \ref{l_B2}, so we also see that the
map $j$ is exactly the one as we claim in the theorem. This finishes the proof.
\end{proof}

We obtain the formulation of the introduction:

\begin{theorem}
\label{thm_main_asannounced}Suppose $\mathfrak{A}$ is a regular order in a
finite-dimensional semisimple $\mathbb{Q}$-algebra $A$. Then there is a long
exact sequence%
\[
\cdots\longrightarrow K_{n}(\mathfrak{A})\longrightarrow K_{n}(A_{\mathbb{R}%
})\longrightarrow K_{n}(\mathsf{LCA}_{\mathfrak{A}})\longrightarrow
K_{n-1}(\mathfrak{A})\longrightarrow\cdots\text{,}%
\]
and for all $n$, there are canonical isomorphisms%
\[
K_{n}(\mathsf{LCA}_{\mathfrak{A}})\cong K_{n-1}(\mathfrak{A},\mathbb{R}%
)\text{.}%
\]
Here $K_{\ast}(-)$ denotes ordinary (Quillen) algebraic $K$-theory.
\end{theorem}

\begin{proof}
We use that \textit{non-connective} algebraic $K$-theory is a localizing
invariant $\mathbb{K}:\operatorname*{Cat}_{\infty}^{\operatorname*{ex}%
}\rightarrow\mathsf{Sp}$ with values in spectra, \cite{MR3070515}. Thus,
Theorem \ref{thm_Main_b} applies, and the long exact sequence of homotopy
groups associated to the fiber sequence gives us a long exact sequence%
\[
\cdots\longrightarrow\mathbb{K}_{n}(\mathfrak{A})\longrightarrow\mathbb{K}%
_{n}(A_{\mathbb{R}})\longrightarrow\mathbb{K}_{n}(\mathsf{LCA}_{\mathfrak{A}%
})\longrightarrow\mathbb{K}_{n-1}(\mathfrak{A})\longrightarrow\cdots
\]
of non-connective $K$-groups. Around degree zero it reads%
\[
\cdots\longrightarrow\mathbb{K}_{0}(\mathfrak{A})\longrightarrow\mathbb{K}%
_{0}(A_{\mathbb{R}})\longrightarrow\mathbb{K}_{0}(\mathsf{LCA}_{\mathfrak{A}%
})\longrightarrow\mathbb{K}_{-1}(\mathfrak{A})\longrightarrow\mathbb{K}%
_{-1}(A_{\mathbb{R}})\longrightarrow\cdots\text{.}%
\]
Since $\mathfrak{A}$ and $A_{\mathbb{R}}$ are regular rings, their
non-connective $K$-theory agrees with the ordinary $K$-theory, i.e.
$\mathbb{K}_{i}(\mathfrak{A})=K_{i}(\mathfrak{A})$ for all $i\geq1$ anyway,
they agree for $K_{0}$ by idempotent completeness (same for $A_{\mathbb{R}}$),
and $\mathbb{K}_{-i}(\mathfrak{A})=\mathbb{K}_{-i}(A_{\mathbb{R}})=0$ for all
$i\geq1$, \cite[Remark 7]{MR2206639}. Observe that the exactness of the
sequence then implies $\mathbb{K}_{-i}(\mathsf{LCA}_{\mathfrak{A}})=0$ for
$i\geq1$. Next, the category $\mathsf{LCA}_{\mathfrak{A}}$ has all kernels, so
it is idempotent complete. It follows that $\mathbb{K}_{0}(\mathsf{LCA}%
_{\mathfrak{A}})=K_{0}(\mathsf{LCA}_{\mathfrak{A}})$, \cite[Remark
3]{MR2206639}. Hence, all letters $\mathbb{K}$ in the above sequence can be
replaced by $K$ without a change. We have proven the claim.
\end{proof}

With the tools which we have available, here is the best we can do at the
moment regarding non-regular orders.

\begin{theorem}
\label{thm_Gthy_Version}Suppose $\mathfrak{A}$ is an arbitrary order in a
finite-dimensional semisimple $\mathbb{Q}$-algebra $A$. Then there is a long
exact sequence%
\[
\cdots\longrightarrow G_{n}(\mathfrak{A})\longrightarrow K_{n}(A_{\mathbb{R}%
})\longrightarrow K_{n}(\mathsf{LCA}_{\mathfrak{A}})\longrightarrow
G_{n-1}(\mathfrak{A})\longrightarrow\cdots\text{,}%
\]
where $G_{n}(\mathfrak{A}):=K_{n}(\mathsf{Mod}_{\mathfrak{A},fg})$ denotes the
$K$-theory of the category of finitely generated right $\mathfrak{A}$-modules
(this is often called \textquotedblleft$G$-theory\textquotedblright).
\end{theorem}

\begin{proof}
Use Proposition \ref{prop_PrepareMain} for non-connective $K$-theory
$\mathbb{K}:\operatorname*{Cat}_{\infty}^{\operatorname*{ex}}\rightarrow
\mathsf{Sp}$ as in the proof of Theorem \ref{thm_main_asannounced}. The
argument that $\mathbb{K}_{n}(A_{\mathbb{R}})=K_{n}(A_{\mathbb{R}})$ holds for
all $n$ is still valid since $A_{\mathbb{R}}$ is regular. Since $\mathsf{Mod}%
_{\mathfrak{A},fg}$ is a Noetherian abelian category, \cite[Theorem
7]{MR2206639} shows that $\mathbb{K}_{n}(\mathsf{Mod}_{\mathfrak{A},fg})=0$
for all $n<0$, and $\mathbb{K}_{0}(\mathsf{Mod}_{\mathfrak{A},fg}%
)=K_{0}(\mathsf{Mod}_{\mathfrak{A},fg})$ since the category is idempotent
complete, \cite[Remark 3]{MR2206639}. As in the proof of Theorem
\ref{thm_main_asannounced}, we obtain $\mathbb{K}_{n}(\mathsf{LCA}%
_{\mathfrak{A}})=K_{n}(\mathsf{LCA}_{\mathfrak{A}})$.
\end{proof}

It is unclear to us whether $G$-theory can fulfill any meaningful r\^{o}le in
the context of the ETNC. To the best of my knowledge, it has not been considered.

\section{Equivariant Haar measure\label{section_HaarMeasureDet}}

In this section we will supply some details regarding the interpretation of
Haar measures as a determinant functor. For background on Picard groupoids we
refer to \cite[\S 4]{MR902592}, \cite[\S 2.1]{MR1884523} or \cite[\S 2]%
{MR2842932}. We write $\mathsf{Picard}$ for the category of Picard groupoids,
defined as in \cite[1.3. Definition]{MR2981952}.

\subsection{Determinant functors}

Let $\mathsf{C}$ be an exact category. Let $\mathsf{C}^{\times}$ denote its
maximal inner groupoid, i.e. the category with the same objects but we only
keep isomorphisms as morphisms.

\begin{definition}
[{\cite[\S 4.3]{MR902592}}]\label{def_DetFunctor}Suppose $(\mathsf{P}%
,\otimes)$ is a Picard groupoid. A \emph{determinant functor} on $\mathsf{C}$
is a functor%
\[
\mathcal{P}:\mathsf{C}^{\times}\longrightarrow\mathsf{P}%
\]
with the following extra structure and axioms:

\begin{enumerate}
\item For every exact sequence $\Sigma:G^{\prime}\hookrightarrow
G\twoheadrightarrow G^{\prime\prime}$ in $\mathsf{C}$ we are given an
isomorphism%
\begin{equation}
\mathcal{P}(\Sigma):\mathcal{P}(G)\overset{\sim}{\longrightarrow}%
\mathcal{P}(G^{\prime})\underset{\mathsf{P}}{\otimes}\mathcal{P}%
(G^{\prime\prime})\text{.}\label{lcey0}%
\end{equation}
Moreover, this isomorphism is functorial in morphisms of exact sequences.

\item For every zero object $Z$ in $\mathsf{C}$ we are given an isomorphism
$z:\mathcal{P}(Z)\overset{\sim}{\longrightarrow}1_{\mathsf{P}}$ to the neutral
object of the Picard groupoid. Henceforth, we simply write $0$ for a zero object.

\item If $f:G\rightarrow G^{\prime}$ is an isomorphism in $\mathsf{C}$, write%
\[
\Sigma_{l}:0\hookrightarrow G\twoheadrightarrow G^{\prime}\qquad
\text{and}\qquad\Sigma_{r}:G\hookrightarrow G^{\prime}\twoheadrightarrow0
\]
for the depicted exact sequences. We demand that%
\begin{equation}
\mathcal{P}(G)\underset{\mathcal{P}(\Sigma_{l})}{\overset{\sim}%
{\longrightarrow}}\mathcal{P}(0)\underset{\mathsf{P}}{\otimes}\mathcal{P}%
(G^{\prime})\underset{z\otimes\operatorname*{id}}{\overset{\sim}%
{\longrightarrow}}1_{\mathsf{P}}\underset{\mathsf{P}}{\otimes}\mathcal{P}%
(G^{\prime})\underset{\mathsf{P}}{\overset{\sim}{\longrightarrow}}%
\mathcal{P}(G^{\prime})\label{lcey1}%
\end{equation}
agrees with the map $\mathcal{P}(f):\mathcal{P}(G)\overset{\sim}%
{\longrightarrow}\mathcal{P}(G^{\prime})$. Analogously, we demand that
$\mathcal{P}(f^{-1})$ agrees with a variant of Equation \ref{lcey1} using
$\Sigma_{r}$ instead of $\Sigma_{l}$.

\item If $G_{1}\hookrightarrow G_{2}\hookrightarrow G_{3}$ is given, we demand
that the diagram%
\[%
\xymatrix{
\mathcal{P}(G_3) \ar[r]^-{\sim} \ar[d]_{\sim} & \mathcal{P}(G_1) \otimes
\mathcal{P}(G_3/G_1) \ar[d]^{\sim} \\
\mathcal{P}(G_2) \otimes\mathcal{P}(G_3/G_2) \ar[r]_-{\sim} & \mathcal
{P}(G_1) \otimes\mathcal{P}(G_2/G_1) \otimes\mathcal{P}(G_3/G_2)
}%
\]
commutes.
\end{enumerate}
\end{definition}

\subsection{Picard groupoids as spectra}

Let us quickly recall that there are (at least) two different perspectives how
to think about Picard groupoids. On the one hand, one can think of the
algebraic definition in terms of a symmetric monoidal category as in
\cite[\S 2]{MR1884523}, and on the other hand one can think of them as spectra
whose possibly non-zero homotopy groups are confined to degrees $0$ and $1$.
We denote the latter category by $\mathsf{Sp}^{0,1}$, its objects are also
known as \textquotedblleft stable $1$-types\textquotedblright.

\begin{theorem}
\label{thm_hoPicvshoSp01}There is an equivalence of homotopy categories%
\begin{equation}
\Psi:\operatorname*{Ho}(\mathsf{Picard})\overset{\sim}{\longrightarrow
}\operatorname*{Ho}(\mathsf{Sp}^{0,1})\text{.}\label{lceyAJ1}%
\end{equation}
This correspondence preserves the notions of homotopy groups $\pi_{0},\pi_{1}$
on either side.
\end{theorem}

Proofs are given in \cite[\S 5.1, Theorem 5.3]{MR2981817} or
\cite[1.5\ Theorem]{MR2981952}, but this fact was known long before, certainly
to Deligne and\ Drinfeld. Under this correspondence, the homotopy groups
$\pi_{i}$ for $i=0,1$ of a Picard groupoid $(\mathsf{P},\otimes)$ which are
discussed in Burns--Flach \cite[\S 2.1]{MR1884523} match isomorphically with
the homotopy groups of $\Psi(\mathsf{P},\otimes)$ as a spectrum. Let us point
out that the spectra in $\mathsf{Sp}^{0,1}$ can also be modelled alternatively
as genuine spaces (even in the format of CW complexes) under a further
equivalence between the homotopy categories of connective spectra with
infinite loop spaces. We will not make use of this at all, but it gives the
homotopy groups $\pi_{i}$ a very concrete meaning. Nonetheless, it is
typically futile to try to visualize the actual topology of such spaces.

\begin{definition}
[{Deligne, \cite[\S 4.2]{MR902592}}]\label{def_VirtObjects}Suppose
$\mathsf{C}$ is an exact category. The truncation $\tau_{\leq1}K(\mathsf{C})$
of the connective $K$-theory spectrum, viewed as a Picard groupoid (via
Equation \ref{lceyAJ1}), is called the \emph{Picard groupoid of virtual
objects}. We denote it by%
\[
V(\mathsf{C}):=\Psi^{-1}\left(  \tau_{\leq1}K(\mathsf{C})\right)  \text{.}%
\]

\end{definition}

\begin{remark}
Picard groupoids can also be modelled in very different ways, e.g., the stable
quadratic modules of Baues \cite{MR2353254}, see also \cite{MR3302579} for a
broader discussion.
\end{remark}

\subsection{Universal determinant}

Let us also recall a basic feature of $K$-theory. The following discussion
depends a little on what definition the reader has in mind for $K_{1}$. If $R
$ denotes a unital associative ring, the first $K$-theory group is most
frequently defined using the formula%
\begin{equation}
K_{1}(R):=\frac{\operatorname*{GL}(R)}{[\operatorname*{GL}%
(R),\operatorname*{GL}(R)]}\text{,}\label{lceyAJ3}%
\end{equation}
i.e. $K_{1}$ is the abelianization of the group $\operatorname*{GL}(R)$, which
in turn is the group of matrices of arbitrary large rank, viewed as a
compatible system by embedding rank $(n\times n)$-matrices into $(n^{\prime
}\times n^{\prime})$-matrices for $n^{\prime}\geq n$ as the top left minor and
completing the diagonal by $1$'s, and all off-diagonal terms by $0$'s.

From this description one can attach to every finitely generated projective
$R$-module $M\in\operatorname*{PMod}(R)$ along with an automorphism
$f:M\overset{\sim}{\rightarrow}M$ an element of $K_{1}$. If $M$ is free, this
is clear: Just pick an $R$-module isomorphism $\phi:M\simeq R^{n}$ for a
suitable $n$; then along this isomorphism $f$ becomes an element $\phi
f\phi^{-1}\in\operatorname*{GL}_{n}(R)\subseteq\operatorname*{GL}(R)$, and
thanks to the abelianization, the choice of the isomorphism $\phi$ was
irrelevant. If $M$ is projective, write $M$ as a direct summand of a free
module, say $M\oplus M^{\prime}\simeq R^{n}$, and prolong $f$ as $f\oplus0$ to
$R^{n}$. Then proceed as before. This construction gives a well-defined map%
\[
\operatorname*{Aut}(M)\longrightarrow K_{1}(R)\text{,}%
\]
for any $M\in\operatorname*{PMod}(R)$. Of course, choosing $M$ also gives a
well-defined element in $K_{0}(R)$, namely the class $[M]$ representing the
module. Both of these constructions can be combined, giving a symmetric
monoidal functor of Picard groupoids%
\[
\operatorname*{PMod}(R)^{\times}\longrightarrow V(R):=V(\operatorname*{PMod}%
(R))\text{.}%
\]
See also \cite[\S 2.5]{MR1884523}.

Of course, other authors right away define $K_{1}(R)$ through generators and
relations, where generators have the shape $[M,f]$, which $M$ being a finitely
generated projective $R$-module and $f$ an automorphism, e.g., \cite[\S 1.1]%
{MR2276851}. A more detailed discussion of the variants of these constructions
can be found in Weibel \cite{MR3076731}, Chapter III, Lemma 1.6 and Corollary 1.6.3.

A key insight due to Deligne is that the universal determinant functor of an
exact category arises through its $K$-theory, and essentially just amounts to
truncating the $K$-theory spectrum to degrees $[0,1]$. For an arbitrary exact
category $\mathsf{C}$, one can set up a canonical symmetric monoidal functor%
\begin{equation}
\mathsf{C}^{\times}\longrightarrow V(\mathsf{C})\text{.}\tag{$(\ast
)$}\label{lceyAJ4}%
\end{equation}
We have allowed ourselves the quick detour through projective $R$-modules and
the explicit presentation of Equation \ref{lceyAJ3} to motivate this
construction in a fairly concrete case.

\begin{remark}
The construction of the functor in Equation \ref{lceyAJ4} requires entering
the simplicial details of the definition of $K$-theory. The $K$-theory of an
exact category $\mathsf{C}$ is a special case of the $K$-theory of a
Waldhausen category (take the admissible monics as the cofibrations and
isomorphisms as weak equivalences). Waldhausen constructs a map%
\[
\left\vert \mathsf{C}^{\times}\right\vert \longrightarrow\Omega\left\vert
(S_{\bullet}\mathsf{C})^{\times}\right\vert =K(\mathsf{C})\text{,}%
\]
where $\left\vert -\right\vert $ denotes geometric realization, $S_{\bullet} $
is Waldhausen's $S$-construction, and $\Omega$ is the loop space functor, see
\cite[bottom of p. 329]{MR802796} or \cite[Chapter IV, Remark 8.3.2]%
{MR3076731}. This induces a map to the truncation,%
\[
\left\vert \mathsf{C}^{\times}\right\vert \longrightarrow K(\mathsf{C}%
)\longrightarrow\tau_{\leq1}K(\mathsf{C})
\]
and under the identification of the spaces with Picard groupoids, this becomes
the functor in Equation \ref{lceyAJ4}.
\end{remark}

The universality of $\mathsf{C}^{\times}\rightarrow V(\mathsf{C})$ as a
determinant functor now essentially reduces to the fact that the truncation
$\tau_{\leq1}$ is the universal operation having a stable $[0,1]$-type as output.

\begin{theorem}
[Deligne]\label{thm_deligne}Suppose $\mathsf{C}$ is an exact category. Then
the functor of Equation \ref{lceyAJ4} is universal for all determinant
functors:\ If $\mathcal{P}:\mathsf{C}^{\times}\rightarrow\mathsf{P}$ is any
determinant functor, then it factors%
\[%
\xymatrix{
\mathsf{C}^{\times} \ar[d]_{\ast} \ar[r]^{\mathcal{P} } & \mathsf{P} \\
{\Psi^{-1}}{\tau}_{\le1}K(\mathsf{C}), \ar[ur]
}%
\]
where $(\ast)$ is the morphism from $\mathsf{C}^{\times}$ to the virtual objects.
\end{theorem}

We refer to Deligne \cite{MR902592} for the proof.

Suppose $A$ is an abelian group. The category of $A$-torsors is a Picard
groupoid, denote it by $\mathsf{Tors}(A)$, see \cite[Appendix A.1]{MR1988970}
for background. In this situation%
\[
\pi_{0}(\mathsf{Tors}(A))=0\qquad\qquad\text{and}\qquad\qquad\pi
_{1}(\mathsf{Tors}(A))=A\text{.}%
\]

The other important Picard groupoid is the \emph{determinant line}
$\mathsf{Pic}_{R}^{\mathbb{Z}}$ for a commutative ring $R$. This formalism is
developed by Knudsen and Mumford \cite[Chapter 1]{MR0437541}. They also prove
its universality.

\begin{theorem}
[Knudsen--Mumford]Suppose $R$ is a commutative ring. Then the universal
determinant functor of the category of finitely generated projective
$R$-modules $\operatorname*{PMod}(R)$ is equivalent to the determinant line,%
\begin{align*}
\operatorname*{PMod}(R)  & \longrightarrow\mathsf{Pic}_{R}^{\mathbb{Z}}\\
M  & \longmapsto\left(  \bigwedge\nolimits^{\operatorname*{top}}%
M,\operatorname*{rk}M\right)  \text{.}%
\end{align*}

\end{theorem}

This is, essentially, \cite[Chapter 1, Theorem 1]{MR0437541}.\textit{\ Loc.
cit.} they give a different definition of a determinant functor, which
involves a normalization, so one has to translate the cited result with some
care. Although not relevant for our purposes, Knudsen and Mumford develop the
determinant line right away not just for commutative rings, but for an
arbitrary scheme.

The determinant line plays a fundamental r\^{o}le for the ETNC with
commutative coefficients, in the works predating \cite{MR1884523}.

With these preparations, we are ready to identify the Haar torsor as the
universal determinant functor on the exact category $\mathsf{LCA}$. We do not
know who made this observation first. In a perhaps not truly proven state it
seems to have been folklore for a while, but a complete proof probably only
became possible with the work of Clausen \cite{clausen}.

\begin{theorem}
[Universality of the Haar torsor]\label{thm_UnivOfHaarTorsor}The Haar functor
$Ha:\mathsf{LCA}^{\times}\rightarrow\mathsf{Tors}(\mathbb{R}_{>0}^{\times})$
is the universal determinant functor of the category $\mathsf{LCA}$. In
particular, Deligne's Picard groupoid of virtual objects for $\mathsf{LCA}$ is
isomorphic to the Picard groupoid of $\mathbb{R}_{>0}^{\times}$-torsors.
\end{theorem}

\begin{proof}
Let $\det:\mathsf{LCA}^{\times}\longrightarrow V(\mathsf{LCA})$ be the
universal determinant functor in the sense of Deligne. Since $Ha$ is a
determinant functor (we checked this in \S \ref{subsect_HaarIsDetFunctor}), we
obtain a canonical factorization $F:V(\mathsf{LCA})\longrightarrow
\mathsf{Tors}(\mathbb{R}_{>0}^{\times})$. Now, by using Clausen's computation
of the $K$-theory of $\mathsf{LCA}$ \cite{clausen}, there is a long exact
sequence%
\[
K_{1}(\mathbb{Z})\longrightarrow K_{1}(\mathbb{R})\longrightarrow
K_{1}(\mathsf{LCA})\longrightarrow K_{0}(\mathbb{Z})\longrightarrow
K_{0}(\mathbb{R})\longrightarrow K_{0}(\mathsf{LCA})\longrightarrow0\text{.}%
\]
While this result is originally due to Clausen, it of course also follows from
our main result Theorem \ref{thm_main_asannounced} in the special case of the
order $\mathfrak{A:}=\mathbb{Z}$ in $A:=\mathbb{Q}$. Unravelling the terms, we
obtain%
\begin{equation}
\mathbb{Z}^{\times}\longrightarrow\mathbb{R}^{\times}\longrightarrow
K_{1}(\mathsf{LCA})\longrightarrow\mathbb{Z}\overset{1}{\longrightarrow
}\mathbb{Z}\longrightarrow K_{0}(\mathsf{LCA})\longrightarrow0\text{,}%
\label{l_cio4}%
\end{equation}
and therefore $K_{0}(\mathsf{LCA})=0$ and the absolute value isomorphism
$K_{1}(\mathsf{LCA})\cong\mathbb{R}^{\times}/\mathbb{Z}^{\times}\overset{\sim
}{\underset{\operatorname{abs}}{\longrightarrow}}\mathbb{R}_{>0}^{\times}$. By
Deligne's Theorem \ref{thm_deligne} the Picard groupoid of virtual objects is,
when regarded as a spectrum, the truncation of $K(\mathsf{LCA})$ to a stable
$[0,1]$-type. Hence, on $\pi_{0},\pi_{1}$ the functor $F$ induces the maps%
\begin{align*}
0  & =\pi_{0}(V(\mathsf{LCA}))\longrightarrow\pi_{0}(\mathsf{Tors}%
(\mathbb{R}_{>0}^{\times}))=0\\
\mathbb{R}_{>0}^{\times}  & =\pi_{1}(V(\mathsf{LCA}))\longrightarrow\pi
_{1}(\mathsf{Tors}(\mathbb{R}_{>0}^{\times}))\cong\mathbb{R}_{>0}^{\times}%
\end{align*}
because the homotopy groups of $V(\mathsf{LCA})$ in degree $0,1$ then agree
with the ones of $K(\mathsf{LCA})$. Since the construction of the modulus, see
Equation \ref{lcey4b}, can be recast as the comparison of a trivialization of
the torsor with the pullback of that trivialization along an automorphism, it
follows that the map on the level of $\pi_{1}$ is an isomorphism. However, a
map among stable $[0,1]$-types, or equivalently a symmetric monoidal functor
among Picard groupoids is an equivalence if it induces an isomorphism of
$\pi_{0}$ and $\pi_{1}$, \cite[Lemma 5.7]{MR2981817}. Hence, $F$ is an
equivalence of Picard groupoids and via the isomorphism $F$,%
\[
\det:\mathsf{LCA}^{\times}\longrightarrow V(\mathsf{LCA})\underset{F}%
{\overset{\sim}{\longrightarrow}}\mathsf{Tors}(\mathbb{R}_{>0}^{\times
})\text{,}%
\]
we identify the Haar functor as the universal determinant functor of
$\mathsf{LCA}$.
\end{proof}

\begin{theorem}
\label{marker_ComputeVLCAA}Let $\mathfrak{A}$ be a regular order in a
finite-dimensional semisimple $\mathbb{Q}$-algebra $A$. Then%
\begin{align*}
\pi_{0}V(\mathsf{LCA}_{\mathfrak{A}})  & \cong K_{-1}(\mathfrak{A}%
,\mathbb{R})\\
\pi_{1}V(\mathsf{LCA}_{\mathfrak{A}})  & \cong K_{0}(\mathfrak{A}%
,\mathbb{R})\text{,}%
\end{align*}
and, stronger, the Picard groupoid of equivariant volumes $V(\mathsf{LCA}%
_{\mathfrak{A}})$ is equivalent to the $1$-truncation of the $(-1)$-shift of
the fiber of the morphism $K(\mathfrak{A})\longrightarrow K(A_{\mathbb{R}})$.
\end{theorem}

\begin{proof}
The ring morphism $c:\mathfrak{A}\longrightarrow A_{\mathbb{R}}$, which
plainly arises from tensoring $(-)\mapsto(-)\otimes_{\mathbb{Z}}\mathbb{R}$,
induces the relative $K$-theory fiber sequence%
\[
K(\mathfrak{A},\mathbb{R})\longrightarrow K(\mathfrak{A})\overset
{c}{\longrightarrow}K(A_{\mathbb{R}})\text{,}%
\]
where $K(\mathfrak{A},\mathbb{R}):=\operatorname*{hofib}(c:K(\mathfrak{A}%
)\longrightarrow K(A_{\mathbb{R}}))$ is just defined as the homotopy fiber.
This is equivalent to the conventions used in \cite{MR1884523}. By Theorem
\ref{thm_Main_b} it follows that there is an equivalence $K(\mathsf{LCA}%
_{\mathfrak{A}})\longrightarrow\Sigma K(\mathfrak{A},\mathbb{R})$, where
$\Sigma$ denotes the suspension. On the level of homotopy groups, this is a
shift of degree by one, thus%
\[
K_{i}(\mathfrak{A},\mathbb{R})\cong K_{i+1}(\mathsf{LCA}_{\mathfrak{A}})
\]
holds for all integers $i$.
\end{proof}

\bigskip

\section{Real and $p$-adic comparison}

Let $\mathfrak{A}$ be an arbitrary order in a finite-dimensional semisimple
$\mathbb{Q}$-algebra $A$. We recall that $A:=\mathfrak{A}\otimes_{\mathbb{Z}%
}\mathbb{Q}$. As in \cite[\S 2.7]{MR1884523}, we define%
\[
\widehat{A}:=A\otimes_{\mathbb{Z}}\widehat{\mathbb{Z}}=\mathfrak{A}%
\otimes_{\mathbb{Z}}\mathbb{A}_{fin}\text{,}%
\]
where $\mathbb{A}_{fin}$ denotes the \textit{finite part }of the ad\`{e}les of
the rational number field $\mathbb{Q}$, i.e. the restricted product $\left.
\prod\nolimits_{p}^{\prime}\right.  (\mathbb{Q}_{p},\mathbb{Z}_{p})$, where
$p$ only runs over the primes (excluding the infinite place). When we speak of
$\widehat{A}$, we regard it with the locally compact topology coming from
$\mathbb{A}$. In more detail:\ As an abelian group, we have $\mathfrak{A}%
\otimes_{\mathbb{Z}}\mathbb{A}_{fin}\overset{\phi}{\simeq}\mathbb{A}_{fin}%
^{n}$ for $n:=\dim_{\mathbb{Q}}A$, and if we equip the right side with the
standard topology of the ad\`{e}les, then this induces a topology on the left
side. This topology is independent of the choice of the isomorphism $\phi$.

\begin{theorem}
[Reciprocity Law]\label{thm_Recip}Let $\mathfrak{A}$ be an arbitrary order in
a finite-dimensional semisimple $\mathbb{Q}$-algebra $A$. Then the composition%
\begin{equation}
K(A)\longrightarrow K(\widehat{A})\oplus K(A_{\mathbb{R}})\overset
{+}{\longrightarrow}K(\mathsf{LCA}_{\mathfrak{A}})\label{lmito1}%
\end{equation}
is zero.

\begin{enumerate}
\item The first arrow is induced from the exact functor $M\mapsto(M\otimes
_{A}\widehat{A}\,,\,M\otimes_{\mathbb{Z}}\mathbb{R)}$.

\item The second arrow is induced from mapping a free right $\widehat{A}%
$-module to itself, regarded with its natural topology, as explained in above
the theorem. Similarly, map a free right $A_{\mathbb{R}}$-module to itself,
equipped with the real topology.
\end{enumerate}
\end{theorem}

This is a non-commutative version of the reciprocity-like statements of
\cite{clausen}. Note in the proof below that the essential bits of the proof
are not so much arithmetic, but rather topological. However, of course
arithmetic and topology are intertwined concepts when it comes to the ad\`{e}les.

\begin{proof}
Write $\mathbb{A}$ for the full ad\`{e}les of the rational numbers
$\mathbb{Q}$, i.e. $\mathbb{A}\cong\mathbb{A}_{fin}\oplus\mathbb{R}$. Recall
that%
\[
\mathbb{Q}\hookrightarrow\mathbb{A}\twoheadrightarrow\mathbb{A}/\mathbb{Q}%
\]
is an exact sequence in $\mathsf{LCA}$, where $\mathbb{Q}$ is embedded
diagonally as a closed subgroup. The quotient $\mathbb{A}/\mathbb{Q}$ is
compact connected. We define three exact functors%
\[
f_{i}:\operatorname*{PMod}(A)\longrightarrow\mathsf{LCA}_{\mathfrak{A}}%
\]%
\[
f_{1}(A):=A\text{,}\qquad f_{2}(A):=A\otimes_{\mathbb{Q}}\mathbb{A}%
\text{,}\qquad f_{3}(A):=(A\otimes_{\mathbb{Q}}\mathbb{A})/A\text{,}%
\]
where $A$ in $f_{1}$ is equipped with the discrete topology, $A\otimes
_{\mathbb{Q}}\mathbb{A}$ in $f_{2}$ (which is $\simeq\mathbb{A}^{n}$ as an
LCA\ group) with the natural topology of the ad\`{e}les, and the values of
$f_{3}$ with the quotient topology: Since $\mathbb{Q}\hookrightarrow
\mathbb{A}$ is closed, $A\hookrightarrow A\otimes_{\mathbb{Q}}\mathbb{A}$ is
also closed. It suffices that we have given the values of $f_{i}$ for
$i=1,2,3$ only for $A$ since $A$ is a projective generator of the split exact
category $\operatorname*{PMod}(A)$, so in order to define an exact functor it
suffices to specify its value on a projective generator. Since for every
$W\in\operatorname*{PMod}(A)$ we obtain that $f_{1}(W)\hookrightarrow
f_{2}(W)\twoheadrightarrow f_{3}(W)$ is exact in $\mathsf{LCA}_{\mathfrak{A}}%
$, and functorially so, the Additivity Theorem implies that%
\[
f_{2\ast}=f_{1\ast}+f_{3\ast}:K(A)\longrightarrow K(\mathsf{LCA}%
_{\mathfrak{A}})
\]
as maps of spectra (this is analoguos to an idea in the proof of Theorem
\ref{thm_Main_b}). As $f_{1\ast}=f_{3\ast}=0$, because they factor over the
fully exact subcategories $\mathsf{LCA}_{\mathfrak{A},D}$ resp. $\mathsf{LCA}%
_{\mathfrak{A},C}$ whose $K$-theory is contractible by \cite[Lemma 4.2]%
{obloc}, we deduce $f_{2\ast}=0$. Finally, observe that the composition of the
functors in the statement of our theorem literally give the functor $f_{2}$,
finishing the proof.
\end{proof}

\begin{example}
Suppose $F$ is a number field and $\mathfrak{A}:=\mathcal{O}$ its ring of
integers. Then the reciprocity law, Theorem \ref{thm_Recip}, states that the
composition%
\[
K(F)\longrightarrow K(\mathbb{A}_{F,fin})\oplus K(F\otimes_{\mathbb{Q}%
}\mathbb{R})\longrightarrow K(\mathsf{LCA}_{\mathcal{O}})
\]
is zero. In degree one, this becomes $F^{\times}\rightarrow K_{1}%
(\mathbb{A}_{F})\rightarrow K(\mathsf{LCA}_{\mathcal{O}})$ being the zero map.
This is quite close to the statement that%
\[
F^{\times}\hookrightarrow\mathbb{A}^{\times}\twoheadrightarrow K_{1}%
(\mathsf{LCA}_{F})
\]
is an exact sequence, proven in joint work with Arndt in \cite{kthyartin}.
These reciprocity statements are modelled after Clausen's ideas in
\cite{clausen}.
\end{example}

At the beginning of the paper we had discussed the difference between the real
determinant line and the Haar torsor, and we had made some claims regarding
their precise relation which so far we have not proved, and we fill this gap now.

\begin{proposition}
\label{prop_CC1}We compare the real/$p$-adic determinant line torsor with the
Haar torsor in the non-equivariant setting:

\begin{enumerate}
\item In Example \ref{ex_DetA1} the map%
\[
\pi_{1}(\mathsf{Pic}_{R}^{\mathbb{Z}})\rightarrow\pi_{1}(V(\mathsf{LCA}))
\]
is given by $\alpha\mapsto\left\vert \alpha\right\vert $, i.e. the Haar torsor
differs from the graded real determinant line precisely by forgetting the sign.

\item In Example \ref{ex_DetA2} the map%
\[
\pi_{1}(\mathsf{Pic}_{\mathbb{Q}_{p}}^{\mathbb{Z}})\longrightarrow\pi
_{1}(V(\mathsf{LCA}))
\]
is given by $\alpha\longmapsto\left\vert \alpha\right\vert _{p}$, i.e. the
Haar torsor differs from the graded $p$-adic determinant line precisely by
just keeping the $p$-adic valuation and nothing else (so it does not see the
multiplication by any unit in $\mathbb{Z}_{p}^{\times}$).
\end{enumerate}
\end{proposition}

\begin{proof}
(1) The map in question is induced from the map of $K$-theory spectra
$K(\mathbb{R})\rightarrow K(\mathsf{LCA})$ upon truncation $\tau_{\leq1}$ and
$\Psi^{-1}$, so we just need to understand $K_{1}(\mathbb{R})\rightarrow
K_{1}(\mathsf{LCA})$. This map already appears in the proof of Theorem
\ref{thm_UnivOfHaarTorsor}, namely in Equation \ref{l_cio4} and our discussion
ibid. showed that it agrees with the absolute value homomorphism. (2) This is
a little trickier, but very beautiful (as we believe). We use the Reciprocity
Law, Theorem \ref{thm_Recip}, and prove this by a global-to-local argument. We
only need the case $\mathfrak{A}:=\mathbb{Z}$ and $A:=\mathbb{Q}$. Then the
Reciprocity Law in degree one tells us that the composition%
\begin{equation}
a:K_{1}(\mathbb{Q})\longrightarrow K_{1}(\mathbb{A}_{fin})\oplus
K_{1}(\mathbb{R})\longrightarrow K_{1}(\mathsf{LCA})\label{lcio5_a}%
\end{equation}
is zero. Clearly multiplication by $p$, $\mathbb{Q}\overset{\cdot p}%
{\underset{\sim}{\longrightarrow}}\mathbb{Q}$, is an isomorphism, so it
defines an element $[p]\in\mathbb{Q}^{\times}\cong K_{1}(\mathbb{Q})$.
Multiplication by $p$ also defines an isomorphism of the ad\`{e}les in
$\mathsf{LCA}$,%
\[
\mathbb{A}\overset{\cdot p}{\mathbb{\underset{\sim}{\longrightarrow}}%
}\mathbb{A}\text{,}\qquad\qquad\left(  \;\text{or }\mathbb{A}_{fin}%
\oplus\mathbb{R}\overset{\cdot p}{\mathbb{\underset{\sim}{\longrightarrow}}%
}\mathbb{A}_{fin}\oplus\mathbb{R}\;\right)
\]
and we can get this isomorphism by applying the functor underlying the map $a$
in Equation \ref{lcio5_a}, namely $a(\mathbb{Q})\overset{a(-\cdot
p)}{\underset{\sim}{\longrightarrow}}a(\mathbb{Q})$. However, since $a=0$ on
the level of $K$-theory by the Reciprocity Law, the class $a\left(
[p]\right)  \in K_{1}(\mathsf{LCA})$ is zero. Now, writing the ad\`{e}les as a
product%
\begin{equation}
\mathbb{A}\simeq\mathbb{Q}_{p}\oplus\left.  \prod\nolimits_{\ell\neq p,\infty
}^{\prime}\right.  (\mathbb{Q}_{\ell},\mathbb{Z}_{\ell})\oplus\mathbb{R}%
\text{,}\label{lcio5}%
\end{equation}
where the restricted product runs over all places $\ell\notin\{p,\infty\}$, we
also have the following: On each $\mathbb{Q}_{\ell}$ with $\ell\neq p$ and
$\ell$ a finite place, we have $p\in\mathbb{Z}_{\ell}^{\times}$, so by the
argument in Example \ref{ex_DetA2} (applied for the prime $\ell$) the
multiplication by $p$ automorphism induces the trivial class in $K_{1}%
(\mathsf{LCA})$. Since this applies to all topological $\ell$-torsion factors
of $\left.  \prod\nolimits_{\ell\neq p,\infty}^{\prime}\right.  (\mathbb{Q}%
_{\ell},\mathbb{Z}_{\ell})$, the induced map of multiplication with $p$ on
$K$-theory on this object is zero. Further, on $\mathbb{R}$ we know by\ part
(1) of this proof that we get $p=\left\vert p\right\vert \in K_{1}%
(\mathsf{LCA})$, and on $\mathbb{Q}_{p}$ we do not yet know what class in
$K_{1}$ we get, but this is what we wanted to compute. We may call it
$\beta\in\mathbb{R}_{>0}^{\times}$. However, since we know that on the entire
ad\`{e}les the $K$-theory class is the neutral element, we obtain by the
additivity along direct sums in Equation \ref{lcio5} the identity%
\[
1=\left\vert p\right\vert \cdot\beta\cdot1\text{,}%
\]
where $\beta$ is the class of multiplication by $p$ on $\mathbb{Q}_{p}$ in
$K_{1}(\mathsf{LCA})$; this corresponds to the multiplicativity of the
modulus, Equation \ref{lcey3b}. We deduce from this equation that
$\beta=p^{-1}=\left\vert p\right\vert _{p}$, i.e. the standard normalization
of the $p$-adic absolute value of $p$. Next, for $x\in\mathbb{Z}_{p}^{\times}$
by the argument in Example \ref{ex_DetA2}, this time applied for the prime $p$
itself, such an $x$ induces also the trivial class in $K$-theory. Since%
\[
\mathbb{Q}_{p}^{\times}\simeq\left\langle p^{\mathbb{Z}}\right\rangle
\oplus\mathbb{Z}_{p}^{\times}\text{,}%
\]
by multiplicativity, these computations describe the entire map%
\[
K_{1}(\mathbb{Q}_{p})\longrightarrow K_{1}(\mathsf{LCA})\text{,}%
\]
and we see that it indeed agrees with the $p$-adic valuation.
\end{proof}

\begin{remark}
The reader will observe that the statement $a=0$ in the previous proof is
precisely the product formula%
\[
\prod_{v}\left\vert x\right\vert _{v}=1\qquad\text{for}\qquad x\in
\mathbb{Q}^{\times}\text{,}%
\]
(where $v$ runs through all places), but here we do \emph{not} get it from
using the standard normalized definitions of absolute values, but rather by
the topological properties of the ad\`{e}les alone. These considerations work
also for ad\`{e}les of arbitrary number fields and can also be exploited in
other degrees of $K$-theory to get the vanishing statements of reciprocity
laws. See \cite{clausen}, \cite{kthyartin}.
\end{remark}

\subsection{Dreams}

\subsubsection{Equivariant measure}

Without doubt, as much as the term \textquotedblleft equivariant Haar
measure\textquotedblright\ and the surrounding picture is suggestive, we have
only constructed the associated torsor, but not an actual formalism of a
measure. For example, our torsor description tells us how to relate the
equivariant measures on objects sitting in exact sequences $A\hookrightarrow
B\twoheadrightarrow C$, where $A$ is a subobject. By using
translation-invariance as a guiding principle, one would like to believe that
this can be extended to all cosets $A+b$ with $b\in B$, too. An approach using
classical $K$-theory is not possible since cosets like $A+b$ for $b\neq0$ are
not subobjects. Perhaps Zakharevich's $K$-theory of assemblers
\cite{MR3558230} could be the basic language for attacking these issues. More
optimistically, there should be a form of a $\sigma$-algebra style calculus,
where not just finite scissor operations are allowed, but also countable ones.
It would be wonderful to have some formalism where computations in the style
of Tate's thesis could exist. All of this, of course, is fairly utopic and
speculative at the moment.

Nonetheless, in the non-equivariant situation, $\mathfrak{A}=\mathbb{Z}$ and
$A_{\mathbb{R}}=\mathbb{R}$, the classical theory of the Haar measure has all
these features, so \textit{one} example exists.\ Clearly, such a conjectural
theory would not be real-valued, and the glueing laws when taking the union or
intersection of measurable subsets will involve homotopical structures,
controlled by the symmetry constraint of the Picard groupoid $V(\mathsf{LCA}%
_{\mathfrak{A}})$ of Definition \ref{def_EquivHaarMeasure}.

\subsubsection{ETNC}

In this section, we rely further on the notation and conventions of
\cite{MR1884523}. Let $K$ be a number field, $M$ a motive in the category of
Chow motives over $K$ with rational coefficients (i.e. a manifestation of a
pure motive using the classical approach of Grothendieck, \cite{MR2115000}).
Suppose $M$ has an action by a finite-dimensional semisimple $\mathbb{Q}%
$-algebra $A$, i.e. a $\mathbb{Q}$-algebra homomorphism%
\[
A\longrightarrow\operatorname*{End}\nolimits_{\mathsf{Motives}}(M)\text{.}%
\]
Let $\mathfrak{A}\subset A$ be an order in $A$ and $T$ a projective
$\mathfrak{A}$-structure in the sense of Burns--Flach, \cite[\S 3.3,
Definition 1]{MR1884523}. Recall the conventions%
\[
\mathfrak{A}_{p}:=\mathfrak{A}\otimes_{\mathbb{Z}}\mathbb{Z}_{p}\text{,}\qquad
A_{p}:=A\otimes_{\mathbb{Q}}\mathbb{Q}_{p}\text{.}%
\]
Assuming the Coherence Hypothesis, \cite[\S 3.3]{MR1884523}, Burns and Flach
define an object%
\[
\Xi(M,T,S)_{\mathbb{Z}}\in\prod_{p}V(\mathfrak{A}_{p})\times_{\prod_{p}%
V(A_{p})}V(A)
\]
for $S$ a suitable choice of places. Here $V(-)$ denotes Deligne's virtual
object functor, see Definition \ref{def_VirtObjects}. They then prove that it
is independent of the choice of $T$ and $S$, so that they obtain a
well-defined object $\Xi(M)_{\mathbb{Z}}\in\mathbb{V}(\mathfrak{A})$. Assuming
\cite[\S 3.2, Conjecture 1]{MR1884523} (the fundamental exact sequence over
the reals), they deduce that after base change to the reals, there is a
canonical trivialization%
\[
\vartheta_{\infty}:\Xi(M)_{\mathbb{Z}}\otimes_{A}A_{\mathbb{R}}\cong%
1_{V(A_{\mathbb{R}})}\text{.}%
\]
Thus, they obtain an object $\Xi(M)_{\mathbb{Z}}\in\mathbb{V}(\mathfrak{A}%
,\mathbb{R})$, and using%
\[
\pi_{0}\mathbb{V}(\mathfrak{A},\mathbb{R})\cong K_{0}(\mathfrak{A},\mathbb{R})
\]
(see \cite[Proposition 2.5]{MR1884523}) this pins down the class%
\[
R\Omega(M,\mathfrak{A})\in K_{0}(\mathfrak{A},\mathbb{R})\text{.}%
\]
It is further shown in \cite[\S 3.4, Lemma 7]{MR1884523} that $R\Omega
(M,\mathfrak{A})$ lies in the locally free class group part
$\operatorname*{Cl}(\mathfrak{A},\mathbb{R})\subseteq K_{0}(\mathfrak{A}%
,\mathbb{R})$, see \cite[\S 2.9]{MR1884523}. Now, if $\mathfrak{A}$ is a
regular order (e.g. if it is hereditary), using our isomorphism%
\[
K_{0}(\mathfrak{A},\mathbb{R})\overset{\sim}{\longrightarrow}K_{1}%
(\mathsf{LCA}_{\mathfrak{A}})
\]
of Theorem \ref{thm_main_asannounced}, we should also have an explicit
description of $R\Omega(M,\mathfrak{A})$ as an element in the right hand side,
and specifically there should exist an expression of this object in terms of
Nenashev's presentation of $K_{1}$.

\begin{question}
How can we identify the subgroup $K_{0}(\mathfrak{A},\mathbb{Q})$ inside
$K_{1}(\mathsf{LCA}_{\mathfrak{A}})$?
\end{question}

\begin{question}
How can we identify the subgroup $\operatorname*{Cl}(\mathfrak{A},\mathbb{R})$
inside $K_{1}(\mathsf{LCA}_{\mathfrak{A}})$?
\end{question}

\begin{question}
How can we identify the extended boundary map $\hat{\delta}_{\mathfrak{A}%
,\mathbb{R}}^{1}$?
\end{question}

The principal conjecture,%
\[
\hat{\delta}_{\mathfrak{A},\mathbb{R}}^{1}(L^{\ast}(\left.  M_{A}\right.
,0))=-R\Omega(M,\mathfrak{A})
\]
should then also translate to an equality in $K_{1}(\mathsf{LCA}%
_{\mathfrak{A}})$ as opposed to $K_{0}(\mathfrak{A},\mathbb{R})$. All of these
aspects remain open for the moment. The isomorphism of Theorem
\ref{thm_main_asannounced} proves however that all these formulations exist if
and only if the original ETNC holds. See \cite[\S 4.3, Conjecture
4]{MR1884523}.

While this may sound like\textit{\ a lot} of open questions, we still hope
this new language for its formulation will turn out to be beneficial.
Certainly, we believe that the computations in the Nenashev presentation in
\S \ref{subsect_Nenashev} have a flavour which is genuinely different from
working in the Swan presentation of $K_{0}(\mathfrak{A},\mathbb{R})$.

\bibliographystyle{amsalpha}
\bibliography{ollinewbib}

\end{document}